\documentclass[a4paper,final,fleqn]{article}
\pdfoutput=1

\usepackage [english] {babel}
\usepackage [utf8]    {inputenc}
\usepackage [T1]      {fontenc}

\usepackage {amsmath,mathtools}

\usepackage {lmodern}
\usepackage {amssymb}
\usepackage [mathscr] {eucal}

\usepackage {bbding}

\usepackage {ifthen,xspace,enumitem}

\usepackage {fixltx2e,titlesec,fancyhdr}

\usepackage [usenames,dvipsnames] {xcolor}

\usepackage [thmmarks,amsmath,amsthm]
			    {ntheorem}

\usepackage[
final,
hyperfootnotes=true,
hyperindex=true,
breaklinks=true,
colorlinks=true,
linkcolor=linkred,
citecolor=gray,
filecolor=magenta,
urlcolor=linkred,
anchorcolor=blue,
menucolor=red,
plainpages=false,
unicode=true,
pdfproducer={pdflatex},
bookmarks=true,
bookmarksnumbered=true,
bookmarksopen=true,
bookmarksopenlevel=2,
pdfview=FitH, 
pdfstartview=FitH, 
pdfstartpage={1},
pdfpagelayout=OneColumn
]{hyperref}

\definecolor {darkblue}  {rgb} {0.09,0.20,0.43}
\definecolor {darkgray}  {rgb} {0.39,0.39,0.40}
\definecolor {darkgreen} {rgb} {0,0.5,0}
\definecolor {purple}    {rgb} {0.55,.16,0.44}
\definecolor {linkred}   {rgb} {0.545,0.275,0.29}
\definecolor {accent}{rgb}{0.25,0.25,0.25}

\titleformat {\section} {\color{accent}\normalfont\bfseries\large}{\thesection}{1em}{}
\titleformat {\subsection} {\color{accent}\normalfont\bfseries}{\thesubsection}{1em}{}

\let\oldtitle\title
\def\title #1 {\oldtitle{\bfseries\color{accent}#1}}

\newenvironment{equivenum}%
	{\begin{enumerate}[label=(\roman{*})]}{\end{enumerate}}
	{\begin{enumerate}[label={\rm (\roman{*})}]}{\end{enumerate}}

\newcommand{\accentheaderfont}%
	{\normalfont\normalsize\bfseries\color{accent}}
\newcommand{\prooffont}%
	{\normalfont\normalsize\itshape}

\newtheoremstyle{accentheader}%
	{\item[\hskip\labelsep \accentheaderfont{##1 ##2}]}%
	{\item[\hskip\labelsep \accentheaderfont{##1 ##2 (##3)}]}%

\newtheoremstyle{nonumberaccentheader}%
	{\item[\hskip\labelsep \accentheaderfont{##1}]}%
	{\item[\hskip\labelsep \accentheaderfont{##1 (##2)}]}%

\newtheoremstyle{blackheader}
 	{\item[\hskip\labelsep \prooffont{##1}]}%
 	{\item[\hskip\labelsep \prooffont{##3}]}
	
\newtheoremstyle{blackheadernumbered}
 	{\item[\hskip\labelsep \prooffont{##1 ##2.}]}%
 	{\item[\hskip\labelsep \prooffont{##1 ##2 (##3).}]}

\newtheoremstyle{nonumberblackheader}%
	{\item[\hskip\labelsep \prooffont{##1}]}%
	{\item[\hskip\labelsep \prooffont{##1 (##2)}]}%

\numberwithin{dummy}{section}

\theoremstyle{accentheader}
\theorembodyfont{\itshape}
\theoremseparator{}
\theoremsymbol{}

\newtheorem{thm}{Theorem}[section]

\newtheorem{prop}[thm]{Proposition}

\newtheorem{lem}[thm]{Lemma}

\newtheorem{cor}[thm]{Corollary}

\theoremstyle{blackheader}
\theorembodyfont{\normalfont}

\renewtheorem{proof}{Proof.}

\theoremstyle{blackheadernumbered}
\theorembodyfont{\normalfont}
\newtheorem{rem}[thm]{Remark}

\theoremstyle{accentheader}
\theorembodyfont{\itshape}

\renewtheorem{prop*}{Proposition}
\renewtheorem{lem*}{Lemma}
\renewtheorem{defn*}{Definition}
\renewtheorem{cor*}{Corollary}

\theoremstyle{blackheader}
\theorembodyfont{\normalfont}

\renewtheorem{rem*}{Remark.}

\let\existsorig\exists 
\renewcommand{\exists}{\ \existsorig\ } 
\let\forallorig\forall 
\renewcommand{\forall}{\ \forallorig\ } 

\newcommand{\overbar}[1]{\mkern 1.5mu\overline{\mkern-1.5mu#1\mkern-1.5mu}\mkern 1.5mu}
\newcommand{\ob}[1]{\overbar{#1}}

\renewcommand{\atop}[2]{\genfrac{}{}{0pt}{}{#1}{#2}}

 \newcommand{\defl}{\coloneqq}%
 \newcommand{\defr}{\eqqcolon}%
 \newcommand{\nequiv}{\not\equiv}

\renewcommand{\le}{\leqslant}
\renewcommand{\ge}{\geqslant}

\newcommand{\lex}{\preccurlyeq}

\renewcommand{\iff}{\Leftrightarrow}

\renewcommand{\qed}{{\quad~\color{accent}\text{\small\RectangleBold}}}
\newcommand{\fish}%
	{\quad\ensuremath{\color{accent}\rtimes}}
\newcommand{\map}%
	{\quad\ensuremath{\multimap}}

\DeclarePairedDelimiter{\abs}{\lvert}{\rvert}

\DeclarePairedDelimiter{\n}{\lVert}{\rVert}
\DeclarePairedDelimiter{\lin}{\langle}{\rangle}
\DeclarePairedDelimiter{\pbr}{\lbrace}{\rbrace}
\DeclarePairedDelimiterX{\setdef}[2]{\{}{\}}{#1\,:\,#2}
\DeclarePairedDelimiter{\setd}{\{}{\}}

\DeclarePairedDelimiter{\p}{(}{)}

\newcommand{\e}{\mathrm{e}}
\newcommand{\ii}{\mathrm{i}}

\newcommand{\opento}{\hookrightarrow}%

\DeclareMathOperator{\dist}{dist}

\newcommand{\Id}{I}

\renewcommand{\C}{\mathbb{C}}%
\newcommand{\N}{\mathbb{N}}%
\newcommand{\R}{\mathbb{R}}%
\newcommand{\Z}{\mathbb{Z}}%
\newcommand{\T}{\mathbb{T}}%

\newcommand{\Nc}{\mathcal{N}}
\newcommand{\Oc}{\mathcal{O}}
\newcommand{\Pc}{\mathcal{P}}

\newcommand{\Sc}{\mathcal{S}}

\newcommand{\Es}{\mathscr{E}}
\newcommand{\Fs}{\mathscr{F}}

\newcommand{\Hs}{\mathscr{H}}

\newcommand{\Ms}{\mathscr{M}}

\newcommand{\Ss}{\mathscr{S}}

\newcommand{\Vs}{\mathscr{V}}
\newcommand{\Ws}{\mathscr{W}}

\newcommand{\al}{\alpha}
\newcommand{\bt}{\beta}
\newcommand{\gm}{\gamma}	 \newcommand{\Gm}{\Gamma}
\newcommand{\dl}{\delta}     \newcommand{\Dl}{\Delta}
\newcommand{\ep}{\varepsilon} 

\newcommand{\lm}{\lambda}    \newcommand{\Lm}{\Lambda}

\newcommand{\sg}{\sigma}     
\newcommand{\vs}{\varsigma}
\newcommand{\om}{\omega}     \newcommand{\Om}{\Omega}
\newcommand{\ph}{\varphi}    \newcommand{\Ph}{\Phi}

\newcommand{\ps}{\psi}       

\renewcommand{\th}{\theta}

\newcommand{\zt}{\zeta}

\newcommand{\m}{\mathfrak{m}}

\newcommand{\upd}{\mathinner{\mathrm{d}\kern.04em\!}}

\newcommand\ddd{\mathrm{d}\mkern1mu}

\newcommand{\dlm}{\upd{\lm}}

\newcommand{\dr}{\upd{r}}
\newcommand{\ds}{\upd{s}}
\newcommand{\dt}{\upd{t}}

\newcommand{\dx}{\upd{x}}

\newcommand{\dz}{\upd{z}}

\def\XXint#1#2#3{{\setbox0=\hbox{$#1{#2#3}{\int}$ }
\vcenter{\hbox{$#2#3$ }}\kern-.6\wd0}}

\usepackage[normalem]{ulem}

\usepackage[numbers]{natbib}
\renewcommand{\thesection}{\arabic{section}}

\def\@biblabel#1{#1.}

\DeclareMathAlphabet{\mathpzc}{OT1}{pzc}{m}{it}

\usepackage{showkeys}

\renewcommand{\o}{\mathrm{o}}

\newcommand{\ld}{\text{\tiny\ensuremath{\bullet}}}
\newcommand{\dDl}{\Dl^{\ld}}

\newcommand{\trho}{\tilde\rho}
\newcommand{\tsg}{\tilde\sg}

\renewcommand{\Id}{\mathrm{Id}}

\DeclareMathOperator{\dir}{dir}

\DeclareMathOperator{\diag}{diag}
\newcommand{\Hm}{H}
\renewcommand{\P}{\mathbb{P}}

\newcommand{\Wp}{\Ws}
\newcommand{\K}{\mathbb{K}}
\renewcommand{\vs}{w}

\makeatletter
\newcommand\newsubcommand[3]{\def#1{#2\sc@sub{#3}}}
\def\sc@sub#1{\def\sc@thesub{#1}\@ifnextchar_{\sc@mergesubs}{_{\sc@thesub}}}
\def\sc@mergesubs_#1{_{\sc@thesub#1}}
\makeatother

\newcommand{\FL}{\Fs\ell}

\title{On the extension of the frequency maps of the KdV and the KdV2 equations}
\author{Thomas Kappeler\footnote{Partially supported by the Swiss National Science Foundation
}, Jan-Cornelius Molnar\footnote{Partially supported by the Swiss National Science Foundation
}}
\date{\today}

\begin{document}

\maketitle

\begin{abstract}
In form of a case study for the KdV and the KdV2 equations, we present a novel approach of representing the frequencies of integrable PDEs which allows to extend them analytically to spaces of low regularity and to study their asymptotics. Applications include convexity properties of the Hamiltonians and wellposedness results in spaces of low regularity. In particular, it is proved that on $\Hs^{s}$ the KdV2 equation is $C^{0}$-wellposed if $s\ge 0$ and illposed (in a strong sense) if $s < 0$.


\end{abstract}


\section{Introduction}

The goal of this paper is to discuss, in form of a case study for the KdV and the KdV2 equations on $\T = \R/\Z$, novel formulas for the frequencies of integrable PDEs, allowing to extend the frequencies analytically to spaces of functions of low regularity or distributions and to study their asymptotics. These results are used to derive properties of the frequency map relevant for perturbation theory, the Hamiltonian, and the solution map of such equations. First we state our results on the solutions maps for the KdV equation
\begin{equation}
  \label{kdv}
  \partial_{t}u = -\partial_{x}^{3}u + 6u\partial_{x}u,
\end{equation}
and the second equation in the KdV hierarchy (KdV2)
\begin{equation}
  \label{kdv2}
    \partial_t u = \partial_x^5 u - 10 u \partial_{x}^{3}u - 20 \partial_{x}u\partial_{x}^{2}u + 30 u^{2}\partial_{x}u,
\end{equation}
and outline the derivation of the novel formulas for the frequencies of~\eqref{kdv} and~\eqref{kdv2} at the end of the introduction. To state our results, we need to introduce some more notation. By $\Hs^{s} \equiv \Hs^{s}(\T,\R)$, $s\ge -1$, we denote the standard Sobolev spaces, endowed with the Gardner bracket
\begin{equation}
  \label{pbr}
  \pbr{F, G} \defl \int_0^1 \partial_u F \, \partial_x \partial_u G \, \dx.
\end{equation}
Here, $\partial_{u}F$ and $\partial_{u}G$ are the $L^{2}$-gradients of functionals $F$ and $G$ on $\Hs^{s}$ assumed to be sufficiently regular, so that the integral~\eqref{pbr} is well defined.
Then equations~\eqref{kdv} and~\eqref{kdv2} take the form
\[
  \partial_{t}u = \partial_{x}\partial_{u}\Hm_{1},\qquad
  \partial_{t}u = \partial_{x}\partial_{u}\Hm_{2},
\]
where $\Hm_{1}$ and $\Hm_{2}$ denote the KdV and, respectively, KdV2 Hamiltonian
\begin{align*}
  \Hm_{1}(u) &= \frac{1}{2}\int_{0}^{1} (u_{x}^{2} + 2 u^{3})\,\dx,\\
  \Hm_{2}(u) &= \frac{1}{2}\int_{0}^{1} (u_{xx}^{2} + 10 u u_{x}^{2} + 5 u^{4})\,\dx.
\end{align*}
Note that $[u] \defl \int_{0}^{1} u(x)\,\dx$ is a Casimir of the bracket~\eqref{pbr} and that the level sets
\[
  \Hs_{c}^{s} \defl \setdef{u\in \Hs^{s}}{[u] = c},\qquad c\in \R,
\]
are symplectic leaves. We concentrate on the leaf $\Hs_{0}^{s}$ only, since our results can be easily extended to any other leaf. Furthermore, for any $s\in\R$ and $1 \le p < \infty$, we introduce the sequence space
\[
  \ell^{s,p}_{0,\C} \equiv \ell_{0}^{s,p}(\Z,\C) \defl \setdef{z = (z_{n})_{n\in\Z}\subset\C}{z_{0} = 0,\quad \n{z}_{s,p} < \infty},
\]
where
\[
  \n{z}_{s,p} \defl \p*{ \sum_{n\in\Z} \lin{n}^{sp}\abs{z_{n}}^{p} }^{1/p},\qquad
  \lin{n} \defl 1+\abs{n}.
\]
In addition, we denote by $\ell_{0}^{s,p}$ the real subspace $\setdef{(z_{n})\in\ell_{0,\C}^{s,p}}{z_{-n} = \ob{z_{n}}}$ of $\ell_{0,\C}^{s,p}$. The spaces $\ell_{\C}^{s,p} \equiv \ell^{s,p}(\N,\C)$ and $\ell^{s,p} = \ell^{s,p}(\N,\R)$ are defined in an analogous way. To further simplify notation, we also define
\[
  h_{0}^{s} \defl \ell_{0}^{s,2},\qquad
  h_{0,\C}^{s} \defl \ell_{0,\C}^{s,2}.
\]
Note that the Sobolev spaces $\Hs_{0}^{s}$ and more generally the Fourier Lebesgue spaces $\FL_{0}^{s,p}$ can then be described by
\begin{equation}
  \label{Hs-FLsp-def}
  \Hs_{0}^{s} = \setdef{u\in S'_{\C}}{(u_{n})_{n\in\Z} \in h_{0}^{s}},\quad
  \FL_{0}^{s,p} = \setdef{u\in S'_{\C}}{(u_{n})_{n\in\Z} \in \ell_{0}^{s,p}}
\end{equation}
with $u_{n} = \lin{u,\e^{\ii 2n\pi x}}$, $n\in\Z$. Here, $\lin{\cdot,\cdot}$ denotes the $L^{2}$-inner product on $L^{2}(\T,\C)$, $\lin{f,g} = \int_{\T} f(x)\ob{g(x)}\,\dx$, extended by duality to a pairing of the Schwartz space $\Sc_{\C}$ of 1-periodic functions $f\in C^{\infty}(\R,\C)$ and its dual $\Sc_{\C}'$.
The following result says that the KdV equation and its hierarchy are integrable PDEs in the strongest possible sense.

\begin{thm}[\cite{Kappeler:2003up,Kappeler:2005fb,Kappeler:2008fl}]
\label{bhf}
There exists a complex neighborhood $\Ws\subset\Hs_{0,\C}^{-1}$ of $\Hs_{0}^{-1}$ and an analytic map
\[
  \Ph\colon \Ws\to h_{0,\C}^{-1+1/2},\qquad u\mapsto (z_{n}(u))_{n\in\Z},
\]
with $\Ph(0) = 0$ so that the following holds:

\begin{equivenum}
\item
For any $s\ge -1$, the restriction $\Ph\big|_{\Hs_{0}^{s}}$ is a real analytic diffeomorphism $\Hs_{0}^{s}\to h_{0}^{s+1/2}$.

\item $\Ph$ is canonical in the sense that $\pbr{z_{n},z_{-n}} = \ii$ for any $n\ge 1$, whereas all other brackets between coordinate functions vanish.

\item $\Hm_{1}\circ\Ph^{-1}$, defined on $h_{0}^{1+1/2}$, and $\Hm_{2}\circ\Ph^{-1}$, defined on $h_{0}^{2+1/2}$, are real analytic functions of the actions $I_{n} \defl z_{n}z_{-n}$, $n\ge 1$, alone.
Corresponding results hold for any of the Hamiltonians in the KdV hierarchy.

\item The differential $\ddd_{0}\Ph$ of $\Ph$ at $u=0$ is the weighted Fourier transform,
\[
  \ddd_{0}\Ph\colon\Hs_{0}^{s}\to h_{0}^{s+1/2},\qquad
  u\mapsto \p*{ \p*{2\abs{n}\pi}^{-1/2} u_{n} }_{n\neq 0}.~\fish
\]
\end{equivenum}
\end{thm}

The equations of motion of~\eqref{kdv} and~\eqref{kdv2}, expressed in Birkhoff coordinates are
\[
  \partial_{t} z_{n} = -\ii \om_{n}^{(j)}z_{n},\qquad
  \partial_{t} z_{-n} = \ii \om_{n}^{(j)}z_{-n},\qquad \forall n\ge 1,
\]
where $\om_{n}^{(j)}$, $n\ge 1$, are the frequencies corresponding to $\Hm_{j}$, $j=1,2$,
\[
  \om_{n}^{(j)} = \partial_{I_{n}}\Hm_{j}.
\]
Let $\ell_{+}^{s,p}(\N)$ denote the positive quadrant of $\ell^{s,p}(\N)$ given by
\begin{equation}
  \label{ellp-pos-quad}
  \ell_{+}^{s,p}(\N) = \setdef{I=(I_{n})_{n\ge 1}\in \ell^{s,p}(\N)}{I_{n}\ge 0}.
\end{equation}
The KdV frequencies $(\om_{n}^{(1)})_{n\in\Z}$ are defined on $\ell_{+}^{3,1}(\N)$, whereas the KdV2 frequencies $(\om_{n}^{(2)})_{n\in\Z}$ are defined on $\ell_{+}^{5,1}(\N)$.  They are real analytic and admit the following expansions at $I=0$, reviewed in Appendix~\eqref{app:bnf-kdv2},
\begin{align}
  \label{exp-om-kdv1}
  \om_{n}^{(1)} &= (2n\pi)^{3} - 6I_{n} + \dotsb,\\
  \label{exp-om-kdv2}
  \om_{n}^{(2)} &= (2n\pi)^{5} + 20(2n\pi)\Hm_{0} - 20(2n\pi)^{2}I_{n} + \dotsb,
\end{align}
where the dots stand for higher order terms in $I$ and
\[
  \Hm_{0} = \frac{1}{2}\int_{0}^{1} u^{2}\,\dx = \sum_{n\ge 1} (2n\pi) I_{n}.
\]
In order to state our results on the analytic extensions of $\om_{n}^{(j)}$, $n\ge 1$, we need to normalize the frequencies as follows
\begin{align}
  \label{exp-om-star-kdv1}
  \om_{n}^{(1)\star} &= \om_{n}^{(1)} - (2n\pi)^{3} = - 6I_{n} + \dotsb,\\
  \label{exp-om-star-kdv2}
  \om_{n}^{(2)\star} &= \om_{n}^{(2)} - (2n\pi)^{5} - 20(2n\pi)\Hm_{0} = - 20(2n\pi)^{2}I_{n} + \dotsb.
\end{align}
Note that $\om_{n}^{(j)\star} = \partial_{I_{n}} \Hm_{j}^{\star}$, with $\Hm_{j}^{\star}$, $j=1,2$, denoting the renormalized Hamiltonians given by
\begin{equation}
  \label{Hm-renormalized}
  \Hm_{1}^{\star} = \Hm_{1} - \sum_{n\ge 1} (2n\pi)^{3}I_{n},\qquad
  \Hm_{2}^{\star} = \Hm_{2} - \sum_{n\ge 1} (2n\pi)^{5}I_{n} - 10\Hm_{0}^{2}.
\end{equation}
The results for the frequency maps $\om^{(1)\star} = (\om_{n}^{(1)\star})_{n\ge 1}$ and $\om^{(2)\star} = (\om_{n}^{(2)\star})_{n\ge 1}$ are now stated separately: It has been shown in \cite{Kappeler:2005fb} that each $\om_{n}^{(1)\star}$, $n\in\Z$, extends to a real analytic map on a complex neighborhood $\Vs$ of $\ell^{-1,1}$.

Our novel approach allows to obtain sharp asymptotics of $\om_{n}^{(1)\star}$ as $n\to \infty$ and at the same time yields a direct proof of their analytic extensions.

\begin{thm}[\boldmath Extension \& asymptotics of $\om^{(1)\star}$]
\label{thm:kdv-freq}
\begin{equivenum}
\item
The map $\om^{(1)\star}$ extends to a map on $\ell_{+}^{-1,1}$ with values in $\bigcap_{r > 1}\ell^{-1,r}$ and its restrictions
\[
  \om^{(1)\star}\colon \ell_{+}^{2s+1,1}\to
  \begin{cases}
  \ell^{-1,r}, & s = -1,\phantom{/2} \quad r > 1,\\
  \ell^{2s+1,1}, & -1 < s < -1/2,\\
  \ell^{r}, & s\ge -1/2, \quad r > 1,
  \end{cases}
\]
are real analytic.

\item
For any $-1 < s < -1/2$ and any $I\in\ell_{+}^{2s+1,1}$, the linear operator  $\ddd_{I}\om^{(1)\star} + 6\Id\colon \ell^{2s+1,1}\to \ell^{2s+1,1}$ is compact.

\item On $\ell_{+}^{2s+1,1}$, the frequencies $\om_{n}^{(1)\star}$, $n\ge 1$, have the following asymptotics
\[
  \om_{n}^{(1)\star} + 6I_{n} =
  \begin{cases}
  o(n^{3\abs{s}-2}), & -1 \le s < -1/3,\\
  O(n^{-1}),         & \phantom{ -1 \le }\; s\ge -1/3,
  \end{cases}
\]
which hold locally uniformly in a complex neighborhood of $\ell_{+}^{2s+1,1}$.

\item
Furthermore, the restriction of $\om^{(1)\star}$ to $\ell_{+}^{2}$ takes values in $\ell^{2}$, the map $\om^{(1)\star}\colon \ell_{+}^{2}\to \ell^{2}$ is real analytic, and $\ddd_{I}\om^{(1)\star} + 6\Id\colon \ell^{2}\to \ell^{2}$ is compact for any $I\in\ell_{+}^{2}$.~\fish
\end{equivenum}
\end{thm}

An extended version of Theorem~\ref{thm:kdv-freq} can be found in Section~\ref{s:kdv1} Theorem~\ref{kdv-freq-as}. Theorem~\ref{thm:kdv-freq} has several applications.
One application concerns convexity properties of the KdV Hamiltonian. Recall that in \cite[Theorem~1]{Kappeler:CNzeErmy} we proved a conjecture of \citet{Korotyaev:2011tw} saying that the Hamiltonian $\Hm_{1}^{\star}$ admits a real analytic extension to $\ell_{+}^{2}$ and that $\ddd_{I}^{2}\Hm_{1}^{\star}\big|_{I=0} = -6\Id$. It implies that $\Hm_{1}^{\star}$ is strictly concave near $I=0$. Since $\Hm_{1}^{\star}$ is known to be concave on the positive quadrant $\ell_{+}^{2}$ (cf \cite{Korotyaev:2011tw}), the question arose whether $\Hm_{1}^{\star}$ is strictly concave on all of $\ell_{+}^{2}$. Theorem~\ref{thm:kdv-freq} implies that by and large, this indeed holds.

\begin{thm}
The renormalized KdV Hamiltonian $\Hm_{1}^{\star}\colon \ell_{+}^{2}\to \R$ is strictly concave on an open and dense subset $\Oc$ of $\ell_{+}^{2}$ containing $I=0$. It means that for any $I\in \Oc$,
\[
  \ddd_{I}^{2}\Hm_{1}^{\star}(J,J) \le -c\n{J}_{\ell^{2}}^{2},\qquad \forall J\in \ell^{2},
\]
where the constant $c>0$ can be chosen locally uniformly in $I$.~\fish
\end{thm}

The compactness of $\ddd\om^{(1)\star}+6\Id$ together with the analyticity of the frequencies $\om_{n}^{(1)\star}$ and their asymptotics, obtained in Theorem~\ref{thm:kdv-freq}, lead to the following result on the frequency map.

\begin{cor}
\label{cor:kdv-freq-localdiffeo}
For any $-1 < s < -1/2$, the map $\om^{(1)\star}\colon \ell_{+}^{2s+1,1}\to \ell^{2s+1,1}$ is a local diffeomorphism on an open and dense subset of $\ell_{+}^{2s+1,1}$ containing $I=0$.~\fish
\end{cor}

The asymptotics of $\om_{n}^{(1)\star}$ of Theorem~\ref{thm:kdv-freq} (iii) lead to the following

\begin{cor}
\label{cor:kdv-L2-freq-O}
\begin{equivenum}
\item
On $\Hs_{0}^{s}$, $-1\le s\le 0$, one has
\begin{equation}
  \label{omn-1-as}
  \om_{n}^{(1)}
  =
  8n^{3}\pi^{3} +
  \begin{cases}
  o(n^{-2s-1}), & -1 \le s < 0,\\
  O(n^{-1}), & s = 0,
  \end{cases}
\end{equation}
where the implicit constant can be chosen locally uniformly in $q\in\Hs_{0}^{s}$ and uniformly in $n\ge 1$.

\item
The estimate is sharp for $-1 < s < 0$ in the sense that $\om_{n}^{(1)} = 8n^{3}\pi^{3} + O(n^{\al})$ does not hold for any $\al < -2s-1$.
In particular, $\om_{n}^{(1)} = 8n^{3}\pi^{3} + O(1)$ does \emph{not} hold for $-1 < s < -1/2$.
\item
For $s=0$, the statement~\eqref{omn-1-as} holds uniformly on bounded subsets of $L_{0}^{2}$.~\fish
\end{equivenum}
\end{cor}

\begin{rem}
The estimate $\om_{n}^{(1)} = 8n^{3}\pi^{3} + O(n^{-1})$ (uniformly on bounded subsets of $L_{0}^{2}$) improves on~\cite[Proposition~8.1]{Kappeler:2013bt} where by other techniques, the estimate was proved for $q$ in $\Hs_{0}^{1}$.\map
\end{rem}

Corollary~\ref{cor:kdv-L2-freq-O} leads to an improvement of the one-smoothing property of solutions of the KdV equation on the circle of~\cite{Erdogan:2013ve} and~\cite{Kappeler:2013bt}.
We again state our result only for the zero leaf $\Hs_{0}^{s}$.
Denote by $L$ the operator $L = -\partial_{x}^{3}$ of the Airy equation $\partial_{t}v = Lv$.

\begin{cor}
\label{cor:Airy-Kdv}
For any initial datum $q\in L_{0}^{2}$, denote by $u(t)$ the unique solution of the KdV equation with $u(0) = q$ and by $\e^{tL}q$ the solution of $\partial_{t}v = Lv$ with $v(0)=q$. Then
\[
  \n{u(t)-\e^{t L}q}_{\Hs^{1}} \le C(1+\abs{t}),\qquad \forall t\in\R.
\]
The constant $C >0$ can be chosen locally uniformly for $q\in L_{0}^{2}$.\fish
\end{cor}

\begin{rem}
For initial data $q\in L_{0}^{2}$, the results in~\cite{Erdogan:2013ve} imply that for any $s < 1$,
 $\n{u(t)-\e^{tL}q}_{\Hs^{s}}$ is bounded by $C(s,\n{q}_{L^{2}})(1+\abs{t})$. In~\cite[Theorem~8.2]{Kappeler:2013bt} it is shown that for any $q\in \Hs_{0}^{N}$ with $N\in\Z_{\ge 1}$, $\n{u(t)-\e^{tL}q}_{\Hs^{1}} \le C(1+\abs{t})$ -- see~\cite[Appendix~B]{Kappeler:2013bt} for a detailed discussion.\map
\end{rem}

Finally, Theorem~\ref{thm:kdv-freq} can be used to answer the question left open for quite some time whether the KdV equation is $C^{k}$ wellposed for $k=1,2$ in $\Hs_{0}^{s}$ for any $-1 < s < -1/2$. We show in Theorem~\ref{thm:kdv-wp} that the answer is negative.

Let us now turn to the frequencies of the KdV2 equation. Recall from \cite{Battig:1997ek} that $\om^{(2)\star}$ admits a real analytic extension to a map $\ell_{+}^{3,1}\to \ell^{\infty}$. In fact, it is shown in \cite{Battig:1997ek} that $\om^{(2)\star}$ takes values in $c_{0}$. We improve this result as follows.

\begin{thm}[\boldmath Extension \& asymptotics of $\om^{(2)\star}$]
\label{thm:kdv2-freq}
The map $\om^{(2)\star}$ can be extended as a real analytic
\[
  \om^{(2)\star}\colon \ell_{+}^{2s+1,1}\to
  \begin{cases}
   \ell^{2s-1,1}, & -1 < s < 1/2,\\
  \ell^{r}, & s \ge 1/2,\quad r > 1,
  \end{cases}
\]
with asymptotics
\[
  \om_{n}^{(2)\star} + 20(2n\pi)^{2}I_{n}
  =
  \begin{cases}
  n^{-3s}\ell_{n}^{1}, & -1< s < 0,\\
  \ell_{n}^{1+}, & s \ge 0,
  \end{cases}
\]
which hold locally uniformly on a complex neighborhood of $\ell_{+}^{2s+1,1}$.
Here $\ell_{n}^{p}$, $1\le p \le \infty$, denotes a sequence of complex numbers which is $\ell^{p}$-summable and $\ell_{n}^{1+}$ one which is $\ell^{r}$ summable for any $r > 1$.~\fish
\end{thm}

\begin{rem}
Additional results on the extension of $\om^{(2)\star}$ to Fourier Lebesgue spaces can be found in Section~\ref{ss:KdV2-freq}.\map
\end{rem}

Theorem~\ref{thm:kdv2-freq} leads to the following result on the frequency map, useful to analyze perturbations of the KdV2 equation.

\begin{cor}
\label{cor:kdv2-freq-localdiffeo}
For any $-1< s < 1/2$, the map $\om^{(2)\star}\colon \ell_{+}^{2s+1,1}\to \ell^{2s-1,1}$ is a local diffeomorphism on an open and dense subset of $\ell_{+}^{2s+1,1}$ containing $I=0$.~\fish
\end{cor}

Theorem~\ref{thm:kdv2-freq} also applies to study the solution map of the KdV2 equation. First we need to introduce some more notation.
According to~\cite{Battig:1997ek}, for any initial datum $q\in\Hs^{m}$ with $m\ge1$ integer, there exists a unique, global in time solution $v(t,x) = v(t,x,q)$ of~\eqref{kdv2}, $v\in C(\R,\Hs^{m})$. In particular, for any time $t\in\R$ and $T > 0$, the nonlinear evolution operator
\[
  \Sc_{t}^{(2)} = \Sc^{(2)}(t,\cdot)\colon \Hs^{m}\to \Hs^{m}
\]
and the uniquely defined solution map
\[
  \Sc^{(2)}\colon \Hs^{m}\to C([-T,T],\Hs^{m}),\qquad q\mapsto v(\cdot,\cdot,q),
\]
are well defined and continuous. In the following, let $\Hs$ denote any invariant subspace of $\Hs^{s}$ with $s$ real and $a < 0 < b$. A continuous curve $\gm\colon (a,b)\to \Hs$, $\gm(0) = q$, is called a \emph{solution} of the KdV2 equation in $\Hs$ with initial datum $q$ if and only if for any sequence of $C^{\infty}$-potentials $(q_{k})_{k\ge 1}$ converging to $q$ in $\Hs$, the corresponding sequence $(\Sc^{(2)}(t,q_{k}))_{k\ge 1}$ of solutions of~\eqref{kdv2} with initial data $q_{k}$ converges to $\gm(t)$ in $\Hs$ for any $t\in (a,b)$. KdV2 is said to be \emph{globally $C^{0}$-wellposed in $\Hs$} if (a) for any initial datum $q\in\Hs$ the initial value problem~\eqref{kdv2} admits a solution $\Sc^{(2)}(\cdot,q)$ in the aforementioned sense which is globally defined in time and (b) the solution map $\Sc^{(2)}\colon \Hs\to C([-T,T],\Hs)$ is continuous for every $T > 0$.
KdV2 is said to be \emph{(uniformly/$C^{k}$/$C^{\om}$) wellposed} if the solution map $\Sc^{(2)}\colon \Hs\to C([-T,T],\Hs)$ is (uniformly continuous/$C^{k}$/$C^{\om}$) for every $T > 0$.
Furthermore, a map $f\colon X\to Y$ between Banach spaces $X$ and $Y$ is said to be \emph{nowhere locally uniformly continous} if for any nonempty open subset $U\subset X$, the map $f$ is not uniformly continuous on $U$.
Finally, for any $d> 0$ and any $s\ge 0$, introduce the level sets of $\Hm_{0}$
\[
  \Ms_{0,d}^{s} \equiv \Ms_{0,d}^{s}(\T,\R)
  = \setdef*{u\in \Hs_{0}^{s}}{\Hm_{0}(u) = d}.
\]

\begin{thm}[Wellposedness for KdV2]
\label{thm:kdv2-wp}
\begin{equivenum}
\item
The KdV2 equation is globally $C^{0}$-wellposed in $\Hs^{s}$ for any $s\ge 0$. In particular, for any $T>0$, the solution map $\Sc^{(2)}\colon \Hs^{s}\to C([-T,T],\Hs^{s})$ is continuous and has the group property $\Sc^{(2)}(t+s,q) = \Sc^{(2)}(t,\Sc^{(2)}(s,q))$ for all $t,s\in\R$ and $q\in\Hs^{s}$. As a consequence, for any $t\in\R$, the flow map $\Sc_{t}^{(2)}\colon\Hs^{s}\to \Hs^{s}$, $q\mapsto \Sc^{(2)}(t,q)$ is a homeomorphism.

\item
For any $s\ge 1/2$ and $d>0$, the KdV2 equation is globally $C^{\om}$-wellposed and uniformly $C^{0}$-wellposed in $\Ms_{0,d}^{s}$.

\item
In contrast, for any $s\ge 1/2$ and $t > 0$, the solution map $\Sc^{t}\colon \Hs_{0}^{s}\to \Hs_{0}^{s}$ is nowhere locally uniformly continuous.
In particular, the KdV2 equation is not $C^{1}$-wellposed and not uniformly $C^{0}$-wellposed in $\Hs_{0}^{s}$ for $s\ge 1/2$.

\item
For any $0 \le s < 1/2$, $d>0$, and $t > 0$, the solution map $\Sc^{t}\colon \Ms_{0,d}^{s}\to \Ms_{0,d}^{s}$ is nowhere locally uniformly continuous.
In particular, the KdV2 equation is not $C^{1}$-wellposed and not uniformly $C^{0}$-wellposed in $\Ms_{0,d}^{s}$ for $0 \le s < 1/2$ and $d > 0$.

\item
The solution map cannot be continuously extended to any initial datum in $\Hs_{0}^{s}\setminus L_{0}^{2}$ for $-1< s < 0$. More precisely, the frequencies are given by the formula
\[
  \om_{n}^{(2)} = (2n\pi)^{5} + 20(2n\pi)\Hm_{0} + \om_{n}^{(2)\star},
\]
where $\om_{n}^{(2)\star}$ extends analytically to $\Hs_{0}^{s}$ with $s> -1$. Hence for any $n\ge 1$, $\om_{n}^{(2)}$ becomes infinite on $\Hs_{0}^{s}\setminus L_{0}^{2}$ for any $-1< s < 0$.~\fish
\end{equivenum}

\end{thm}

\begin{rem}
The KdV2 equation and generalizations of it appear in the analysis of long-wave approximations to the water wave equation - cf. for instance~\cite{Craig:2005fv} as well as the references therein.
Wellposedness results for such equations in the periodic setup are discussed in~\cite{Bourgain:1995wk}, but the case of the KdV2 equation is not explicitly treated there. Earlier results were obtained in~\cite{Saut:1979bl}. To the best of our knowledge, the results in\cite{Battig:1997ek} are the best available so far.
In contrast, the wellposedness of this type of equations on the line have been studied extensively. Recently various new results have been obtained -- see \cite{Grunrock:2010bg,Guo:2012th,Kato:2012ih,Kenig:2015dx} and references therein.
In particular, in \cite{Guo:2012th,Kato:2012ih}, wellposedness results were established for initial data in $\Hs^{s}(\R)$ with $s\ge 2$ whereas in \cite{Grunrock:2010bg,Kato:2012ih}, such results were obtained for initial data in certain classes of Fourier Lebesgue spaces.
Since it is believed that for such equations  stronger wellposedness results can be obtained on the line than in the periodic setup (cf. \cite{Bourgain:1995wk}), it can be expected that results analogous to the ones of Theorem~\ref{thm:kdv2-wp} hold for the KdV2 equation on the line.~\map
\end{rem}

Finally, we prove that the renormalized KdV2 Hamiltonian $\Hm_{2}^{\star}$ extends real analytically to $h_{+}^{1} = h_{+}^{1}$, and discuss its convexity properties which are similar to those of $\Hm_{1}^{\star}$ -- see Section~\ref{ss:kdv2-hamiltonian}.

\paragraph{Formulas for the frequencies.}
As mentioned at the beginning of the introduction, the proofs of Theorem~\ref{thm:kdv-freq} and Theorem~\ref{thm:kdv2-freq} are based on new formulas for the frequencies of the KdV and the KdV2 equations. Let us outline how to derive them in the case of the KdV equation:
Consider for any $n\ge 1$ the $n$th frequency $\om_{n}^{(1)}$.
By a density argument, it suffices to consider real potentials with $I_{n} > 0$.
In a real neighborhood of such a potential, the Birkhoff coordinate $z_{n}$ can be expressed in terms of action angle variables $z_{n} = \sqrt[+]{I_{n}}\e^{-\ii \th_{n}}$ and Hamilton's equations of motion take the form
\[
  \partial_{t}\th_{n} = \om_{n}^{(1)} = \pbr{\Hm_{1},\th_{n}},\qquad \partial_{t} I_{n} = 0.
\]
The identity $\om_{n}^{(1)} = \pbr{\Hm_{1},\th_{n}}$ is the starting point for the new formula for $\om_{n}^{(1)}$. The asymptotics of the discriminant $\Dl(\lm,u)$ of the operator $-\partial_{x}^{2}+u$ at $\lm=\infty$ and the residue calculus allow us to expand $\om_{n}^{(1)} = \pbr{\Hm_{1},\th_{n}}$ into the constant term $(2n\pi)^{3}$ plus a weighted sum $\sum_{k\ge 1} k \Om_{nk}^{(2)}$ of functionals $\Om_{nk}^{(2)}$, each of which is an expression depending only on the discriminant $\Dl(\lm,u)$ -- or equivalently, the periodic spectrum of $-\partial_{x}^{2}+u$. Using that by~\cite{Kappeler:2005fb} the discriminant can be analytically extended to $\Hs_{0}^{-1}$, one shows that the same holds true for each functional $\Om_{nk}^{(2)}$. From the asymptotics of the periodic eigenvalues of $-\partial_{x}^{2}+u$ one then deduces that $(\sum_{k\ge 1} k \Om^{(2)}_{nk})_{n\ge 1}$ converges absolutely in $\ell^{-1,r}$, $r>1$, for $u$ in a complex neighborhood $\Ws$ of $\Hs_{0}^{-1}$ and when restricted to $\Hs_{0}^{s}$, $-1 \le s \le 0$, has the asymptotics stated in Theorem~\ref{thm:kdv-freq}.

In \cite{Kappeler:CNzeErmy} we proved by similar techniques that the renormalized Hamiltonian $\Hm_{1}^{\star} = H_{1} - \sum_{n\ge 1} (2n\pi)^{3}I_{n}$ analytically extends to the Fourier Lebesgue spaces $\FL_{0}^{-1/2,4}(\T,\R)$ or, considering $\Hm_{1}^{\star}$ as a function of the actions, to $\ell_{+}^{2}(\N)$. It implies that $(\om_{n}^{(1)\star} = \partial_{I_{n}}\Hm_{1}^{\star})_{n\ge 1}$ is in $\ell^{2}(\N)$. Since by \cite[Theorem~1]{Kappeler:CNzeErmy} $(\partial_{I_{n}}\Hm_{1}^{\star})_{n\ge 1}$ is a local diffeomorphism near $I=0$, $\Hm_{1}^{\star}$ does not admit a $C^{1}$-smooth extension to a neighborhood of $0$ in $\ell^{2s+1,p/2}$ for any $(s,p)\neq (-1/2,4)$ with $s\le -1/2$ and $4 \le p < \infty$.
Similarly, $\Hm_{1}^{\star}$ does not admit a $C^{1}$ extension to a neighborhood of the origin in $\ell_{+}^{2s+1,1}$ for any $-1 < s < -1/2$. In fact, this would imply that $\partial \Hm_{1}^{\star}$ takes values in $\ell^{-2s-1,\infty}$ while at the same time
$\partial \Hm_{1}^{\star} = \om^{(1)\star}\colon \ell_{+}^{2s+1,1}\to \ell^{2s+1,1}$ is a homeomorphism locally around the origin by Corollary~\ref{cor:kdv-freq-localdiffeo}, which is impossible.
Nevertheless, according to Theorem~\ref{thm:kdv-freq}, the frequencies $(\om_{n}^{(1)\star})_{n\ge 1}$ analytically extend to $\ell_{+}^{2s+1,1}(\N)$ for any $-1 < s < -1/2$.

{\em Notation.}
We collect a few notations used throughout the paper. A sequence of complex numbers $(a_{n})_{n\in\mathbb{A}}$ is denoted $a_{n} = \ell_{n}^{p} + \ell_{n}^{q}$ if it can be decomposed as $a_{n} = x_{n} + y_{n}$ with $(x_{n})\in\ell_{\C}^{p}(\mathbb{A})$ and $(y_{n})\in\ell_{\C}^{q}(\mathbb{A})$.
Here $0 < p,q \le \infty$ and $\ell_{\C}^{p}(\mathbb{A})$ denotes the vector space of sequences with $\sum_{n\in\mathbb{A}} \abs{x_{n}}^{p} < \infty$. Moreover, $a_{n} = \ell_{n}^{1+}$ means that $(a_{n})\in\ell_{\C}^{r}(\mathbb{A})$ for any $r > 1$. Finally, we define $x_{+} = \max(x,0)$.

We say the tuple $(s,p)$ of real numbers is admissible if either $p=2$ and $-1\le s < \infty$ or $2 < p < \infty$ and $-1/2\le s\le 0$.

\section{Preliminaries}

In this section we review results from \cite{Kappeler:CNzeErmy,Kappeler:2001bi,Kappeler:2005fb, Kappeler:2003up,Kappeler:2013bt,Kappeler:2003vh,Molnar:2016hq}. In addition, we prove asymptotics of spectral quantities for potentials in Fourier Lebesgue spaces.

\subsection{Spectral theory of Schrödinger operators}

Let $q$ be a complex potential in $\Hs_{0,\C}^{-1}$ and consider the differential operator
\begin{equation}
  \label{Lq}
  L(q) = -\partial_{x}^{2} + q.
\end{equation}
In the sequel we will only consider potentials $q\in\Wp$ with $\Wp$ denoting the complex neighborhood of $\Hs_{0}^{-1}$ in $\Hs_{0,\C}^{-1}$ of Theorem~\ref{bhf}.
If needed, we will shrink $\Ws$ further.
The spectrum of $L(q)$,  called the \emph{periodic spectrum of $q$}, is known to be discrete and the eigenvalues, when counted with their multiplicities and ordered lexicographically -- first by their real part and second by their imaginary part -- satisfy
\begin{align}
\label{lm-asymptotics}
  \lm_{0}^{+}(q) \lex \lm_{1}^{-}(q) \lex \lm_{1}^{+}(q) \lex \dotsb
          \ ,
  \qquad
  \lm_{n}^{\pm}(q) = n^{2}\pi^{2} + n\ell_{n}^{2}.
\end{align}
Furthermore, we define the \emph{gap lengths} $\gm_{n}(q)$ and the \emph{mid points} $\tau_{n}(q)$ by
\[
  \gm_{n}(q) \defl \lm_{n}^{+}(q)-\lm_{n}^{-}(q) = n\ell_{n}^{2},
  \quad
  \tau_{n}(q) \defl \frac{\lm_n^+(q) + \lm_n^-(q)}{2} = n^{2}\pi^{2} + n\ell_{n}^{2}.
\]
If $q$ is real-valued, then the periodic spectrum of $q$ as well as its gap lengths and mid points are real. Therefore, the lexicographical ordering reduces to the real ordering
\[
  \lm_{0}^{+} < \lm_{1}^{-} \le \lm_{1}^{+} < \lm_{2}^{-} \le \lm_{2}^{+} < \dotsb.
\]

The \emph{discriminant} $\Dl(\lm,q)$ of $L(q)$, defined for $q\in \Hs^{0}_{0,\C}$ by Floquet theory, admits an analytic extension to  $\C\times \Wp$. Furthermore, $\Dl^{2}(\lm,q)-4$ has the product representation
\begin{equation}
  \label{Dl2}
  \Dl^2(\lm,q) - 4
   =
  -4(\lm-\lm_{0}^{+})\prod_{m\ge 1} \frac{(\lm_m^+-\lm)(\lm_m^--\lm)}{m^4\pi^4}.
\end{equation}
The $\lm$-derivative $\dDl$ of the discriminant $\Dl$ is analytic on $\C\times\Ws$, too, and admits  the product representation
\begin{equation}
  \label{dDl2}
  \dDl(\lm)  = -\prod_{m\ge 1} \frac{\lm_{m}^{\ld}-\lm}{m^{2}\pi^{2}}
\end{equation}
where $(\lm_m^\ld)_{m \ge 1} \subset \C$ are ordered lexicographically
\[
\lm_1^\ld \lex \lm_2^\ld \lex \cdots \ , \quad \lm_n^\ld = n^2\pi^2 + n\ell_n^2 \ .
\]

We also consider the spectrum of the operator $L_{\dir}(q) = -\partial_{x}^{2} + q$ on $\Hs^{-1}_{\dir}([0,1],\C)$ with domain of definition $\Hs^{1}_{\dir}([0,1],\C)$ -- cf. e.g. \cite{Djakov:2009fx,Kappeler:2001bi,Kappeler:2003vh,Korotyaev:2003gp,Savchuk:2003vl} for a more detailed discussion. This spectrum, referred to as the \emph{Dirichlet spectrum of $q$}, is known to be discrete and to be given by a sequence of eigenvalues $(\mu_{n})_{n\ge 1}$, counted with multiplicities and ordered lexicographically so that
\[
  \mu_{1} \lex \mu_{2} \lex \mu_{2} \lex \dotsb,\qquad \mu_{n}  = n^{2}\pi^{2} + n\ell_{n}^{2}.
\]

For each potential $q\in \Hs_{0}^{-1}$ there exists a complex neighborhood $\Wp_{q}$ of $q$ in $\Wp$ such that the closed intervals
\[
  G_{0} = \setdef{\lm_{0}^{+} + t}{-\infty < t \le 0},\qquad
  G_{n} = [\lm_{n}^{-},\lm_{n}^{+}],\quad n\ge 1,
\]
are disjoint from each other for every $r\in \Wp_{q}$. Moreover, there exist open, connected, convex, and mutually disjoint neighborhoods $U_{n}\subsetneq \C$, $n\ge 0$, called \emph{isolating neighborhoods}, which satisfy:

\begin{equivenum}
\item $G_{n}$, $\mu_{n}$, and $\lm_{n}^{\ld}$ are contained in the interior of $U_{n}$ for every $r\in \Wp_{q}$,

\item there exists a constant $c \ge 1$ such that for all $n,m\ge 1$ with $m\neq n$
\begin{equation}
  \label{iso-est}
  c^{-1}\abs{m^{2}-n^{2}} \le \dist(U_{n},U_{m}) \le c\abs{m^{2}-n^{2}},
\end{equation}

\item
there exists an integer $n_{0}\ge 1$ so that
\begin{equation}
  \label{Un-Dn}
  U_{n} = D_{n} \defl \setdef{\lm\in\C}{\abs{\lm - n^{2}\pi^{2}} \le n},\qquad n\ge n_{0}.
\end{equation}
\end{equivenum}

\noindent
In the sequel, for any $q \in \Wp$, $\Wp_{q}$ denotes a neighborhood of $q$ in $\Wp$ such that a common set of isolating neighborhoods for all $r\in \Wp_{q}$ exists which are denoted by $U_{n}$, $n\ge 0$. We  shrink $\Wp$, if necessary, such that $\Wp$ is contained in the union of all $\Wp_{q}$ with $q\in \Hs_{0}^{-1}$.

We say that $q\in\Ws$ is a finite-gap potential if
\begin{equation}
  \label{fin-gap}
  S(q) = \setdef{k\in\N}{\gm_{k}(q)\neq 0}
\end{equation}
is finite. By Theorem~\ref{bhf} it follows that such potentials are $C^{\infty}$-smooth and dense in $\Ws$.

\subsection{Birkhoff coordinates for the KdV hierarchy}

Following~\cite{Kappeler:2005fb}, for $q \in \Wp$, one can define action variables for the KdV equation by
\begin{equation}
  \label{action}
  I_n
   =
  \frac{1}{\pi}\int_{\Gm_n}
  \frac{\lm \dDl(\lm)}{\sqrt[c]{\Dl^2(\lm)-4}}
  \,\dlm,\qquad n\ge 1.
\end{equation}
Here  $\Gm_{n}$ denotes any counter clockwise oriented circuit around and sufficiently close to $G_{n}$, and the \emph{canonical root} $\sqrt[c]{\Dl^{2}(\lm)-4}$ is defined on $\C\setminus \bigcup_{\atop{n\ge 0}{\gm_{n}\neq 0}} G_{n}$ with $\gm_{0}= \infty$, where, for $q$ real, the sign of the root is determined by
\[
  \ii\sqrt[c]{\Dl^{2}(\lm)-4} > 0,\qquad \lm_{0}^{+} < \lm < \lm_{1}^{-} \ ,
\]
and for $q\in\Wp$ it is defined by continuous extension.

The  Dirichlet eigenvalues and the discriminant can be used to construct the angles $\th_k(q)$, $k\ge 1$, which are conjugated to the actions $I_n(q)$, $n\ge 1$. In more detail, according to \cite{Kappeler:2005fb}, for any given $k\ge 1$ the action $I_{k}$ is a real analytic function on $\Wp$, whereas the angle $\th_{k}$ is defined modulo $2\pi$ on $\Wp\setminus Z_{k}$ and is a real analytic function on $\Wp\setminus Z_{k}$ when considered modulo $\pi$, where
\begin{equation}
  \label{Zk}
  Z_k \defl \setdef{q\in \Wp}{\gm_{k}(q) = 0}.
\end{equation}
It was shown in \cite[Proposition~4.3]{Kappeler:2005fb} that $Z_{k}\cap \Hs_{0}^{-1}$ is a real analytic submanifold of $\Hs_{0}^{-1}$ of codimension two.
Moreover, by \cite[Section 6]{Kappeler:2005fb} the following commutator relations hold for any $m,n\ge 1$
\begin{align}
  \label{e:commutators0}
  \pbr{I_m,I_n} = 0,\qquad
  \pbr{I_m,\th_n} = \dl_{nm},\qquad
  \pbr{\th_m,\th_n} = 0,
\end{align}
whenever the bracket is defined.

For any $q\in \Hs^{-1}_0\setminus Z_k$ with $k \ge 1$ define
\begin{equation}
  \label{e:z_k}
  z_k(q) \defl \sqrt[+]{I_k(q)}\, \e^{-\ii \th_k(q)},  \qquad
  z_{-k}(q) \defl \sqrt[+]{I_k(q)}\, \e^{\ii \th_k(q)}.
\end{equation}
It is shown in \cite[Section 5]{Kappeler:2005fb} that the mappings $\Hs^{-1}_0\setminus Z_k\to{\C}$, $q\mapsto z_{\pm k}(q)$, analytically extend to the neighborhood $\Wp$.
The {\em Birkhoff map} is then defined as follows
\begin{equation}
  \label{Phi.def}
  \Phi\colon \Wp \to  h_{0,\C}^{-1/2}, \quad
  q \mapsto \Phi(q) \defl \left( z_k(q)\right)_{k \in  \Z}
\end{equation}
with $z_0(q) = 0$. Its main properties are stated in Theorem~\ref{bhf}.

In addition, it was shown in~\cite{Kappeler:2013bt} that $\Phi$ and its inverse are 1-smoothing. More precisely, for any integer $N\ge 0$, the maps
\begin{equation}
  \label{Phi-1-smoothing}
  \Phi - \ddd_{0}\Phi \colon \Hs_{0}^{N} \to h_{0}^{N+3/2},\qquad
  \Phi^{-1} - (\ddd_{0}\Phi)^{-1} \colon h_{0}^{N+1/2} \to \Hs_{0}^{N+1}
\end{equation}
are analytic and bounded, i.e. bounded on bounded subsets. We note the following immediate consequence for later use.

\begin{lem}
\label{Om-unif-continous}
For any $s > 0$ the Birkhoff map and its inverse
\[
  \Phi\colon \Hs^{s}\to  h^{s+1/2},\qquad
  \Phi^{-1}\colon  h^{s+1/2}\to \Hs^{s}
\]
are uniformly continuous on bounded subsets.~\fish
\end{lem}

\begin{proof}
By the 1-smoothing property~\eqref{Phi-1-smoothing} for any integer $N\ge 0$ the map $\Phi - \ddd_{0}\Phi\colon \Hs^{N}\to  h^{N+3/2}$ is continuous and hence uniformly continuous on compacts. Since $\Hs^{s}$ embeds compactly into $\Hs^{[s]}$ if $s > [s]$, we conclude that
\[
  \Phi - \ddd_{0}\Phi\colon\Hs^{s}\opento_{c} \Hs^{[s]} \to \ell^{[s]+3/2} \opento \ell^{s+1/2}
\]
is uniformly continuous on bounded sets for any $s > 0$. Clearly, $\ddd_{0}\Ph\colon \Hs^{s}\to  h^{s+1/2}$ is uniformly continuous as well, which gives the claim for $\Phi$. One argues analogously for the inverse.~\qed
\end{proof}

\subsection{Roots and Abelian integrals}

It is convenient to define the \emph{standard root}s
\[
  \vs_{n}(\lm) = \sqrt[\mathrm{s}]{(\lm_{n}^{+}-\lm)(\lm_{n}^{-}-\lm)},
                 \qquad \lm\in\C\setminus G_{n},\qquad n\ge 1,
\]
by the condition
\begin{align}
  \label{s-root}
  \vs_{n}(\lm) = (\tau_{n}-\lm)\sqrt[+]{1 - \gm_{n}^{2}/4(\tau_{n}-\lm)^{2}},
  						 \qquad \tau_{n} = (\lm_{n}^{-}+\lm_{n}^{+})/2.
\end{align}
Here $\sqrt[+]{\phantom{a}}$ denotes the principal branch of the square root on the complex plane minus the ray $(-\infty,0]$. The standard root is analytic in $\lm$ on $\C\setminus G_{n}$ and in $(\lm,r)$ on $(\C\setminus \ob{U_{n}})\times \Wp_{q}$, and one can choose $c > 0$ locally uniformly on $U_{n}$ so that for all $n,m\ge 1$
\begin{equation}
  \label{s-root-est}
  \inf_{\lm\in U_{n}}\abs{\vs_{m}(\lm)} \ge c^{-1}\abs{n^{2}-m^{2}}.
\end{equation}
If $\gm_{n} = 0$, then $\vs_{n}(\lm) = (\tau_{n}-\lm)$ is an entire function of $\lm$.
On the other hand, if $\gm_{n}\neq 0$, then $\vs_{n}$ extends continuously to both sides of $G_{n}$, denoted by $G_{n}^{\pm}$,
\begin{equation}
  \label{Gn-sides}
  G_{n}^{\pm} = \setdef{\lm_{t}^{\pm} = \tau_{n} + (t \pm \ii 0)\gm_{n}/2}{-1\le t\le 1},
\end{equation}
and we have
\begin{equation}
  \label{s-root-sides}
  w_{n}(\lm_{t}^{\pm}) = \mp \ii \frac{\gm_{n}}{2}\sqrt[+]{1-t^{2}},\qquad -1\le t\le 1.
\end{equation}

\begin{lem}
\label{s-root-reciprocal-estimate}
Suppose $\gm_{n}\neq 0$ and $f$ is continuous on $G_{n}$, then
\[
  \sup_{\lm\in G_{n}^{+}\cup G_{n}^{-}}
  \abs*{\frac{1}{\pi}\int_{\lm_{n}^{-}}^{\lm} \frac{f(z)}{w_{n}(z)}\,\dz}
  \le
  \max_{\lm\in G_{n}} \abs{f(\lm)}.\fish
\]
\end{lem}

\begin{proof}
We choose the parametrization $\lm_{t}^{\pm}$ of $G_{n}^{\pm}$ to obtain for $-1\le t\le 1$,
\[
  \int_{\lm_{n}^{-}}^{\lm_{t}^{\pm}} \frac{f(z)}{w_{n}(z)}\,\dz
  =
  \pm \ii \int_{-1}^{t} \frac{f(\lm_{r}^{\pm})}{\sqrt[+]{1-r^{2}}}\,\dr.
\]
Since $\int_{-1}^{1} \frac{1}{\sqrt[+]{1-r^{2}}}\,\dr = \pi$, the claim follows immediately.\qed
\end{proof}

\begin{lem}
\label{int-wm-quot-est}
Suppose $f$ is analytic in a neighborhood of $G_{n}$ containing $\Gm_{n}$, then
\[
  \frac{1}{2\pi}\abs*{\int_{\Gm_{n}} \frac{f(\lm)}{\vs_{n}(\lm)}\,\dlm} \le
  \max_{\lm\in G_{n}}\abs{f(\lm)}.~\fish
\]
\end{lem}

\begin{proof}
If $\gm_{n} = 0$, then $\vs_{m}(\lm) = \tau_{n}-\lm$ and the claim follows from Cauchy's theorem. Conversely, if $\gm_{n}\neq 0$, then we may apply the previous lemma.~\qed
\end{proof}

The \emph{canonical root} $\sqrt[c]{\Dl^{2}(\lm)-4}$ can be written in terms of standard roots as follows
\begin{equation}
  \label{c-root}
  \sqrt[c]{\Dl^{2}(\lm)-4} \defl
   -2\ii\sqrt[+]{\lm-\lm_{0}^{+}}\prod_{m\ge 1} \frac{\vs_{m}(\lm)}{m^{2}\pi^{2}}
\end{equation}
and is analytic in $\lm$ on $\C\setminus\bigcup_{\gm_{m}\neq 0} G_{m}$ and in $(\lm,r)$ on $(\C\setminus \bigcup_{m\ge 0} \ob{U_{m}})\times \Wp_{q}$. In particular, the quotient
\begin{equation}
  \label{Dl-quot-prod}
  \frac{\dDl(\lm)}{\sqrt[c]{\Dl^{2}(\lm)-4}}
   = \frac{1}{2\ii}\frac{1}{\sqrt[+]{\lm-\lm_{0}^{+}}}\prod_{m\ge 1}
     \frac{\lm_{m}^{\ld}-\lm}{\vs_{m}(\lm)},
\end{equation}
is analytic in $(\lm,r)$ on $(\C\setminus \bigcup_{m\ge 0} \ob{U_{m}}) \times \Wp_{q}$, and analytic in $\lm$ on $\C\setminus \bigcup_{\atop{\gm_{m}\neq 0}{m\ge 0}} G_{m}$ where we set $\gm_{0} = \infty$ for convenience.

A path in the complex plane is said to be \emph{admissible} for $q$ if, except possibly at its endpoints, it does not intersect any non collapsed gap $G_{n}(q)$. For any $n\ge 1$ and any admissible path from $\lm_{n}^{-}$ to $\lm_{n}^{+}$ in $U_{n}$ we have
\begin{equation}
  \label{w-closed}
  \int_{\lm_{n}^{-}}^{\lm_{n}^{+}} \frac{\dDl(\lm)}{\sqrt[c]{\Dl^{2}(\lm)-4}}\,\dlm =  0.
\end{equation}
As a consequence,
$
  \int_{\Gm_{n}} \frac{\dDl(\lm)}{\sqrt[c]{\Dl^{2}(\lm)-4}}\,\dlm = 0.
$
for any closed circuit $\Gm_{n}$ in $U_{n}$ around $G_{n}$ -- cf. e.g. \cite{Molnar:2016hq}.

Next we define for any $q\in\Wp$ and $\lm\in \C\setminus \bigcup_{\atop{\gm_{m}\neq 0}{m\ge 0}} G_{m}$ the improper integral
\begin{equation}
  \label{F-def}
    F(\lm) \defl \int_{\lm_{0}^{+}}^{\lm} \frac{\dDl(z)}{\sqrt[c]{\Dl^{2}(z)-4}}\,\dz,
\end{equation}
computed along an arbitrary admissible path. The improper integral $F(\lm)$ exists as in the product representation~\eqref{Dl-quot-prod} the factor ${1}/{\sqrt[+]{\lm-\lm_{0}^{+}}}$ is integrable on $\C\setminus G_{0}$. Furthermore, in view of~\eqref{w-closed} it is independent of the chosen admissible path and hence well defined. Moreover, $F(\lm)$ continuously extends to $G_{n}^{\pm}$ for any $n\ge 0$, where $G_{n}^{+}$ [$G_{n}^{-}$] is the left [right] hand side of $G_{n}$.

\begin{rem}
For $q\in L_{0}^{2}$, $F(\lm)$ is one of the two Floquet exponents of the operator $L(q)-\lm$ meaning that $\e^{F(\lm)}$ is an eigenvalue of the Floquet matrix associated to $L(q)-\lm$.\map
\end{rem}

\begin{lem}[\cite{Kappeler:CNzeErmy}]
\label{F-prop}
For any $q\in \Wp$ the following holds:

\begin{equivenum}
\item $F$ is analytic in $(\lm,r)$ on $(\C\setminus\bigcup_{m\ge 0} \ob{U_{m}}) \times \Wp_{q}$ with $L^{2}$-gradient
\[
  \partial_{q} F(\lm) = \frac{\partial_{q}\Dl(\lm)}{\sqrt[c]{\Dl^{2}(\lm)-4}}.
\]
Further, $F(\lm)\equiv F(\lm,q)$ is analytic in $\lm$ on $\C\setminus\bigcup_{\atop{\gm_{m}\neq 0}{m\ge 0}} G_{m}$. (We recall that $\gm_{0} = \infty$.)

\item $F(\lm_{0}^{+}) = 0$ and $F(\lm_{n}^{+}) = F(\lm_{n}^{-}) = -\ii n\pi$ for any $n\ge 1$.

\item
Locally uniformly on $\Wp_{q}$
\[
  \sup_{\lm\in G_{n}^{+}\cup G_{n}^{-}}\abs{F(\lm)+\ii n\pi} = O(\gm_{n}/n),\qquad n\to\infty.
\]

\item For $q = 0$, $F(\lm)$ is analytic on $\C\setminus(-\infty,0]$ and given by $F(\lm) = -\ii \sqrt[+]{\lm}$.

\item
If $q$ is a real-valued finite-gap potential cf.~\eqref{fin-gap} with $[q] = 0$ and $\nu_{n} = (n+1/2)\pi$, then for any $K\ge 0$
\begin{equation}
  \label{F-exp}
    F(\nu_{n}^{2})
   = -\ii\nu_{n} + \ii \sum_{0\le k \le K} \frac{\Hm_{k}}{4^{k+1}\nu_{n}^{2k+3}}
     + O(\nu_{n}^{-2K-5}),\qquad n\to \infty,
\end{equation}
where $\Hm_{k}$ denotes the $k$th Hamiltonian of the KdV hierarchy.~\fish
\end{equivenum}

\end{lem}

Occasionally we write for $n\ge 0$
\begin{equation}
  \label{Fn-integral}
  F_n(\lm) \defl F(\lm) + \ii n\pi = \int_{\lm_n^+}^{\lm} \frac{\dDl}{\sqrt[c]{\Dl^{2}-4}}\,\dz
\end{equation}
to denote the primitive of $\frac{\dDl}{\sqrt[c]{\Dl^{2}-4}}$ normalized by the condition $F_n(\lm_n^+) = 0$.

\begin{lem}
\label{F2-analytic}
\begin{equivenum}
\item
For any $q\in \Wp$ and any $n\ge 0$, the function $F_{n}^{2}(\lm)$ is analytic on $U_{n}$ and hence on $\C\setminus\bigcup_{\atop{n\neq m\ge 0}{\gm_{m}\neq 0}} G_{m}$.

In particular, for a finite-gap potential cf.~\eqref{fin-gap}, $F^{2}(\lm) = F_{0}^{2}(\lm)$ is analytic on $U_{0}$ and hence outside a disc centered at zero of sufficiently large radius.

\item
For any $q\in \Wp$ and any $n\ge 0$ and $l\ge 0$, the function $\frac{F_{n}^{2l+1}(\lm)}{\sqrt[c]{\Dl^{2}(\lm)-4}}$ is analytic on $U_{n}$. In particular, $\frac{F^{2l+1}(\lm)}{\sqrt[c]{\Dl^{2}(\lm)-4}}$ is analytic on $\C\setminus\bigcup_{m\ge 1} G_{m}$.~\fish
\end{equivenum}
\end{lem}

\begin{proof}
The proof of item (i) can be found in \cite[Lemma~4.3]{Kappeler:CNzeErmy}.
To prove item (ii) we note that both $F_{n}$ and the canonical root admit opposite signs on opposite sides of $G_{n}$ and they vanish on $U_{n}$ only at $\lm_{n}^{\pm}$ hence the quotient $F_{n}(\lm)/\sqrt[c]{\Dl^{2}(\lm)-4}$ is analytic on $U_{n}\setminus\setd{\lm_{n}^{-},\lm_{n}^{+}}$. Moreover, by  l'Hopital's rule one has $F_{n}(\lm)/\sqrt[c]{\Dl^{2}(\lm)-4}\big|_{\lm_{n}^{\pm}} = \frac{1}{\Dl(\lm_{n}^{\pm})} = \frac{(-1)^{n}}{2}$ so the quotient is continuous and hence analytic on all of $U_{n}$. Together with item (i) it thus follows that $F_{n}^{2l+1}(\lm)/\sqrt[c]{\Dl^{2}(\lm)-4}$ is analytic on $U_{n}$ as well.\qed
\end{proof}

\subsection{Asymptotics of spectral quantities in Fourier Lebesgue spaces}

Recall that we say the tuple $(s,p)$ of real numbers is admissible if either $p=2$ and $-1\le s < \infty$ or $2 < p < \infty$ and $-1/2\le s\le 0$. For $(s,p)$ admissible we introduce
\[
  \Ws^{s,p} = \Ws \cap \FL^{s,p}_{0,\C},\qquad \FL^{s,p}_{0,\C} = 
  \setdef{u\in \Sc_{\C}'}{(u_{n})\in \ell_{0,\C}^{s,p}},
\]
 -- cf. \eqref{Hs-FLsp-def}.
According to \cite{Kappeler:CNzeErmy,Kappeler:1999er} for $q\in\Ws^{s,p}$ the estimates of the periodic eigenvalues~\eqref{lm-asymptotics} can be refined to
\begin{equation}
  \label{per-ev-est-sp}
    \lm_{n}^{\pm} = n^{2}\pi^{2} + n^{-s}\ell_{n}^{p}.
\end{equation}
This estimate holds locally uniformly on $\Ws^{s,p}$. In more detail,
\[
  \sum_{n\ge 1} n^{sp}\abs{\lm_{n}^{\pm}-n^{2}\pi^{2}}^{p} \le C,
\]
where the constant $C$ can be chosen locally uniformly on $\Ws^{s,p}$.
We note that~\eqref{per-ev-est-sp} immediately implies that $\tau = (\tau_{n})_{n\ge 1}$ and $\gm = (\gm_{n})_{n\ge 1}$ satisfy
\begin{equation}
  \label{taun-lmn-est}
  \tau_{n} = n^{2}\pi^{2} + n^{-s}\ell_{n}^{p},\qquad
  \gm_{n} = n^{-s}\ell_{n}^{p}.
\end{equation}
It was shown in~\cite[Proposition 2.18]{Kappeler:2005fb} that for every $\ep > 0$ there exists $n_{\ep}\ge n_{0}$ so that
\begin{equation}
  \label{lmld-taun-selecta-ep}
  \abs{\lm_{n}^{\ld} - \tau_{n}} \le \ep \abs{\gm_{n}},\qquad n\ge n_{\ep},
\end{equation}
and $n_{\ep}$ can be chosen locally uniformly on $\Ws$. As an immediate consequence of~\eqref{taun-lmn-est} and~\eqref{lmld-taun-selecta-ep},
\begin{equation}
  \label{lm-ld-rough}
  \lm_{n}^{\ld} = n^{2}\pi^{2} + n^{-s}\ell_{n}^{p}.
\end{equation}
We proceed by further refining the estimate of $\lm_{n}^{\ld}$ as well as other quantities derived from the periodic spectrum of $q$. A key ingredient into the proof of these estimates is an estimate of functions of the form
\begin{equation}
  \label{phi-n-func}
    f_{n}(\lm) = \frac{n\pi}{\sqrt[+]{\lm- \lm_{0}^{+}}}\prod_{m\neq n}\frac{\sg_{m}-\lm}{\vs_{m}(\lm)},
\end{equation}
with $\tilde\sg  = (\sg_{n}-n^{2}\pi^{2})_{n\ge 1}\in h_{\C}^{-1}$. Note that for each $n\ge 1$, the function $f_{n}$ is analytic in $(\lm,\tilde \sg,q)$ on $(\C\setminus\bigcup_{m\neq n} \ob{U_{m}})\times h_{\C}^{-1} \times \Ws_{q}$ - cf.~\cite[Corollary~12.8]{Grebert:2014iq}.

In a first step we estimate the infinite-product part of $f_{n}$. To simplify notation we write for any subset $U$ of the complex plane $\abs{f}_{U} \defl \sup_{\lm\in U} \abs{f(\lm)}$.

\begin{lem}
\label{inf-prod-root-quot}
Suppose $(s,p)$ is admissible with $-1 \le s\le 0$. For any $\sg = (\sg_{n})_{n\ge 1}\subset\C$ with $\sg_{n}-\tau_{n} = n^{-t}\ell_{n}^{r}$, $-1\le t\le 0$, and $1 < r < \infty$,
\[
  \abs*{\prod_{m\neq n}\frac{\sg_{m}-\lm}{\vs_{m}(\lm)} - 1}_{U_{n}}
   = n^{-1-t}\ell_{n}^{r}+n^{-2-2s}\ell_{n}^{p/2},
\]
uniformly in $\n{\sg-\tau}_{t,r}$ and locally uniformly on $\Ws^{s,p}$. In more detail, one has $\prod_{m\neq n}\frac{\sg_{m}-\lm}{\vs_{m}(\lm)} = 1 + a_{n}(\lm) + b_{n}(\lm)$ where the functions $a_{n}$ and $b_{n}$ satisfy the estimate
\[
  \sum_{n\ge 1} \p*{n^{(1+t)r}\abs{a_{n}}_{U_{n}}^{r} + n^{(2+2s)p/2}\abs{b_{n}}_{U_{n}}^{p/2}} \le C,
\]
and the absolute constant $C$ can be chosen uniformly in $\n{\sg-\tau}_{t,r}$ and locally uniformly on $\Ws^{s,p}$.~\fish
\end{lem}

\begin{proof}
Write the product in the form
\begin{equation}
  \label{prod-root-quot-1}
  \prod_{m\neq n}\frac{\sg_{m}-\lm}{\vs_{m}(\lm)}
  = \prod_{m\neq n}\frac{\sg_{m}-\lm}{\tau_{m}-\lm}
    \prod_{m\neq n}\p*{1- \frac{\gm_{m}^{2}}{4(\tau_{m}-\lm)^{2}}}^{-1/2}.
\end{equation}
Here $x^{-1/2}$ denotes the standard branch $\sqrt[+]{x}$ of the square root which is analytic on $\C\setminus(-\infty,0]$.
By~\eqref{Un-Dn} we have $U_{n} = D_{n}$ for $n\ge n_{0}$ where $n_{0}$ can be chosen locally uniformly on $\Ws$. Consequently, the first factor is $1 + n^{-1-t}\ell_{n}^{r}$ in view of~\eqref{iso-est} and Lemma~\ref{inf-prod-quot}.
For the second factor, note that $\abs*{\frac{\gm_{m}^{2}}{4(\tau_{m}-\lm)^{2}}}_{U_{n}}
   = O\p*{ \frac{\gm_{m}^{2}}{(n^{2}-m^{2})^{2}} }$ for all $n\ge n_{0}$ again in view of~\eqref{iso-est}. Applying the estimate
\[
  \frac{\abs{\gm_{m}}^{2}}{(n^{2}-m^{2})^{2}}
  \le
  \begin{cases}
  4\n{\gm}_{h^{-1}}^{2}/n^{2}, & \abs{n-m} > n/2,\\
  \n{R_{n/2}\gm}_{h^{-1}}^{2}, & 1 \le \abs{n-m} \le n/2,
  \end{cases}
\]
where $R_{n/2}(\gm) = (\gm_{m})_{m\ge n/2}$, shows that one can choose $\tilde n_{0}\ge n_{0}$ locally uniformly in $\Ws$ so that $\abs*{\frac{\gm_{m}^{2}}{4(\tau_{m}-\lm)^{2}}}_{U_{n}}\le 1/2$ for all $m\ge 1$ with $m\neq n$ and all $n\ge\tilde n_{0}$. Invoking the estimate $\abs{1/\sqrt[+]{1+x}-1}\le \abs{x}$ for $\abs{x} \le 1/2$ then gives
\[
  \abs*{\p*{1- \frac{\gm_{m}^{2}}{4(\tau_{m}-\lm)^{2}}}^{-1/2}-1} \le
  \abs*{\frac{\gm_{m}^{2}}{4(\tau_{m}-\lm)^{2}}} = \frac{n^{-2-2s}\ell_{m}^{p/2}}{(n-m)^{2}},
\]
where we used that $\gm_{m}^{2} = m^{-2s}\ell_{m}^{p/2}$ and $\frac{1}{(n^{2}-m^{2})^{2}} \le \frac{1}{n^{2+2s}m^{-2s}(n-m)^{2}}$.
Therefore, Lemma~\ref{appl-young} yields $\sum_{m\neq n} \abs*{\frac{\gm_{m}^{2}}{4(\tau_{m}-\lm)^{2}}} = n^{-2-2s}\ell_{n}^{p/2}$ and finally, in view of Lemma~\ref{inf-prod-est}, we conclude
\[
  \prod_{m\neq n}\p*{1- \frac{\gm_{m}^{2}}{4(\tau_{m}-\lm)^{2}}}^{-1/2}
  = 1+n^{-2-2s}\ell_{n}^{p/2}.
\]
By going through the arguments of the proof one verifies the claimed uniformity statement.\qed
\end{proof}

It remains to estimate the $\frac{n\pi}{\sqrt[+]{\lm-\lm_{0}^{+}}}$ term of $f_{n}$, introduced in~\eqref{phi-n-func}. 

\begin{lem}
\label{lm0-est}
Suppose $(s,p)$ is admissible with $-1\le s\le 0$, then
\[
 \abs*{\frac{n\pi}{\sqrt[+]{\lm- \lm_{0}^{+}}}-1}_{G_{n}}
 = n^{-2-s}\ell_{n}^{p} + n^{-2}\ell_{n}^{\infty}
 = n^{-3/2-s}(\ell_{n}^{p/2}+\ell_{n}^{1+}) + n^{-1}\ell_{n}^{1+},
\]
locally uniform on $\Ws^{s,p}$.~\fish
\end{lem}

\begin{proof}
Let $(\lm_{n})_{n\ge 1}$ be any sequence with $\lm_{n} \in G_{n}$, then by~\eqref{taun-lmn-est} we can write $\lm_{n} = n^{2}\pi^{2} + n^{-s}a_{n}$ with $a_{n} = \ell_{n}^{p}$.
Using that $\abs{(1+x)^{-1/2}-1}\le \abs{x}$ for $\abs{x} \le 1/2$, we conclude
\[
  \abs*{\frac{n\pi}{\sqrt[+]{\lm - \lm_{0}^{+}}}-1}
  = \abs*{\p*{1  + \frac{n^{-s}a_{n}}{n^{2}\pi^{2}} - \frac{\lm_{0}^{+}}{n^{2}\pi^{2}}}^{-1/2} - 1}
  = \frac{\ell_{n}^{p}}{n^{2+s}} + \frac{\ell_{n}^{\infty}}{n^{2}}.
\]
For $p>2$, one has $n^{-2-s}\ell_{n}^{p} = n^{-3/2-s}\ell_{n}^{p/2}$ whereas for $p=2$ we have $n^{-2-s}\ell_{n}^{p} = n^{-3/2-s}\ell_{n}^{1+}$. Moreover, $n^{-2}\ell_{n}^{\infty} = n^{-1}\ell_{n}^{1+}$. By going through the arguments of the proof, one sees that estimate holds locally uniformly on $\Ws^{s,p}$.\qed
\end{proof}

Combining the previous two lemmas yields the following estimate for the function $f_{n} = \frac{n\pi}{\sqrt[+]{\lm- \lm_{0}^{+}}}\prod_{m\neq n}\frac{\sg_{m}-\lm}{\vs_{m}(\lm)}$.

\begin{prop}
\label{std-prd-est}
Suppose $(s,p)$ is admissible with $-1\le s\le 0$, and $\sg=(\sg_{n})_{n\ge1}\subset\C$ with $\sg_{n}-\tau_{n} = n^{-t}\ell_{n}^{r}$, $s\le t \le 0$, and $1 < r\le p$. Then
\begin{align*}
  \abs*{f_{n}(\lm)-1}_{G_{n}}
   &= n^{-1-t}\ell_{n}^{r} + n^{-1+(-1-2s)_{+}}(\ell_{n}^{p/2}+\ell_{n}^{1+}).
\end{align*}
uniformly in $\n{\sg-\tau}_{t,r}$ and locally uniformly on $\Ws^{s,p}$.~\fish
\end{prop}

\begin{rem}
\label{std-prd-simple-est}
Proposition~\eqref{std-prd-est} implies the simpler estimate
\[
  \abs{f_{n}(\lm)-1}_{G_{n}} = n^{-1-s}\ell_{n}^{p}.\map
\]
\end{rem}

We are now in a position to obtain refined estimates for $\lm_{n}^{\ld}$ and $I_{n}$.

\begin{lem}
\label{In-lmd-est}
If $(s,p)$ with $-1\le s < \infty$ is admissible, then locally uniformly on $\Ws^{s,p}$
\begin{equivenum}
\item
\[
  \lm_{n}^{\ld} = \tau_{n} + \gm_{n}^{2}n^{-1}\ell_{n}^{p},
\]
\item
\[
  \frac{8n\pi I_{n}}{\gm_{n}^{2}} = 1 + n^{-1+(-1-2s)_{+}}(\ell_{n}^{p/2}+ \ell_{n}^{1+}).~\fish
\]
\end{equivenum}
\end{lem}

\begin{rem}
\begin{equivenum}
\item
Estimate (i) implies that for any $n\ge 1$ with $\gm_{n} = 0$ one has the standard identity $\lm_{n}^{\ld} = \tau_{n}$.
\item
Since $\gm_{n} = n^{-s}\ell_{n}^{p}$ by~\eqref{taun-lmn-est}, we conclude for any $(s,p)$ admissible
\begin{equation}
  \label{In-gmn-est}
  I_{n} - \frac{\gm_{n}^{2}}{8n\pi} = 
  \begin{cases}
  n^{-3-4s}\ell_{n}^{1} & -1\le s\le -1/2,\\
  n^{-2-2s}(\ell_{n}^{p/4} + \ell_{n}^{1}) & -1/2 < s <\infty.
  \end{cases}
  \map
\end{equation}
\end{equivenum}

\end{rem}

\begin{proof}
(i)
It was shown in~\cite[Proposition 2.19]{Kappeler:2005fb} that
\begin{equation}
  \label{lmld-taun-selecta-2}
  \lm_{n}^{\ld} - \tau_{n} = O(\gm_{n}^{2})
\end{equation}
locally uniformly in $\Ws$ whence it suffices to prove the claimed asymptotics for $n$ sufficiently large.

With $\Dl_{n}(\lm) \defl \frac{\lm_{0}^{+}-\lm}{n^{4}\pi^{4}}\prod_{m\neq n} \frac{(\lm_{m}^{+}-\lm)(\lm_{m}^{-}-\lm)}{m^{4}\pi^{4}}$ the identity $0 = \left.\frac{1}{4}\partial_{\lm}(\Dl^{2}(\lm)-4)\right|_{\lm_{n}^{\ld}}$ can be written as
\begin{align}
\label{lmd-tau-relation}
  0 = 2(\lm_{n}^{\ld}-\tau_{n})\Dl_{n}(\lm_{n}^{\ld}) + \p*{(\lm_{n}^{\ld} - \tau_{n})^{2} - {\gm_{n}^{2}}/{4} }\dDl_{n}(\lm_{n}^{\ld}).
\end{align}
By Lemma~\ref{inf-prod-sin} from Appendix~\ref{app:inf-prod} and since $2\le p < \infty$ and hence $n^{-2} = n^{-1}\ell_{n}^{p}$, one has uniformly for $\lm \in D_{n}$ with $n\ge n_{0}$,
\begin{align*}
  \Dl_{n}(\lm)
  &= \frac{-\frac{\lm}{n^{2}\pi^{2}} + O(n^{-2})}{n^{2}\pi^{2}}
  \p*{\frac{n^{2}\pi^{2}}{n^{2}\pi^{2}-\lm}\frac{\sin \sqrt{\lm}}{\sqrt{\lm}}}^{2}\p*{1 + n^{-1-s}\ell_{n}^{p}}^{2}\\
  &= \frac{-1}{n^{2}\pi^{2}}
  \p*{
  \frac{\lm}{n^{2}\pi^{2}}
  \p*{\frac{n^{2}\pi^{2}}{n^{2}\pi^{2}-\lm}\frac{\sin \sqrt{\lm}}{\sqrt{\lm}}}^{2} + n^{-1-s}\ell_{n}^{p}}.
\end{align*}
Since $\inf_{\lm\in D_{n}}\abs*{\frac{n^{2}\pi^{2}}{n^{2}\pi^{2}-\lm}\frac{\sin \sqrt{\lm}}{\sqrt{\lm}}} \ge 1/4$ by~\eqref{prod-sin-2} for all $n\ge n_{0}$, it follows that
\begin{equation}
  \label{Dln-est}
  \inf_{\lm\in D_{n}}\abs{\Dl_{n}(\lm)} \ge \frac{1}{32\pi^{2}n^{2}}
\end{equation}
for all $n$ is sufficiently large. On the other hand, let $D_{n}' = \setdef{\lm\in\C}{\abs{\lm-n^{2}\pi^{2}} \le n/2}$ hence $\dist(D_{n}',\partial D_{n}) = n/2$, and by Cauchy's estimate and the above estimate of $\Dl_{n}$,
\begin{equation}
  \label{dDln}
  \abs*{\partial_{\lm}\p*{
  \Dl_{n}(\lm)
   - \frac{-\lm}{n^{4}\pi^{4}}\p*{\frac{n^{2}\pi^{2}}{n^{2}\pi^{2}-\lm}\frac{\sin \sqrt{\lm}}{\sqrt{\lm}}}^{2} }}_{D_{n}'}
    = \frac{n^{-3-s}\ell_{n}^{p}}{n/2}
    = n^{-4-s}\ell_{n}^{p}.
\end{equation}
Writing $\sqrt{\lm} = n\pi + \al$, a straightforward computation gives
\begin{align*}
  \partial_{\lm}\p*{
  \frac{\lm}{n^{4}\pi^{4}}
  \p*{\frac{n^{2}\pi^{2}}{n^{2}\pi^{2}-\lm}\frac{\sin \sqrt{\lm}}{\sqrt{\lm}}}^{2}
  }
  = \frac{-2(n\pi + \al)\sin^{2}(\al) + \frac{1}{2}\al(2n\pi + \al)\sin(2\al)}{\al^{3}(n\pi+\al)(2n\pi+\al)^{3}}.
\end{align*}
Inserting the expansion $\sin x = x + O(x^{3})$ gives
\[
  -2(n\pi + \al)\sin^{2}(\al) + \frac{1}{2}\al(2n\pi + \al)\sin(2\al)
   = -\al^{3} + n O(\al^{4})
\]
so that
\begin{equation}
  \label{dDln0}
  \partial_{\lm}\p*{
  \frac{\lm}{n^{4}\pi^{4}}
  \p*{\frac{n^{2}\pi^{2}}{n^{2}\pi^{2}-\lm}\frac{\sin \sqrt{\lm}}{\sqrt{\lm}}}^{2}
  } = O(1/n^{4}) + O(\al/n^{3}).
\end{equation}
In view of~\eqref{lm-ld-rough} we have $\sqrt{\lm_{n}^{\ld}} = n\pi + n^{-1-s}\ell_{n}^{p}$,
and by combining~\eqref{dDln} and~\eqref{dDln0} one gets
\begin{equation}
  \label{dDln-lmnld-est}
  \dDl_{n}(\lm_{n}^{\ld}) = O(1/n^{4}) + n^{-4-s}\ell_{n}^{p} = n^{-3}\ell_{n}^{p}.
\end{equation}
Substituting the lower bound~\eqref{Dln-est} of $\Dl_{n}$ and the estimate~\eqref{dDln-lmnld-est} of $\dDl_{n}(\lm_{n}^{\ld})$ into identity~\eqref{lmd-tau-relation} and using that by~\eqref{lmld-taun-selecta-ep} $\abs{\lm_{n}^{\ld}-\tau_{n}}\le \abs{\gm_{n}}$ for $n$ sufficiently large thus gives
\[
  \abs{\lm_{n}^{\ld}-\tau_{n}} = \gm_{n}^{2}n^{-1}\ell_{n}^{p}.
\]

(ii)
By~\cite[Proposition 3.3]{Kappeler:2005fb} the quotient $I_{n}/\gm_{n}^{2}$ is real-analytic on $\Ws$, hence it suffices to prove the asymptotics for $n$ sufficiently large.
Since $\int_{\Gm_{n}} \frac{\dDl(\lm)}{\sqrt[c]{\Dl^{2}(\lm)-4}}\,\dlm = 0$ by~\eqref{w-closed}, one can write the actions, defined in~\eqref{action}, in the form $I_{n} = -\frac{1}{\pi}\int_{\Gm_{n}} (\lm_{n}^{\ld}-\lm)\frac{\dDl(\lm)}{\sqrt[c]{\Dl^{2}(\lm)-4}}\,\dlm$.
In case $\gm_{n}\neq 0$, one can deform the contour of integration $\Gm_{n}$ to the straight line $G_{n}$ and insert the product representation~\eqref{Dl-quot-prod} of $\dDl(\lm)/\sqrt[c]{\Dl^{2}(\lm)-4}$ to obtain
\begin{equation}
  \label{In-chin}
  n\pi I_{n} = 
  -\frac{1}{\ii}
  \int_{G_{n}^{-}}
  \frac{(\lm_{n}^{\ld}-\lm)^{2}\chi_{n}(\lm)}{w_{n}(\lm)}\,\dlm,\quad
  \chi_{n}(\lm) \defl \frac{n\pi}{\sqrt[+]{\lm-\lm_{0}^{+}}}\prod_{m\neq n}
     \frac{\lm_{m}^{\ld}-\lm}{\vs_{m}(\lm)}.
\end{equation}
Using the parametrization~\eqref{Gn-sides}, $\lm_{t}^{\pm} = \tau_{n} + (t\pm \ii)\gm_{n}/2$, of the gap then gives in view of~\eqref{s-root-sides} provided $\gm_{n}\neq 0$
\begin{equation}
  \label{In-gmn}
  \frac{8n\pi I_{n}}{\gm_{n}^{2}}
   = \frac{2}{\pi}\int_{-1}^{1} \frac{(t-t_{n})^{2}}{\sqrt[+]{1-t^{2}}}
   \chi_{n}(\lm_{t})\,\dt,
  \quad t_{n} = \frac{2(\lm_{n}^{\ld}-\tau_{n})}{\gm_{n}}.
\end{equation}
It follows from a limiting argument for $\gm_{n}\to 0$ that identity~\eqref{In-gmn} holds as well in the case $\gm_{n} = 0$. Then $t_{n} = 0$ and one has
\[
  \frac{8n\pi I_{n}}{\gm_{n}^{2}}
  =
  \chi_{n}(\tau_{n})
  \frac{2}{\pi}\int_{-1}^{1} \frac{t^{2}}{\sqrt[+]{1-t^{2}}}\,\dt.
\]
Since $\lm_{k}^{\ld}-\tau_{k} = \gm_{k}^{2}k^{-1}\ell_{k}^{p} = k^{-1-2s}\ell_{k}^{p/3}$ by item (i),  Proposition~\ref{std-prd-est} (ii) yields
\begin{equation}
  \label{chi-n}
  \left.\chi_{n}(\lm)\right|_{G_{n}}
   = 1 + n^{-1+(-1-2s)_{+}}(\ell_{n}^{p/2}+ \ell_{n}^{1+}).
\end{equation}
Moreover, $t_{n} = \gm_{n}n^{-1}\ell_{n}^{p} = o(1)$ so that
\begin{align*}
  \frac{2}{\pi}\int_{-1}^{1} \frac{(t-t_{n})^{2}}{\sqrt[+]{1-t^{2}}}\,\dt
  &=
  \frac{2}{\pi}\int_{-1}^{1} \frac{t^{2}}{\sqrt[+]{1-t^{2}}}\,\dt
  +
  t_{n}^{2}\frac{2}{\pi}\int_{-1}^{1} \frac{1}{\sqrt[+]{1-t^{2}}}\,\dt\\
  &= 1 + \gm_{n}^{2}n^{-2}\ell_{n}^{p/2} = 1 + n^{-2-2s}\ell_{n}^{p/4}.
\end{align*}
Therefore, $\frac{8n\pi I_{n}}{\gm_{n}^{2}} = 1 + n^{-1+(-1-2s)_{+}}(\ell_{n}^{p/2}+ \ell_{n}^{1+})$.
Going through the arguments of the proof, one sees that the estimate holds locally uniformly on $\Ws^{s,p}$.~\qed
\end{proof}

We also need to refine the estimate $\abs{F_{n}}_{G_{n}} = O(\gm_{n}/n)$ from Lemma~\ref{F-prop} (iii), where $F_{n}$ is defined in~\eqref{Fn-integral}. In view of~\eqref{w-closed} one has
\[
  F_{n}(\lm) = \int_{\lm_{n}^{-}}^{\lm} \frac{\dDl(z)}{\sqrt[c]{\Dl^{2}(z)-4}}\,\dz.
\]

\begin{lem}
\label{Fk-exp}
For any $(s,p)$ admissible with $-1\le s\le 0$
\[
  \sup_{\lm\in G_{n}^{+}\cup G_{n}^{-}} \abs*{F_{n}(\lm) - \frac{\ii \vs_{n}(\lm)}{2n\pi}}
   = \frac{1}{2n\pi}\p*{\gm_{n}n^{-s-1}(\ell_{n}^{p/2} + \ell_{n}^{1+})},
\]
locally uniformly on $\Ws^{s,p}$.~\fish
\end{lem}
\begin{proof}
If $\gm_{n} = 0$, then $G_{n} = \setd{\lm_{n}^{\pm}}$ and $F(\lm_{n}^{\m}) = 0 = \frac{\ii\vs_{n}(\lm_{n}^{\pm})}{2n\pi}$. Therefore, we only consider the case $\gm_{n}\neq 0$.
With~$\chi_{n}(\lm)$ given as in~\eqref{In-chin}, $F_{n}$ takes form
\[
  F_{n}(\lm) = \int_{\lm_{n}^{-}}^{\lm} \frac{\dDl(z)}{\sqrt[c]{\Dl^{2}(z)-4}}\,\dz
             = -\frac{\ii}{2n\pi}
                \int_{\lm_{n}^{-}}^{\lm} \frac{\lm_{n}^{\ld}-z}{\vs_{n}(z)} \chi_{n}(z) \,\dz.
\]
By~\eqref{chi-n}, $\abs{\chi_{n}(\lm)-1}_{G_{n}} =  n^{-1+(-1-2s)_{+}}(\ell_{n}^{p/2} + \ell_{n}^{1+})$ and in view of~\eqref{lmld-taun-selecta-ep}, $\abs{\lm_{n}^{\ld}-\lm}_{G_{n}} \le \abs{\gm_{n}}$ for all $n$ sufficiently large. Therefore, by Lemma~\ref{s-root-reciprocal-estimate}
\begin{equation}
\label{Fn-first}
\begin{split}
  \abs*{F_{n}(\lm) - 
  \frac{-\ii}{2n\pi}
  \int_{\lm_{n}^{-}}^{\lm} \frac{\lm_{n}^{\ld}-z}{\vs_{n}(z)} \,\dz
  }_{G_{n}^{+}\cup G_{n}^{-}}
  &\le \abs*{(\lm_{n}^{\ld}-\lm)(\chi_{n}(\lm)-1)}_{G_{n}}\\
  &= \gm_{n} n^{-1+(-1-2s)_{+}}(\ell_{n}^{p/2} + \ell_{n}^{1+}).
\end{split}
\end{equation}
One further checks that $\partial_{\lm} \vs_{n}(\lm) = -\frac{\tau_{n}-\lm}{\vs_{n}(\lm)}$, hence
\[
  \int_{\lm_{n}^{-}}^{\lm} \frac{\lm_{n}^{\ld}-z}{\vs_{n}(z)} \,\dz
   =
  -\vs_{n}(\lm)
  -(\tau_{n}-\lm_{n}^{\ld})\int_{\lm_{n}^{-}}^{\lm} \frac{1}{\vs_{n}(z)} \,\dz.
\]
Since by Lemma~\ref{In-lmd-est} (ii) we have $\tau_{n}-\lm_{n}^{\ld} = \gm_{n}^{2}n^{-1}\ell_{n}^{p}$, and $\gm_{n} = n^{-s}\ell_{n}^{p}$ by~\eqref{taun-lmn-est}, we get $\tau_{n}-\lm_{n}^{\ld} = n^{-s-1}\gm_{n}\ell_{n}^{p/2}$. Hence by Lemma~\ref{s-root-reciprocal-estimate}
\[
  \abs*{\int_{\lm_{n}^{-}}^{\lm} \frac{\lm_{n}^{\ld}-z}{\vs_{n}(z)} \,\dz + \vs_{n}(\lm)}%
  _{G_{n}^{+}\cup G_{n}^{-}}
  = n^{-s-1}\gm_{n}\ell_{n}^{p/2}.
\]
Inserting the latter estimate into~\eqref{Fn-first} and noting that for $-1\le s \le 0$ one has $-s\ge (-1-2s)_{+}$, the claimed estimate follows.
By going through the arguments of the proof, one sees that this estimate holds locally uniformly on $\Ws^{s,p}$.\qed
\end{proof}

\subsection{Refined estimates for the roots of the psi-functions}

For $q\in \Ws$ we denote by $\psi_{n}$, $n\ge1$, the entire function of the form
\begin{equation}
  \label{psi-form}
  \psi_{n}(\lm) \defl \frac{2}{n\pi} \prod_{m\neq n} \frac{\sg_{m}^{n}-\lm}{m^{2}\pi^{2}},
  \qquad
  \sg_{m}^{n} = m^{2}\pi^{2} + O(m),
\end{equation}
which is characterized by the property that
\begin{equation}
  \label{psi-int-eqn}
  \frac{1}{2\pi}\int_{\Gm_{k}} \frac{\psi_{n}(\lm)}{\sqrt[c]{\Dl^{2}(\lm)-4}}\,\dlm = \dl_{nk},
  \qquad n,k\ge 1.
\end{equation}
The roots of $\psi_{n}$ are precisely the complex numbers $\sg_{k}^{n}$, $k\neq n$, and can be shown to satisfy
\begin{equation}
  \label{psi-root-asymptotics}
  \sg_{k}^{n} - \tau_{k} = O\p*{\gm_{k}^{2}/k}
\end{equation}
uniformly in $n\ge 1$ and locally uniformly in $q\in \Ws$ -- see~\cite{Kappeler:2005fb,Kappeler:2003up}. It appears, in a sense to be made precise, that $\sg_{k}^{n}$ is closer to $\lm_{k}^{\ld}$ than to $\tau_{k}$. By Lemma~\ref{In-lmd-est}, one has
\begin{equation}
  \label{sg-k-lm-k}
    \sg_{k}^{n}-\lm_{k}^{\ld}
   = (\sg_{k}^{n}-\tau_{k}) + (\tau_{k}-\lm_{k}^{\ld})
   = k^{-1-s}\gm_{k}\ell_{k}^{p} = k^{-1-2s}\ell_{k}^{p/2},
   \qquad k\neq n.
\end{equation}
The purpose of this subsection is to improve on these estimates.

\begin{prop}
\label{prop:sg-lm}
For any $(s,p)$ admissible, and any $k\neq n$
\[
  \sg_{k}^{n}-\lm_{k}^{\ld} = \begin{cases}
  \gm_{k}\ell_{k}^{2}, & s = -1,\quad p = 2\\
  n^{-(1-\rho)s}k^{-1-s\rho}\gm_{k}\ell_{k}^{p}, & -1 < s \le 0,\quad 2\le p < \infty,\quad 0\le \rho\le 1,
  \end{cases}
\]
locally uniformly on $\Ws^{s,p}$.\fish
\end{prop}

\begin{proof}
If $s=-1$, $p=2$, then the claimed estimate $\sg_{k}^{n}-\lm_{k}^{\ld} = \gm_{k}\ell_{k}^{2}$ follows from~\eqref{sg-k-lm-k}. Thus it remains to consider the case $s > -1$.
Note that in view of~\eqref{psi-form} and~\eqref{c-root} for any $n,k\ge 1$
\begin{equation}
  \label{psi-c-root-quot}
  \frac{\psi_{n}(\lm)}{\sqrt[c]{\Dl^{2}(\lm)-4}}
   =
  \frac{n}{k}\frac{\ii}{\vs_{k}(\lm)}\frac{\sg_{k}^{n}-\lm}{\sg_{n}^{n}-\lm}\zt_{k}(\lm),\qquad
  \zt_{k}(\lm)
   =
  \frac{k\pi}{\sqrt[+]{\lm-\lm_{0}^{+}}}\prod_{m\neq k}\frac{\sg_{m}^{n}-\lm}{\vs_{m}(\lm)},
\end{equation}
where the function $\zt_{k}$ is analytic on $U_{k}$ and we set $\sg_{n}^{n} \defl \lm_{n}^{\ld}$.
By~\eqref{psi-int-eqn}, the roots $\sg_{k}^{n}$, $k\neq n$, of $\psi_{n}$ are characterized by the equation
\begin{equation}
  \label{psi-1}
  0
   = \int_{\Gm_{k}}\frac{\psi_{n}(\lm)}{\sqrt[c]{\Dl^{2}(\lm)-4}}\,\dlm
   = \ii\frac{n}{k}\int_{\Gm_{k}}\frac{\sg_{k}^{n}-\lm}{\vs_{k}(\lm)}\frac{\zt_{k}(\lm)}{\sg_{n}^{n}-\lm}\,\dlm,\qquad
   k\neq n.
\end{equation}
It implies that
\begin{equation}
  \label{zt-k-identity}
    (\sg_{k}^{n}-\lm_{k}^{\ld})
  \int_{\Gm_{k}}\frac{1}{\vs_{k}(\lm)}\frac{\zt_{k}(\lm)}{\sg_{n}^{n}-\lm}\,\dlm
  =
  \int_{\Gm_{k}}\frac{\lm - \lm_{k}^{\ld}}{\vs_{k}(\lm)}\frac{\zt_{k}(\lm)}{\sg_{n}^{n}-\lm}\,\dlm,
  \qquad k\neq n.
\end{equation}
This identity is the starting point for estimating $\sg_{k}^{n}-\lm_{k}^{\ld}$.
A key step in the proof of the claimed estimate is to rewrite this identity in an appropriate way.
Let us multiply the identity by $(\sg_{n}^{n}-\lm_{k}^{\ld})$ and introduce
\[
  \xi_{k}(\lm) = \frac{\sg_{n}^{n}-\lm_{k}^{\ld}}{\sg_{n}^{n}-\lm}\zt_{k}(\lm)
  = \p*{1 + \frac{\lm-\lm_{k}^{\ld}}{\sg_{n}^{n}-\lm}}\zt_{k}(\lm).
\]
It then follows from~\eqref{zt-k-identity} that
\[
  (\sg_{k}^{n}-\lm_{k}^{\ld})
  \int_{\Gm_{k}}\frac{\xi_{k}(\lm)}{\vs_{k}(\lm)}\,\dlm
  =
  \int_{\Gm_{k}}\frac{(\lm - \lm_{k}^{\ld})\xi_{k}(\lm)}{\vs_{k}(\lm)}\,\dlm,
  \qquad k\neq n,
\]
where
\[
  \int_{\Gm_{k}}\frac{(\lm - \lm_{k}^{\ld})\xi_{k}(\lm)}{\vs_{k}(\lm)}\,\dlm
  =
  \int_{\Gm_{k}}\frac{(\lm-\lm_{k}^{\ld})\zt_{k}(\lm)}{\vs_{k}(\lm)}\,\dlm
  +
  \int_{\Gm_{k}}\frac{(\lm-\lm_{k}^{\ld})^{2}\zt_{k}(\lm)}{(\sg_{n}^{n}-\lm)\vs_{k}(\lm)}\,\dlm.
\]
The second term on the right hand side is expected to be small in comparison to the first term since $\sg_{n}^{n}-\lm$ is of the size of $n^{2}-k^{2}$. We proceed by writing the first term in a more convenient form. Note that the roots $\lm_{k}^{\ld}$, $k\ge 1$, of $\dDl$ are characterized by the equation
\begin{equation}
  \label{dDl-1}
  0 = \int_{\Gm_{k}} \frac{\dDl(\lm)}{\sqrt[c]{\Dl^{2}(\lm)-4}}\,\dlm
  =
  \frac{-\ii}{2k\pi}
  \int_{\Gm_{k}} 
  \frac{\lm_{k}^{\ld}-\lm}{\vs_{k}(\lm)}\chi_{k}(\lm)\,\dlm,\qquad k\ge 1,
\end{equation}
where $\chi_{k}$ is given by~\eqref{In-chin}. Hence
\[
  \int_{\Gm_{k}}\frac{(\lm-\lm_{k}^{\ld})\zt_{k}(\lm)}{\vs_{k}(\lm)}\,\dlm
  =
  \int_{\Gm_{k}}\frac{(\lm-\lm_{k}^{\ld})(\zt_{k}(\lm)-\chi_{k}(\lm))}{\vs_{k}(\lm)}\,\dlm.
\]
Altogether, identity~\eqref{zt-k-identity} then reads
\begin{equation}
  \label{zt-k-identity-2}
  \begin{split}
  (\sg_{k}^{n}-\lm_{k}^{\ld})
  \int_{\Gm_{k}}\frac{\xi_{k}(\lm)}{\vs_{k}(\lm)}\,\dlm
  &=
  \int_{\Gm_{k}}\frac{(\lm-\lm_{k}^{\ld})(\zt_{k}(\lm)-\chi_{k}(\lm))}{\vs_{k}(\lm)}\,\dlm\\
  &\qquad+
  \int_{\Gm_{k}}\frac{(\lm-\lm_{k}^{\ld})^{2}\zt_{k}(\lm)}{(\sg_{n}^{n}-\lm)\vs_{k}(\lm)}\,\dlm,\qquad k\neq n.
    \end{split}
\end{equation}
The integrals in~\eqref{zt-k-identity-2} are now estimated separately.
First note that for $\lm\in G_{k}$ we have
\[
  \frac{\lm-\lm_{k}^{\ld}}{\sg_{n}^{n}-\lm} = O\p*{\frac{\gm_{k}}{n^{2}-k^{2}}}
  = k^{-1-s}\ell_{k}^{p}.
\]
Since by Remark~\ref{std-prd-simple-est} we have $\zt_{k}\big|_{G_{k}} = 1 + k^{-1-s}\ell_{k}^{p}$, we find
\[
  \xi_{k}\big|_{G_{k}} = (1+k^{-1-s}\ell_{k}^{p})\zt_{k}\big|_{G_{k}} = 1+k^{-1-s}\ell_{k}^{p}.
\]
Since $\frac{1}{2\pi\ii}\int_{\Gm_{k}} \frac{1}{\vs_{k}(\lm)}\,\dlm = -1$, we then conclude
\begin{equation}
  \label{xi-k-est-1}
  \frac{1}{2\pi\ii}\int_{\Gm_{k}}\frac{\xi_{k}(\lm)}{\vs_{k}(\lm)}\,\dlm = -1 + k^{-1-s}\ell_{k}^{p}.
\end{equation}
Concerning the integral $\int_{\Gm_{k}}\frac{(\lm-\lm_{k}^{\ld})(\zt_{k}(\lm)-\chi_{k}(\lm))}{\vs_{k}(\lm)}\,\dlm$, it follows from Lemma~\ref{int-wm-quot-est} that
\begin{equation}
  \label{xi-k-est-2}
  \abs*{\frac{1}{2\pi}\int_{\Gm_{k}}\frac{(\lm-\lm_{k}^{\ld})(\zt_{k}(\lm)-\chi_{k}(\lm))}{\vs_{k}(\lm)}\,\dlm}
  \le
  \abs{\gm_{k}}\abs{\zt_{k}(\lm)-\chi_{k}(\lm)}_{G_{k}}.
\end{equation}
Similarly, for the integral $\int_{\Gm_{k}}\frac{(\lm-\lm_{k}^{\ld})^{2}\zt_{k}(\lm)}{(\sg_{n}^{n}-\lm)\vs_{k}(\lm)}\,\dlm$, we get
\begin{equation}
  \label{xi-k-est-3}
  \abs*{\frac{1}{2\pi}\int_{\Gm_{k}}\frac{(\lm-\lm_{k}^{\ld})^{2}\zt_{k}(\lm)}{(\sg_{n}^{n}-\lm)\vs_{k}(\lm)}\,\dlm}
  \le
  \abs*{\frac{(\lm-\lm_{k}^{\ld})^{2}}{\sg_{n}^{n}-\lm}\zt_{k}(\lm)}_{G_{k}}
  = O\p*{\frac{\gm_{k}^{2}}{n^{2}-k^{2}}}.
\end{equation}

Using that for $1\le k< n/2$, $n/2\le k\le 3n/2$, and $3n/2 < k$ one has $\frac{1}{n^{2}-k^{2}} = O(1/n^{2})$, $\frac{1}{n^{2}-k^{2}} = O(1/(nk))$, and $\frac{1}{n^{2}-k^{2}} = O(1/k^{2})$, respectively, one concludes that for any $-3/2\le \al \le 3/2$ we have
\begin{equation}
  \label{hill-est}
  \frac{1}{n^{2}-k^{2}}
  =
  O(n^{-1/2+\al}k^{-1/2-\al}),\qquad n\neq k,
\end{equation}
where the implicit constant can be chosen uniformly in $\al$ and $n,k\ge 1$.
Therefore, if $-1\le s \le 0$ and $0\le \rho\le 1$, we may choose $\al = 1/2-s(1-\rho)$ to obtain
\begin{equation}
  \label{gm-quot-est}
  O\p*{\frac{\gm_{k}^{2}}{n^{2}-k^{2}}}
  =
  n^{-s(1-\rho)}k^{-1-s\rho}\gm_{k}\ell_{k}^{p}.
\end{equation}
Inserting estimates~\eqref{xi-k-est-1}-\eqref{gm-quot-est} into~\eqref{zt-k-identity-2} yields
\begin{equation}
  \label{system-1}
  (\sg_{k}^{n}-\lm_{k}^{\ld})
  (1+k^{-1-s}\ell_{k}^{p})
  =
  O\p*{\gm_{k}\abs{\zt_{k}(\lm)-\chi_{k}(\lm)}_{G_{k}}}
   + n^{-s(1-\rho)}k^{-1-s\rho}\gm_{k}\ell_{k}^{p}.
\end{equation}

To estimate $\abs{\zt_{k}(\lm)-\chi_{k}(\lm)}_{G_{k}}$, write the product expansions~\eqref{psi-c-root-quot} and~\eqref{In-chin}, respectively, as $\zt_{k}(\lm) = f_{k}(\lm,\al^{1})$ and  $\chi_{k}(\lm) = f_{k}(\lm,\al^{0})$, where we have set $\al^{1} = (\sg_{m}^{n})$, $\al^{0} = (\lm_{m}^{\ld})$, and
\[
  f_{k}(\lm,\al) = \frac{k\pi}{\sqrt{\lm-\lm_{0}^{+}}}\prod_{m\neq k}\frac{\al_{m}-\lm}{\vs_{m}(\lm)}.
\]
By~\eqref{phi-n-func}, the function~$f_{k}$ is analytic on $(\C\setminus\bigcup_{m\neq k} G_{k})\times\ell_{\C}^{p}$ and by Remark~\ref{std-prd-simple-est} satisfies the estimate $\abs{f_{k}(\lm,\al)-1}_{G_{k}} = 1+k^{-1-s}\ell_{k}^{p}$ locally uniformly on $\ell_{\C}^{p}$. Thus we may write for any $\lm\in G_{k}$,
\[
  \zt_{k}(\lm)-\chi_{k}(\lm)
  = f_{k}(\lm,\al^{1})-f_{k}(\lm,\al^{0})
  = \int_{0}^{1} \sum_{m\neq k} \partial_{\al_{m}} f_{k}(\lm,\al^{t}) (\sg_{m}^{n}-\lm_{m}^{\ld})\,\dt,
\]
where $\al^{t} \defl (\al_{m}^{t}) = ((1-t) \sg_{m}^{n} + t \lm_{m}^{\ld})$. Since for any $m\neq k$ one has
\[
  \partial_{\al_{m}} f_{k}(\lm,\al) = \frac{1}{\al_{m}-\lm}f_{k}(\lm,\al),
\]
we conclude
\[
  \zt_{k}(\lm)-\chi_{k}(\lm)
  =
  \int_{0}^{1} f_{k}(\lm,\al^{t})\sum_{m\neq k} 
  \frac{\sg_{m}^{n}-\lm_{m}^{\ld}}{\al_{m}^{t}-\lm} \,\dt.
\]
By Remark~\ref{std-prd-simple-est}, we can choose $M > 0$ so that $\sup_{0\le t\le 1}\abs{f_{k}(\lm,\al_{t})}_{G_{k}} \le M$ for all $k\ge1$. Moreover, by the mean value theorem there exists a sequence $(\nu_{k})\subset \C$ with $\nu_{k}\in G_{k}$ such that
\begin{equation}
  \label{zt-k-chi-k-est}
  \abs{\zt_{k}(\lm)-\chi_{k}(\lm)}_{G_{k}} \le
  M\int_{0}^{1}\abs*{\sum_{m\neq k} 
  \frac{\sg_{m}^{n}-\lm_{m}^{\ld}}{\al_{m}^{t}-\nu_{k}}}\,\dt.
\end{equation}
Since $U_{m}$ is assumed to be convex, we have $\al_{m}^{t}\in U_{m}$ for all $m\ge 1$. Moreover, $\nu_{k}\in G_{k}\subset U_{k}$ for all $k\ge 1$. Thus by~\eqref{iso-est}, there exists $c > 0$ so that
\[
  \abs{\al_{m}^{t}-\nu_{k}} \ge c\abs{m^{2}-k^{2}},\qquad m\neq k,\quad 0\le t\le 1.
\]

If $\rho=1$, the claimed estimate is the one of~\eqref{sg-k-lm-k}. In the case $0 \le \rho < 1$ we argue by iteration using~\eqref{system-1}. As a starting point, write the estimate~\eqref{sg-k-lm-k} in the form $\sg_{m}^{n}-\lm_{m}^{\ld} = n^{\bt_{1}}m^{-t_{1}}\gm_{m}\ell_{m}^{p}$ with $\bt_{1} = 0$ and $t_{1} = 1+s\in (0,1]$ and suppose that for some $j\ge 1$
\[
  \sg_{m}^{n}-\lm_{m}^{\ld} = n^{\bt_{j}}m^{-t_{j}}\gm_{m}\ell_{m}^{p},
  \qquad \bt_{j}\ge 0,\qquad t_{j}\in (0,1].
\]
It follows with Lemma~\ref{mht-2} from~\eqref{zt-k-chi-k-est} that
\[
  \abs{\zt_{k}(\lm)-\chi_{k}(\lm)}_{G_{k}} = n^{\bt_{j}}k^{-1-\min(0,t_{j}+s)}(\ell_{k}^{p/2}+\ell_{k}^{1+}).
\]
We conclude with~\eqref{system-1} that for all $k$ sufficiently large
\begin{equation}
  \label{sg-k-lm-k-dot-2}
  \begin{split}
  \sg_{k}^{n}-\lm_{k}^{\ld}
  &=
  n^{\bt_{j}}k^{-1-\min(0,t_{j}+s)}\gm_{k}(\ell_{k}^{p/2}+\ell_{k}^{1+})
  +
  n^{-s(1-\rho)}k^{-1-s\rho}\gm_{k}\ell_{k}^{p}\\
  &= n^{\bt_{j+1}}k^{-t_{j+1}}\gm_{k}\ell_{k}^{p},
  \end{split}
\end{equation}
where
\[
  \bt_{j+1} \defl \max(\bt_{j},-s(1-\rho))\ge 0,\qquad
  t_{j+1} \defl \min(1+s\rho,1+t_{j}+s) \in (0,1].
\]
We may thus iterate the estimate.
Since $\bt_{1} = 0$, we conclude $\bt_{j} = -s(1-\rho)\ge 0$ for all $j\ge 2$. On the other hand, since $t_{1} = 1+s$, we conclude $t_{j} = \min(1+s\rho,j(1+s))$ for all $j\ge 2$.
After finitely many iterations of this estimate we have $1+s\rho < j(1+s)$ and hence
\[
  \sg_{k}^{n}-\lm_{k}^{\ld} = 
  n^{-(1-\rho)s}k^{-1-s\rho}\gm_{k}\ell_{k}^{p}.
\]
By going through the arguments of the proof, one verifies that the estimate holds locally uniformly on $\Ws^{s,p}$.\qed
\end{proof}

\section{KdV}
\label{s:kdv1}

\subsection{Frequencies}
\label{ss:kdv-freq}

In this section we derive a novel formula for the KdV frequencies $\om_{n}^{(1)}$, $n\ge~1$, which we then use to study their asymptotics for $n\to \infty$. The frequencies can be viewed either as analytic functionals of the potential $q$ on $\Ws$ or as analytic functionals of the actions $I = (I_{m})_{m\ge1}$ on a neighborhood $\Vs$ of $\ell_{+}^{-1,1}$ within $\ell_{\C}^{-1,1}$. Which case is at hand should be always clear from the context, hence we do not introduce different notations for them.  Our starting point for deriving the novel formula is the following identity for the $n$th KdV frequency
\[
  \om_{n}^{(1)} = \pbr{\Hm_{1},\th_{n}},
\]
which a priori holds on $\Hs_{0,\C}^{1}\cap (\Ws\setminus Z_{n})$, where $Z_{n}\defl\setdef{q\in\Ws}{\gm_{n}^{2}(q)=0}$ is an analytic subvariety of $\Ws$.
Recall that $\gm_{n}^{2}$ is analytic on $\Ws$ whereas $\gm_{n}$ is not.

It turns out to be convenient to introduce for any integers $n,k\ge 1$ and $m\ge 0$ the moments
\[
  \Om_{nk}^{(m)}
  \defl \int_{\Gm_{k}} \frac{(F_{k}(\lm))^{m}\psi_{n}(\lm)}{\sqrt[c]{\Dl^{2}(\lm)-4}}\,\dlm.
\]

\begin{lem}
\label{Om-nk-prop}
\begin{equivenum}
\item
$\Om_{nk}^{(0)} = 2\pi\dl_{nk}$, for all $n,k\ge 1$.

\item
Each moment $\Om_{nk}^{(m)}$, $n,k\ge 1$, $m\ge 1$, is analytic on $\Ws$.

\item
$\Om_{nk}^{(2l+1)} = 0$, for all $n,k\ge 1$ and $l\ge 0$.

\item
$\Om_{nk}^{(m)} = 0$, for all $n, k\ge 1$, $m\ge 1$, if $\gm_{k} = 0$.~\fish
\end{equivenum}
\end{lem}

\begin{proof}
(i) follows immediately from the characterization~\eqref{psi-int-eqn} of the functions $\psi_{n}$.

(ii)
For any $q\in\Ws$ we have a set of isolating neighborhoods $(U_{m})_{m\ge 1}$ which work for a whole neighborhood $\Ws_{q}\subset\Ws$ of $q$. Moreover, by the locally uniform asymptotic behavior of the periodic and Dirichlet eigenvalues, we can choose a set of counterclockwise oriented circuits $\Gm_{m}$, $m\ge 1$, and a set of open neighborhoods $U_{m}'$ of $\Gm_{m}$ so that $\Gm_{m}\subset U_{m}$ circles around $G_{m}$ and $\ob{U_{m}'}\subset U_{m}\setminus G_{m}$ for any potential in $\Ws_{q}$.
By the properties of the function $F_{k}(\lm)$ (cf Lemma~\ref{F-prop}),
$\sqrt[c]{\Dl^{2}(\lm)-4}$ (see~\eqref{c-root} and the discussion following it), and $\psi_{n}(\lm)$ (see~\eqref{psi-form} above), for each $q\in\Ws$, the integrand $\frac{F_{k}^{m}(\lm)\psi_{n}(\lm)}{\sqrt[c]{\Dl^{2}(\lm)-4}}$ is analytic on $(\bigcup_{n\ge 1} U_{n}')\times \Ws_{q}$ hence $\Om_{nk}^{(m)}$ is analytic on $\Ws_{q}$ as well.

(iii) For any $k,l\ge 0$, the function $(F_{k}(\lm))^{2l+1}/\sqrt[c]{\Dl^{2}(\lm)-4}$ is analytic on $U_{k}$ by Lemma~\ref{F2-analytic}. Therefore, $\Om_{nk}^{(2l+1)} = 0$ for all $n,k\ge 1$.

(iv) In view of item (iii) it remains to consider the case where $m=2l$ with $l\ge 1$ and $\gm_{k} = 0$. We first consider the case $n\neq k$. It follows from~\eqref{psi-root-asymptotics} that $\sg_{k}^{n} = \tau_{k}$ for any $n\ge 1$, so that by the product representations~\eqref{c-root} and \eqref{psi-form} the quotient $\psi_{n}(\lm)/\sqrt[c]{\Dl^{2}(\lm)-4}$ is analytic on $U_{k}$.
Moreover, $(F_{k}(\lm))^{2l}$ is analytic on $U_{k}$ by Lemma~\ref{F2-analytic}, which proves the claim for $n\neq k$.
Now suppose $n=k$, then by~\eqref{psi-c-root-quot}
\[
  \frac{\ps_{n}(\lm)}{\sqrt[c]{\Dl^{2}(\lm)-4}} = \frac{\ii}{\vs_{n}(\lm)}\zt_{n}(\lm),
\]
where $\vs_{n}(\lm) = (\tau_{n}-\lm)$ if $\gm_{n} = 0$ by~\eqref{s-root} and $\zt_{n}$ is analytic on $U_{n}$. Since $F_{n}^{2l}$ is analytic on $U_{n}$ by Lemma~\ref{F2-analytic} (i), we conclude with Cauchy's theorem that
\[
  \Om_{nn}^{(2l)} = 2\pi \ii F_{n}(\tau_{n})^{2l}\zt_{n}(\tau_{n}) = 0.
\]
Here we used that $F_{n}(\tau_{n}) = F(\lm_{n}^{\pm}) = 0$ by Lemma~\ref{F-prop} (ii) if $\gm_{n} = 0$.~\qed
\end{proof}

\begin{lem}
\label{om-n-kRnk}
For any real-valued finite-gap potential cf.~\eqref{fin-gap} with $[q] = 0$ and any $n\ge 1$
\begin{equation}
  \label{om-star-Om-nk}
  \om_{n}^{(1)\star} \defl \om_{n}^{(1)} - (2n\pi)^{3}
   = -12\sum_{k\ge1} k \Om_{nk}^{(2)}.~\fish
\end{equation}
\end{lem}

\begin{proof}
Let $q$ be a finite-gap potential, meaning that $S = \setdef{k\in\N}{\gm_{k}(q)\neq 0}$ is finite.
By Lemma~\ref{F2-analytic}, the function $F^{2}$ is analytic outside a sufficiently large circle $C_{r}$ which encloses all open gaps $G_{k}$, $k\in S$, and whose exterior contains $G_{0}$.
Furthermore, $F$ admits according to \eqref{F-exp} an asymptotic expansion for $\nu_{k} = (k+1/2)\pi$. In particular,
\[
  F(\lm)^{4} = \lm^{2} - \Hm_{0} - \frac{\Hm_{1}}{4}\frac{1}{\lm} + O(\lm^{-2}),
\]
so that by Cauchy's Theorem
\[
  -\frac{\Hm_{1}}{4} = \frac{1}{\ii 2\pi}\int_{C_{r}} F^{4}(\lm)\,\dlm.
\]
Let $n\in S$ then $\gm_{n}(q)\neq 0$ and $\th_{n}\mod \pi$ is analytic near $q$ so that
\begin{align*}
  \om_{n}^{(1)} = \pbr{\Hm_{1},\th_{n}}
  &=
  \frac{2}{\ii\pi}\int_{C_{r}} \!\!\!\pbr{\th_{n},F^{4}(\lm)}\,\dlm
  =
  \frac{8}{\ii\pi}\int_{C_{r}}\!\!\! F^{3}(\lm)\pbr{\th_{n},F(\lm)}\,\dlm.
\end{align*}
Since $\pbr{\th_{n},F(\lm)} = \frac{\pbr{\th_{n},\Dl(\lm)}}{\sqrt[c]{\Dl^{2}(4)-4}}$ by Lemma~\ref{F-prop} (i) and $2\{\th_{n},\Dl(\lm)\} = \ps_{n}(\lm)$ by \cite[Proposition F.3]{Kappeler:2003up}, one obtains
\[
  \om_{n}^{(1)}
  =
  \frac{4}{\ii\pi}\int_{C_{r}} \frac{F^{3}(\lm)\psi_{n}(\lm)}{\sqrt[c]{\Dl^{2}(\lm)-4}}\,\dlm.
\]
By Lemma~\ref{F2-analytic} (ii) and formula~\eqref{psi-form}, the integrand is analytic on $U_{0}$, while for any $k\in\N\setminus S$, one has $\sg_{k}^{n} = \tau_{k}$ and $\vs_{k}(\lm) = \tau_{k}-\lm$ so that in view of the product representations~\eqref{c-root} and~\eqref{psi-form} the integrand extends analytically to $U_{k}$. Consequently, the integrand is analytic on $\C\setminus\bigcup_{k\in S} G_{k}$ and one obtains by contour deformation
\[
  \om_{n}^{(1)}
   = \frac{4}{\ii\pi} \sum_{k\in S}
     \int_{\Gm_{k}} \frac{F^{3}(\lm)\psi_{n}(\lm)}{\sqrt[c]{\Dl^{2}(\lm)-4}}\,\dlm.
\]
Proceeding by expanding $F(\lm)^{3} = (F_{k}(\lm)-\ii k\pi)^{3}$ gives
\[
  F^{3}(\lm) = F_{k}^{3}(\lm) - 3\ii (k\pi) F_{k}^{2}(\lm) - 3(k\pi)^{2} F_{k}(\lm) + \ii (k\pi)^{3},
\]
and since $\Om_{nk}^{(3)} \equiv \Om_{nk}^{(1)} \equiv 0$ by Lemma~\ref{Om-nk-prop} (iii) and $\Om_{nk}^{(0)} = 2\pi\dl_{nk}$ by Lemma~\ref{Om-nk-prop} (i), we thus get
\begin{align*}
  \om_{n}^{(1)}
   &= \frac{4}{\ii\pi}
   \sum_{k\in S}\p*{-3\ii (k\pi) \Om_{nk}^{(2)} + \ii (k\pi)^{3}\Om_{nk}^{(0)}}\\
   &= \sum_{k\ge 1} \p*{-12k\Om_{nk}^{(2)} + (2k\pi)^{3}\dl_{kn}},
\end{align*}
where in the second line we used that $\Om_{nk}^{(2)} = 0$ for all $k\in \N\setminus S$ by Lemma~\ref{Om-nk-prop} (iv).
This shows that~\eqref{om-star-Om-nk} holds for all $n\in S$.

Now consider any $n\in\N\setminus S$, that is $\gm_{n}(q)=0$. We can choose a sequence of real-valued finite-gap potentials $q_{l}$ with $\gm_{k}(q_{l}) = \gm_{k}(q)$ for $k\neq n$, $\gm_{n}(q_{l})\neq 0$, and $q_{l}\to q$ in $\Hs_{0}^{1}$. In particular, $S^{(l)} \equiv S(q_{l}) \defl \setdef{j\in\N}{\gm_{j}(q_{l})\neq 0}$ is given by $S \cup \setd{n}$ for any $l\ge 1$. Since each $\Om_{nk}^{(2)}$, $k\ge 1$, is continuous, indeed analytic on $\Ws$ by Lemma~\ref{Om-nk-prop} (ii), it follows that
\[
  \sum_{k\ge 1} k\Om_{nk}^{(2)}(q_{l})
   =
  \sum_{k\in S\cup\setd{n}} k\Om_{nk}^{(2)}(q_{l})
   \to
  \sum_{k\in S\cup\setd{n}} k\Om_{nk}^{(2)}(q)
   =
  \sum_{k\ge 1} k\Om_{nk}^{(2)}(q).
\]
Since $\om_{n}^{(1)\star}(q) = \lim\limits_{l\to \infty} \om_{n}^{(1)\star}(q_{l})$ and $\om_{n}^{(1)\star}(q_{l}) = -12\sum_{k\in S\cup\setd{n}} k\Om_{nk}^{(2)}(q_{l})$ for any $l\ge 1$, one concludes that $\om_{n}^{(1)\star}(q) = -12\sum_{k\ge 1} k \Om_{nk}^{(2)}(q)$.\qed
\end{proof}

We proceed by deriving decay estimates for $\Om_{nk}^{(2)}$. 

\begin{lem}
\label{decay-kRnk}
For any $n\ge 1$ and any $q\in\Ws^{s,p}$ with $(s,p)$ admissible and $-1\le s\le 0$ for $k\neq n$,
\[
  k\Om_{nk}^{(2)}
  = 
  \begin{cases}
  \frac{n}{n^{2}-k^{2}}k^{-2}\gm_{k}^{3}\ell_{k}^{2}, & s = -1,\quad p = 2,\\
  \frac{n^{1-(1-\rho)s}}{n^{2}-k^{2}}k^{-3-\rho s}\gm_{k}^{3}\ell_{k}^{p}, & -1 < s \le 0,\quad 2\le p  <\infty,\quad 0\le\rho\le 1,
  \end{cases}
\]
and
\[
  n\Om_{nn}^{(2)}
  = \frac{\gm_{n}^{2}}{16n\pi}\p*{1 + n^{-s-1}\ell_{n}^{p}},
\]
locally uniformly on $\Ws^{s,p}$.~\fish
\end{lem}

\begin{proof}
We begin with the case $k\neq n$. Our goal is to obtain a representation of $\Om_{nk}^{(2)}$ involving a difference of the quotients
\begin{align}
\label{prod-exp-1}
\begin{split}
  \frac{(2n\pi)\dDl(\lm)}{\sqrt[c]{\Dl^{2}(\lm)-4}}
  &=
  -\frac{\ii n\pi}{\sqrt[+]{\lm-\lm_{0}^{+}}}\prod_{m\ge1} \frac{\lm_{m}^{\ld}-\lm}{\vs_{m}(\lm)},\\
  \frac{(\sg_{n}^{n}-\lm)\psi_{n}(\lm)}{\sqrt[c]{\Dl^{2}(\lm)-4}}
  &=
  \frac{\ii n\pi}{\sqrt[+]{\lm-\lm_{0}^{+}}}
  \prod_{m\ge1} \frac{\sg_{m}^{n}-\lm}{\vs_{m}(\lm)},
\end{split}
\end{align}
which are obtained from~\eqref{Dl-quot-prod} and~\eqref{psi-c-root-quot} -- recall that $\sg_{n}^{n} = \lm_{n}^{\ld}$.
This allows to reduce the estimate of $\Om_{nk}^{(2)}$ to an estimate of $\sg_{k}^{n}-\lm_{k}^{\ld}$ where we can apply Proposition~\ref{prop:sg-lm}.
We first note that
\begin{align*}
  &(\sg_{n}^{n}-\tau_{k})\Om_{nk}^{(2)}
  = A_{nk} + B_{nk},\\
  &A_{nk} \defl \int_{\Gm_{k}} \frac{F_{k}^{2}(\lm) (\sg_{n}^{n}-\lm)\psi_{n}(\lm) }{\sqrt[c]{\Dl^{2}(\lm)-4}}\,\dlm,\qquad
  B_{nk} \defl
  \int_{\Gm_{k}} \frac{F_{k}^{2}(\lm) (\lm-\tau_{k})\psi_{n}(\lm) }{\sqrt[c]{\Dl^{2}(\lm)-4}}\,\dlm,
\end{align*}
and proceed by estimating the second term which is expected to be small since $\lm$ is close to $\tau_{k}$. In view of~\eqref{psi-c-root-quot} we may write for $k\neq n$,
\begin{equation}
  \label{zt-k-n}
    \frac{\psi_{n}(\lm)}{\sqrt[c]{\Dl^{2}(\lm)-4}} = \frac{\sg_{k}^{n}-\lm}{\vs_{k}(\lm)}\zt_{k}^{n}(\lm),\qquad
  \zt_{k}^{n}(\lm) = \ii\frac{n}{k}\frac{1}{\sg_{n}^{n}-\lm}\zt_{k}(\lm),
\end{equation}
where $\zt_{k}$ is given by~\eqref{psi-c-root-quot}.
By Remark~\ref{std-prd-simple-est}, $\zt_{k}(\lm)\big|_{G_{k}} = 1 + k^{-s-1}\ell_{k}^{p}$ while on the other hand, uniformly for $\lm\in G_{k}$,
\begin{align*}
  \frac{\sg_{n}^{n}-\tau_{k}}{\sg_{n}^{n}-\lm}-1
  =
  \frac{\lm-\tau_{k}}{\sg_{n}^{n}-\lm}
  =
  O\p*{\frac{k^{-s}\ell_{k}^{p}}{n^{2}-k^{2}}}
  = k^{-1-s}\ell_{k}^{p}.
\end{align*}
Both estimates together yield
\begin{equation}
  \label{zt-k-n-est}
  \zt_{k}^{n}(\lm)\bigg|_{G_{k}} = \ii\frac{n}{k}\frac{1}{\sg_{n}^{n}-\tau_{k}}\p*{1 + k^{-1-s}\ell_{k}^{p}}.
\end{equation}
By Lemma~\ref{F-prop} $\abs{F_{k}^{2}}_{G_{k}} = O(\gm_{k}^{2}/k^{2})$, and by~\eqref{taun-lmn-est} and \eqref{psi-root-asymptotics} we have $\abs{\sg_{k}^{n}-\lm}_{G_{k}} = O(\gm_{k})$, while by~\eqref{iso-est} there exists $c > 0$ so that $\abs{\sg_{n}^{n}-\tau_{k}} \ge c\abs{n^{2}-k^{2}}$. Thus it follows with Lemma~\ref{int-wm-quot-est} that
\begin{align*}
  B_{nk} 
  &=
  \int_{\Gm_{k}} \frac{F_{k}^{2}(\lm)(\lm-\tau_{k})(\sg_{k}^{n}-\lm)\zt_{k}^{n}(\lm)}{\vs_{k}(\lm)}\,\dlm\\
  &= O\p*{\frac{n}{k^{3}}\frac{\gm_{k}^{4}}{n^{2}-k^{2}}}
   = \frac{n}{n^{2}-k^{2}}k^{-3-s}\gm_{k}^{3}\ell_{k}^{p}.
\end{align*}
By~\eqref{hill-est} one then gets for any $-3/2\le \al \le 3/2$,
\[
  B_{nk}
   = n^{1/2+\al}k^{-7/2-s-\al}\gm_{k}^{3}\ell_{k}^{p}.
\]

We proceed with estimating $A_{nk}$.
By Lemma~\ref{F-prop} the function $F_{k}^{3}(\lm)$ is analytic on $U_{k}\setminus G_{k}$ and we compute
\[
  \partial_{\lm}\p*{ \frac{1}{3} F_{k}^{3}(\lm) }
   = \frac{F_{k}^{2}(\lm)\dDl(\lm)}{\sqrt[c]{\Dl^{2}(\lm)-4}}.
\]
Therefore,
\[
  \int_{\Gm_{k}} \frac{F_{k}^{2}(\lm)\dDl(\lm)}{\sqrt[c]{\Dl^{2}(\lm)-4}}\,\dlm = 0.
\]
Thus $A_{nk}$ may be written in the form
\begin{align*}
  A_{nk}
  &=
  \int_{\Gm_{k}} F_{k}^{2}(\lm)\frac{(\sg_{n}^{n}-\lm)\psi_{n}(\lm) + (2n\pi)\dDl(\lm)}{\sqrt[c]{\Dl^{2}(\lm)-4}}\,\dlm.
\end{align*}
Note that in view of~\eqref{prod-exp-1}
\[
  \frac{k\pi}{\ii n\pi}\frac{(\sg_{n}^{n}-\lm)\psi_{n}(\lm) + (2n\pi)\dDl(\lm)}{\sqrt[c]{\Dl^{2}(\lm)-4}}
  =
  f(\lm,\al^{1}) - f(\lm,\al^{0}),
\]
where $\al^{1} = (\sg_{m}^{n})_{m\in\Z}$, $\al^{0} = (\lm_{m}^{\ld})_{m\in\Z}$, and
\[
  f(\lm,\al) = \frac{\al_{k}-\lm}{\vs_{k}(\lm)}f_{k}(\lm,\al),\qquad
  f_{k}(\lm,\al) = \frac{k\pi}{\sqrt[+]{\lm-\lm_{0}^{+}}}\prod_{m\neq k} \frac{\al_{m}-\lm}{\vs_{m}(\lm)}.
\]
By~\eqref{phi-n-func} the functions
$f_{k}\colon (\C\setminus\bigcup_{m\neq k} G_{m})\times \ell_{\C}^{p}\to \C$ and $f\colon (\C\setminus\bigcup_{m\in\Z} G_{m})\times \ell_{\C}^{p}\to \C$ are analytic. One further computes that
\[
  \partial_{\al_{m}} f(\lm,\al) =  \frac{\al_{k}-\lm}{\al_{m}-\lm}\frac{f_{k}(\lm,\al)}{\vs_{k}(\lm)},\quad
  m\neq k,\qquad
  \partial_{\al_{k}} f(\lm,\al) =  \frac{f_{k}(\lm,\al)}{\vs_{k}(\lm)} .
\]
For $0\le t\le 1$, let $\al^{t} = (\al_{m}^{t}) = ((1-t)\sg_{m}^{n} + t\lm_{m}^{\ld})$. Then
\begin{align*}
  f(\lm,\al^{1}) - f(\lm,\al^{0})
  &=
  \int_{0}^{1}\sum_{m}
  \partial_{\al_{m}} f(\lm,\al^{t})(\sg_{m}^{n}-\lm_{m}^{\ld})\,\dt\\
  &=
  \int_{0}^{1}
  \p*{\sum_{m\neq k} \frac{\sg_{m}^{n}-\lm_{m}^{\ld}}{\al_{m}^{t}-\lm}}
  \frac{(\al_{k}^{t}-\lm)f_{k}(\lm,\al^{t})}{\vs_{k}(\lm)}\,\dt\\
  &\qquad +
  (\sg_{k}^{n}-\lm_{k}^{\ld})
  \int_{0}^{1}
  \frac{f_{k}(\lm,\al^{t})}{\vs_{k}(\lm)}\,\ds.
\end{align*}
Consequently, by Lemma~\ref{int-wm-quot-est}
\begin{align*}
  \abs{A_{nk}} 
  \le \frac{n}{k}\abs{F_{k}^{2}}_{G_{k}}
  \sup_{0\le t\le 1}\Biggl(&
   \abs[\bigg]{\sum_{m\neq k} \frac{\sg_{m}^{n}-\lm_{m}^{\ld}}{\al_{m}^{t}-\lm}}_{G_{k}}
  \abs{(\al_{k}^{t}-\lm)f_{k}(\lm,\al^{t})}_{G_{k}}\\
  &\qquad + \abs{\sg_{k}^{n}-\lm_{k}^{\ld}}\abs{f_{k}(\lm,\al^{t})}_{G_{k}}\Biggr).
\end{align*}
Since $\abs{F_{k}^{2}}_{G_{k}} = O(\gm_{k}^{2}/k^{2})$ by Lemma~\ref{F-prop}, $\abs{f_{k}(\lm,\al^{t})}_{G_{k}}$ is bounded uniformly in $k$ and $0\le t\le 1$ by Proposition~\ref{std-prd-est}, and $\abs{\al_{k}^{t}-\lm}_{G_{k}} = O(\gm_{k})$ uniformly in $k$ and $0\le t \le 1$, we get
\[
  \abs{A_{nk}} \le \frac{n}{k^{3}}\abs{\gm_{k}}^{2}
  \p*{
  \abs{\gm_{k}}
   \sup_{0\le t\le 1}\abs*{\sum_{m\neq k} \frac{\sg_{m}^{n}-\lm_{m}^{\ld}}{\al_{m}^{t}-\lm}}_{G_{k}}
   + O(\abs{\sg_{k}^{n}-\lm_{k}^{\ld}})}.
\]

Finally, note that for some $c > 0$,
\[
  \inf_{\lm\in G_{k}}\abs{\al_{m}^{t}-\lm} \ge c\abs{m^{2}-k^{2}},\qquad m\neq k,\quad 0\le t\le 1.
\]

\textbf{Case A. $s=-1$}: Since $\sg_{m}^{n}-\lm_{m}^{\ld} = \gm_{m}\ell_{m}^{2}$ by Proposition~\ref{prop:sg-lm}, we obtain from Lemma~\ref{mht-2} that $A_{nk} = nk^{-3}\gm_{k}^{3}\ell_{k}^{2}$.
Moreover, with $\al = -1/2-s$, we obtain $B_{nk} = nk^{-3}\gm_{k}^{3}\ell_{k}^{p}$.
Hence, altogether we have shown that
\[
  (\sg_{n}^{n}-\tau_{k})\Om_{nk}^{(2)} = A_{nk} + B_{nk}
   = nk^{-3}\gm_{k}^{3}\ell_{k}^{2}.
\]

\textbf{Case B. $-1 < s\le 0$}: Since $\sg_{m}^{n}-\lm_{m}^{\ld} = n^{-(1-\rho)s}m^{-1-\rho s}\gm_{m}\ell_{m}^{p}$ by Proposition~\ref{prop:sg-lm}, we obtain from Lemma~\ref{mht-2} that
\[
  A_{nk} = \frac{n^{1-(1-\rho)s}}{k^{3}}k^{-1-\rho s} \gm_{k}^{3}\ell_{k}^{p}
  = n^{1-(1-\rho)s}k^{-4-\rho s} \gm_{k}^{3}\ell_{k}^{p}.
\]
Moreover, with $\al = 1/2 - (1-\rho)s$,
\[
  B_{nk} = n^{1-(1-\rho)s}k^{-4-\rho s} \gm_{k}^{3}\ell_{k}^{p}.
\]
Altogether we thus have shown that
\[
  (\sg_{n}^{n}-\tau_{k})\Om_{nk}^{(2)} = A_{nk} + B_{nk}
   = n^{1-(1-\rho)s}k^{-4-\rho s} \gm_{k}^{3}\ell_{k}^{p}.
\]
This completes the proof of the claimed asymptotics for $\Om_{nk}^{(2)}$ with $k\neq n$.

It remains to consider the case $k=n$. In view of~\eqref{psi-c-root-quot},
\begin{align*}
  \Om_{nn}^{(2)} 
  &= \ii\int_{\Gm_{n}} \frac{F_{n}^{2}(\lm)\zt_{n}(\lm)}{\vs_{n}(\lm)}\,\dlm\\
  &= \ii\int_{\Gm_{n}} \frac{
  \p[\Big]{-\frac{\vs_{n}^{2}(\lm)}{4n^{2}\pi^{2}} + F_{n}^{2}(\lm) + \frac{\vs_{n}^{2}(\lm)}{4n^{2}\pi^{2}}}
  \p[\Big]{1+\zt_{n}(\lm)-1}}{\vs_{n}(\lm)}\,\dlm\\
  &= -\ii\int_{\Gm_{n}}
  \frac{\vs_{n}(\lm)}{4n^{2}\pi^{2}}\,\dlm
  +
  \ii\int_{\Gm_{n}}
   \frac{
  F_{n}^{2}(\lm) + \frac{\vs_{n}^{2}(\lm)}{4n^{2}\pi^{2}}
  }{\vs_{n}(\lm)}\,\dlm
  +
  \ii\int_{\Gm_{n}}
   \frac{
  F_{n}^{2}(\lm)
  \p*{\zt_{n}(\lm)-1}}{\vs_{n}(\lm)}\,\dlm.
\end{align*}
By Remark~\ref{std-prd-simple-est} we have $\abs{\zt_{n}(\lm) - 1}_{G_{n}} = n^{-s-1}\ell_{n}^{p}$ and by Lemma~\ref{Fk-exp},
\[
  \abs*{F_{n}^{2}(\lm)+\frac{\vs_{n}^{2}(\lm)}{4n^{2}\pi^{2}}}_{G_{n}} 
  = \frac{\gm_{n}^{2}}{4n^{2}\pi^{2}}n^{-s-1}\ell_{n}^{p}.
\]
We may thus apply Lemma~\ref{int-wm-quot-est} to obtain the estimate
\begin{align*}
  &\frac{1}{2\pi}\abs*{
  \Om_{nn}^{(2)}  + \ii\int_{\Gm_{n}} \frac{\vs_{n}(\lm)}{4n^{2}\pi^{2}}\,\dlm}\\
  &\qquad\le
  \abs*{F_{n}^{2}(\lm)+\frac{\vs_{n}^{2}(\lm)}{4n^{2}\pi^{2}}}_{G_{n}}
  + \abs*{F_{n}^{2}(\lm)}_{G_{n}}\abs{\zt_{n}(\lm)-1}_{G_{n}}
  \\
  &\qquad = \frac{\gm_{n}^{2}}{4n^{2}\pi^{2}}n^{-s-1}\ell_{n}^{p}.
\end{align*}
Finally, the integral $\int_{\Gm_{n}} \vs_{n}(\lm)\,\dlm$ can be explicitly computed.
If $\gm_{n}=0$, then $\vs_{n}(\lm) = (\tau_{n}-\lm)$ and hence $\int_{\Gm_{n}} \vs_{n}(\lm)\,\dlm = 0$.
On the other hand, if $\gm_{n}\neq 0$, then we may use~\eqref{s-root-sides} to compute
\[
  \int_{\Gm_{n}} \vs_{n}(\lm)\,\dlm
   = \ii \frac{\gm_{n}^{2}}{2} \int_{-1}^{1} \sqrt[+]{1-t^{2}}\,\dt
   = \ii\pi \frac{\gm_{n}^{2}}{4}.
\]
Consequently,
\begin{align*}
  \Om_{nn}^{(2)}
   = \frac{\gm_{n}^{2}}{16n^{2}\pi}\p*{1 + n^{-1-s}\ell_{n}^{p}}.
\end{align*}

Going through the arguments of the proof, one verifies that the estimates hold locally uniformly on $\Ws^{s,p}$.\qed
\end{proof}

Our first main result for the KdV frequencies is the following formula for their analytic extension.

\begin{thm}
\label{om-n-analytic}
For any $n\ge 1$, the sum $-12\sum_{k\ge 1} k \Om_{nk}^{(2)}$ converges locally uniformly on $\Ws$ to the analytic function $\om_{n}^{(1)\star}$,
\begin{equation}
  \label{omn-1-star-formula}
  \om_{n}^{(1)\star} = -12\sum_{k\ge 1} k \Om_{nk}^{(2)}.~\fish
\end{equation}
\end{thm}

\begin{rem}
\label{rem:om-n-analytic}
(i) In~\cite{Kappeler:2006fr} it has been shown that $\om_{n}^{(1)\star}$ extends to an analytic function by a different formula. The formula~\eqref{omn-1-star-formula} allows us to obtain asymptotic estimates for $\om_{n}^{(1)\star}$.

(ii)
Let $\Vs$ denote the image of the map $\Ws\to \ell_{\C}^{-1,1}$, $q\mapsto (z_{n}z_{-n})_{n\ge 1}$, then $\Vs$ defines a complex neighborhood of $\ell_{+}^{-1,1}$. 
Using Theorem~\ref{om-n-analytic}, one can argue as in the proof of~\cite[Theorem~20.3]{Grebert:2014iq}, to see that for any $n\ge 1$, the frequency $\om_{n}^{(1)\star}$ is a real analytic function of the actions on $\Vs$.\map

\end{rem}

\begin{proof}
By Lemma~\ref{Om-nk-prop} (ii), all moments $\Om_{nk}^{(2)}$ are analytic on $\Ws$. Moreover, combining the asymptotics~\eqref{taun-lmn-est}, $\gm_{k} = k^{-s}\ell_{k}^{p}$, and Lemma~\ref{decay-kRnk}, yields for $k\neq n$ and $s=-1$, $p=2$, and $\rho=0$,
\[
  k\Om_{nk}^{(2)}
  = \frac{n^{2}}{n^{2}-k^{2}}\ell_{k}^{1}
  = n\ell_{k}^{1}.
\]
Thus, the sum $\Om_{n}^{(2)} \defl -12\sum_{k\ge 1} k\Om_{nk}^{(2)}$ is absolutely and locally uniformly convergent to an analytic function on $\Ws$. Moreover, the identity $\om_{n}^{(1)\star} = \Om_{n}^{(2)}$, $n\ge 1$, holds for any real-valued finite-gap potential by Lemma~\ref{om-n-kRnk}. Consequently, $\Om_{n}^{(2)}$ is the unique analytic extension of $\om_{n}^{(1)\star}$ from the set of finite-gap potentials to $\Ws$.\qed
\end{proof}

Our second main result for the KdV frequencies concerns their asymptotic behavior. To this end, we introduce frequency map $\om^{(1)\star} = (\om_{n}^{(1)\star})_{n\ge 1}$.

\begin{thm}
\label{kdv-freq-as}
\begin{equivenum}
\item
The map $\om^{(1)\star}\colon \Hs_{0}^{-1}\to \ell^{-1,r}$ is real-analytic for any $r>1$, whose restriction to $\FL_{0}^{s,p}$ with $(s,p)$ admissible is a real analytic map
\[
  \om^{(1)\star} \colon \FL_{0}^{s,p} \to
  \begin{cases}
  \ell^{1+2s,1} & -1 < s < -1/2,\quad p = 2,\\
  \ell^{r}, & s = -1/2,\quad p = 2,\quad r > 1,\\
  \ell^{p/2}, & s = -1/2,\quad p > 2.
  \end{cases}
\]

\item
For any $(s,p)$ admissible
\[
  \om_{n}^{(1)\star} + 6I_{n} =
  \begin{cases}
  n^{-3s-2}\ell_{n}^{1}, & -1 < s < -2/3,\quad p = 2,\\
  \ell_{n}^{1+}, & -2/3 \le s \le -1/2,\quad p = 2,\\
  \ell_{n}^{1+}+\ell_{n}^{p/4}, & s = -1/2,\quad p > 2.
  \end{cases}
\]

\item
Moreover, for any $-1 \le s \le 0$ and $p=2$
\[
  \om_{n}^{(1)\star} + 6I_{n} =
  \begin{cases}
  o(n^{-3s-2}), & -1 \le s < -1/3,\\
  O(n^{-1}), & -1/3 \le s \le 0.
  \end{cases}
\]
\end{equivenum}
All estimates are locally uniform on $\Ws^{s,p}$.~\fish
\end{thm}

\begin{rem}
\begin{equivenum}
\item
Combining Remark~\ref{rem:om-n-analytic} (ii) and the decay estimates of Theorem~\ref{kdv-freq-as} (ii), one obtains that the frequency map, as a function of the actions, is real analytic on $\Vs$. Moreover, for any $(s,p)$ admissible, we introduce
\begin{equation}
  \label{Vsp}
  \Vs^{2s+1,p/2} \defl \setdef{I\in \Vs}{I\in \ell_{\C}^{2s+1,p/2}}.
\end{equation}
Then $\Vs^{2s+1,p/2}\subset\Vs$ defines a complex neighborhood of $\ell_{+}^{2s+1,p/2}$ in $\ell_{\C}^{2s+1,p/2}$.
The restriction of $\om^{(1)\star}$ to $\Vs^{2s+1,p/2}$ is a real analytic map
\[
  \om^{(1)\star}\colon \ell^{2s+1,p/2} \to 
  \begin{cases}
  \ell^{1+2s,1} & -1 < s < -1/2,\quad p = 2,\\
  \ell^{r}, & s = -1/2,\quad p = 2,\quad r > 1,\\
  \ell^{p/2}, & s = -1/2,\quad p > 2.
  \end{cases}
\]
The asymptotics of $\om_{n}^{(1)\star}$, viewed as a function of the actions on $\Vs^{2s+1,p/2}$, are the same as the ones stated in Theorem~\ref{kdv-freq-as} (ii) and (iii) for $\om_{n}^{(1)\star}$ on $\Ws^{s,p}$.

\item
Suppose $u\in\Hs_{c}^{s}$ with $c$ an arbitrary real number. Write $u = c + q$ with $q\in\Hs_{0}^{s}$, then
\[
  \om_{n}^{(1)}(u) = (2n\pi)^{3} + 6c (2n\pi) + \om_{n}^{(1)\star}(q).\map
\]
\end{equivenum}
\end{rem}

\begin{proof}
(ii)
Combining Lemma~\ref{decay-kRnk} and the asymptotics~\eqref{taun-lmn-est}, $\gm_{n} = n^{-s}\ell_{n}^{p}$, yields with $\rho=1$, $k\Om_{nk}^{(2)} = \frac{n}{n^{2}-k^{2}}k^{-4s-3}\ell_{k}^{p/4}$ for $k\neq n$ and $(s,p)$ admissible with $-1\le s\le 0$.
Note that for $(s,2)$ with $-1\le s \le -3/4$, one has $1\ge -4s-3\ge 0$ and we conclude with Lemma~\ref{ht-2} that
\begin{equation}
  \label{kOmnk-2-aux-est-1}
  \sum_{k\neq n} k \Om_{nk}^{(2)}
   = n^{-4s-3} \ell_{n}^{1+}.
\end{equation}
Similarly, for $(s,2)$ with $-3/4 < s < -1/2$, we have
\begin{equation}
  \label{kOmnk-2-aux-est-2}
  \sum_{k\neq n} k \Om_{nk}^{(2)}
   = \ell_{n}^{1+}.
\end{equation}
Finally, for $s = -1/2$, we have $-4s-3=-1$, hence $k^{-4s-3}\ell_{k}^{p/4} = \ell_{k}^{1}$ for any $2\le p < \infty$, and we conclude with Lemma~\ref{ht-2} that~\eqref{kOmnk-2-aux-est-2} holds in this case as well.

Next we consider the case $k=n$. By Lemma~\ref{decay-kRnk}
\begin{equation}
  \label{nOmnn-2-aux}
    n\Om_{nn}^{(2)}
   = \frac{\gm_{n}^{2}}{16n\pi} + n^{-3s-2}\ell_{n}^{p/3}
   = n^{-2s-1}\ell_{n}^{p/2}.
\end{equation}

Combining estimates~\eqref{kOmnk-2-aux-est-1}-\eqref{nOmnn-2-aux}, and using $\om_{n}^{(1)\star} = -12\sum_{k\ge 1} \Om_{nk}^{(2)}$, we obtain for any $(s,p)$ admissible with $-1\le s\le 0$
\begin{align}
  \label{omn-1-star-est-admn}
  \om_{n}^{(1)\star} &= 
  -\frac{3}{4}\frac{\gm_{n}^{2}}{n\pi}
   + n^{-3s-2}\ell_{n}^{p/3} + n^{(-4s-3)_{+}} \ell_{n}^{1+}.
\end{align}
By~\eqref{In-gmn-est} we have for any $(s,p)$ admissible with $-1\le s\le 0$
\begin{equation}
  \label{gmn2-In-asymptotics}
  \frac{\gm_{n}^{2}}{8n\pi} - I_{n} = n^{-3s-2}(\ell_{n}^{p/4}+\ell_{n}^{1}).
\end{equation}

In particular, $\om_{n}^{(1)\star} = n\ell_{n}^{1+}$ for $s=-1$, and for $-1 < s < -2/3$
\[
  \om_{n}^{(1)\star}
   = -6I_{n} + n^{-3s-2}\ell_{n}^{1}
   = n^{-1-2s}\ell_{n}^{1},
\]
while for $-2/3\le s < -1/2$ we have
\[
  \om_{n}^{(1)\star}
   = -6I_{n} + \ell_{n}^{1+}
   = n^{-1-2s}\ell_{n}^{1},
\]
and if $s=-1/2$ and $2 \le p < \infty$ using that $ n^{-1/2}\ell_{n}^{p/3}  = \ell_{n}^{p/4}$ for $p\ge 3$
\[
  \om_{n}^{(1)\star}
  = -6I_{n} + \ell_{n}^{p/4} + \ell_{n}^{1+}.
\]
By going through the arguments of the proof, one sees that the estimates hold locally uniformly on $\Ws^{s,p}$.

(i) Since $\gm_{n}^{2}/n = n^{-2s-1}\ell_{n}^{p/2}$ by~\eqref{taun-lmn-est}, we find $I_{n} = n^{-2s-1}\ell_{n}^{p/2}$ and conclude together with item (ii) that locally uniformly on $\Ws^{s,p}$
\[
  \om_{n}^{(1)\star} = 
  \begin{cases}
  n\ell_{n}^{+1}, & s = -1,\quad p = 2,\\
  n^{-2s-1}\ell_{n}^{1}, & -1 < s < -1/2,\quad p = 2,\\
  \ell_{n}^{p/2}, & s=-1/2,\quad 2\le p < \infty.
  \end{cases}
\]
Since by Theorem~\ref{om-n-analytic} each $\om_{n}^{(1)\star}$, $n\ge 1$, is analytic on $\Ws$, the claimed analyticity statements for $\om^{(1)\star}$ follow.

(iii)
By \eqref{kOmnk-2-aux-est-1} we have $\sum_{k\neq n} k \Om_{nk}^{(2)} = o(n^{-4s-3})$ for $k\neq n$ and $-1\le s \le -3/4$. On the other hand, for $(s,p)$ admissible with $-3/4 < s \le 0$ and $p\le 4$, we use that $\abs{n-k} \le n/2$ implies $\abs{k} \ge \abs{n}/2$ to conclude
\begin{align*}
  \sum_{k\neq n} k \Om_{nk}^{(2)}
  &=
  n\sum_{\abs{n-k}>n/2} \frac{k^{-4s-3}}{n^{2}-k^{2}}\ell_{k}^{1}
  +
  n\sum_{1\le \abs{n-k}\le n/2} \frac{k^{-4s-3}}{n^{2}-k^{2}}\ell_{k}^{1}\\
  &= O(n^{-1}) + O(n^{-4s-3}).
\end{align*}
Finally, by~\eqref{nOmnn-2-aux} and~\eqref{gmn2-In-asymptotics}, $n\Om_{nn}^{(2)} = \frac{\gm_{n}^{2}}{16n\pi} + o(n^{-3s-2}) = \frac{1}{2}I_{n} + o(n^{-3s-2})$.

For $-1 \le s < -1/3$ we have $-3s-2 > -4s-3 > -1$ and hence
\[
  \om_{n}^{(1)\star} = -6I_{n} + o(n^{-3s-2}),
\]
while for $-1/3\le s \le 0$ with $p=2$ we have
\[
  \om_{n}^{(1)\star} = -6I_{n} + O(n^{-1}).
\]
By going through the arguments of the proof, one sees that the estimates hold locally uniformly on $\Ws^{s,2}$.\qed
\end{proof}

\begin{thm}
\label{om1-diffeo}
If either $-1 < s < -1/2$ and $p=2$ or $s = -1/2$ and $p > 2$, then the map $\om^{(1)\star}\colon \ell_{+}^{2s+1,p/2}\to \ell^{2s+1,p/2}$
\begin{equivenum}
\item
is a local diffeomorphism near $I=0$,
\item
is a local diffeomorphism on a dense open subset of $\ell_{+}^{2s+1,p/2}$,
\item
is a Fredholm map of index zero everywhere.~\fish
\end{equivenum}
\end{thm}

\begin{proof}
Throughout this proof we assume that either $-1 < s < -1/2$ and $p=2$ or $s = -1/2$ and $p > 2$. In either case $\om^{(1)\star}\colon \ell_{+}^{2s+1,p/2}\to \ell^{2s+1,p/2}$ is real analytic in view of Theorem~\ref{kdv-freq-as} (i).

(i): Since $\ddd_{0}\om^{(1)\star} = -6 \Id$, it follows from the inverse function theorem that $\om^{(1)\star}$ is a local diffeomorphism near $I=0$.

(ii): We show that for any $I$ in $\Vs^{2s+1,p/2}$, defined in~\eqref{Vsp}, the map $\Lm_{I} \defl \ddd_{I}\om^{(1)\star} + 6\Id_{\ell_{\C}^{2s+1,p/2}}$ is a compact operator on $\ell_{\C}^{2s+1,p/2}$. We treat the three cases
\begin{align*}
  \text{(A)}\quad & \;\;\;\;-1 < s < -2/3,\quad p = 2,\\
  \text{(B)}\quad & -2/3 \le s < -1/2,\quad p = 2,\\
  \text{(C)}\quad & \qquad \qquad\; s = -1/2,\quad p > 2,
\end{align*}
separately. In case (A), by Theorem~\ref{kdv-freq-as} (ii) $\om_{n}^{(1)\star} + 6 I_{n} = n^{-3s-2}\ell_{n}^{1}$ locally uniformly, hence by Cauchy's estimate for any $I\in \Vs^{2s+1,1}$ the map $\Lm_{I}\colon \ell^{2s+1,1}\to \ell^{3s+2,1}$ is bounded. Since $\ell^{3s+2,1}$ embeds compactly into $\ell^{2s+1,1}$, the operator $\Lm_{I}$ is compact on $\ell^{2s+1,1}$. In case (B), $\om_{n}^{(1)\star} + 6 I_{n} = \ell_{n}^{1+}$ locally uniformly whence $\Lm_{I}\colon \ell^{2s+1,1}\to \ell^{r}$ is bounded for any $r > 1$. The claim in the case (B) follows from the fact that $\ell^{r}$ embeds compactly into $\ell^{2s+1,1}$ if $r > 1$ is chosen sufficiently small.
Finally in case (C), we have $\om_{n}^{(1)\star} + 6I_{n} = \ell_{n}^{p/4} + \ell_{n}^{1+}$. Hence there exists $1 < r < p/2$ so that $\Lm_{I}\colon \ell^{p/2}\to \ell^{r}$ is bounded. It now follows from Pitt's Theorem -- see~\cite{Delpech:2009ek} for a short proof -- that $\Lm_{I}$ is compact.

When combined with item (i) above, Proposition~\ref{diffeo-prop} from Appendix~\ref{app:diffeo} implies that $\om_{n}^{(1)\star}$ is a local diffeomorphism on a dense open subset of $\ell_{+}^{2s+1,p/2}$.

(iii): Since $\Lm_{I}\colon \ell^{2s+1,p/2}\to \ell^{2s+1,p/2}$ is compact it follows that $\ddd_{I}\om^{(1)\star}$ is a compact perturbation of the identity and hence a Fredholm operator of index zero.~\qed
\end{proof}

\begin{proof}[Proof of Corollary~\ref{cor:kdv-freq-localdiffeo}.]
The claimed result follows from Theorem~\ref{om1-diffeo} (i)-(ii).~\qed
\end{proof}

\begin{proof}[Proof of Corollary~\ref{cor:kdv-L2-freq-O}.]
It follows from item (iii) of Theorem~\ref{kdv-freq-as} that locally uniformly on $\Ws^{s,2}$
\[
  \om_{n}^{(1)\star} =
  \begin{cases}
  o(n^{-1-2s}), & -1\le s < 0,\\
  O(n^{-1}), & s = 0.
  \end{cases}
\]
For $-1 < s < 0$, the leading term of $\om_{n}^{(1)\star}$ is $-6\frac{\gm_{n}^{2}}{8n\pi}$. Hence, the estimate is sharp in the sense that for $\ep > 0$ arbitrary small
\[
  \om_{n}^{(1)\star} = O(n^{-1-2s-\ep})
\]
does not hold locally uniformly on $\Hs_{0}^{s}$.
Moreover, For $q\in L_{0}^{2}$ the corresponding action variables are in $\ell_{+}^{1,1}$ and one has that
\[
  \frac{\gm_{n}^{2}}{8n\pi} = I_{n} + n^{-1}\ell_{n}^{2}
\]
uniformly on bounded subsets of $L_{0}^{2}$ - cf. \cite{Kappeler:2003up}. Consequently, $\om_{n}^{(1)\star} = O(n^{-1})$ uniformly on bounded subsets of $L_{0}^{2}$.\qed
\end{proof}

\begin{proof}[Proof of Corollary~\ref{cor:Airy-Kdv}.]
The same arguments used in the proof of Theorem~8.2 in~\cite{Kappeler:2013bt} apply. Using Corollary~\ref{cor:kdv-L2-freq-O} and the one-smoothing property of the Birkhoff map established in~\cite{Kappeler:2013bt}, Theorem~1.1, the claimed result follows.~\qed
\end{proof}

\subsection{Hamiltonian}

In~\cite{Korotyaev:2011tw}, using results from~\cite{Bikbaev:1993jl}, it was shown that the renormalized KdV Hamiltonian $\Hm_{1}^{\star}\colon \ell_{+}^{2}\to \R$, introduced in~\eqref{Hm-renormalized}, is a continuous function which is concave on all of $\ell_{+}^{2}$. Subsequently, it was shown in~\cite{Kappeler:CNzeErmy} that $\Hm_{1}^{\star}$ is real analytic, $\ddd_{I}^{2}\Hm_{1}^{\star} \le 0$ for all $I\in \ell_{+}^{2}$, and that $\Hm_{1}^{\star}$ is strictly concave near $I=0$ in the sense that
\[
  \ddd_{I}^{2}\Hm_{1}^{\star}(J,J) \le -\lin{J,J}_{\ell^{2},\ell^{2}},\qquad \forall J\in \ell^{2},
\]
for all $I$ in a sufficiently small neighborhood of the origin in $\ell_{+}^{2}$.

\begin{thm}
The KdV Hamiltonian $H^{\star}\colon \ell_{+}^{2}\to \C$ is strictly concave on a dense open subset of $\ell_{+}^{2}$.~\fish
\end{thm}

\begin{proof}
By the above discussion, $\ddd_{I}^{2}\Hm_{1}^{\star} \le 0$ holds for all $I\in\ell_{+}^{2}$. Further, $\ddd_{I}^{2}\Hm_{1}^{2} < 0$ holds whenever $\ddd_{I}^{2}\Hm_{1}^{\star} = \ddd_{I}\om^{(1)\star}\colon \ell_{+}^{2}\to \ell_{+}^{2}$ is a diffeomorphism at $I$. By Theorem~\ref{om1-diffeo}, $\ddd_{I}\om^{(1)\star}$ is a diffeomorphism on a dense open subset of $\ell_{+}^{2}$, which proves the claim.\qed
\end{proof}

\subsection{Wellposedness}

We briefly recall some wellposedness results for the KdV equation on the circle which are most closely related to our main result. According to \cite{Kappeler:2006fr}, the KdV equation is globally $C^{0}$-wellposed in $\Hs_{0}^{s}$ for any $s \ge -1$, i.e. for any $T > 0$ the solution map
\[
  \Sc\colon \Hs_{0}^{s}\to C^{0}([-T,T],\Hs_{0}^{s})
\]
is continuous. It was further shown in \cite{Kappeler:CNzeErmy} that the KdV equation is also globally $C^{0}$-wellposed in the Fourier Lebesgue spaces $\FL_{0}^{s,p}$ for any $p > 2$ and $-1/2\le s\le 0$. In the analytic class, \citet{Colliander:2003fv} proved that the KdV equation is $C^{\om}$-wellposedness in $\Hs_{0}^{s}$ for any $s\ge -1/2$, i.e. for any $T > 0$ the solution map $\Sc\colon \Hs_{0}^{s}\to C^{0}([-T,T],\Hs_{0}^{s})$ is real-analytic. They also proved that the KdV equation is globally uniformly $C^{0}$-wellposed in $\Hs_{0}^{s}$ for any $s\ge -1/2$, i.e. for any $T > 0$ the solution map $\Sc\colon \Hs_{0}^{s}\to C^{0}([-T,T],\Hs_{0}^{s})$ is uniformly continuous on bounded subsets.
There also exist several illposedness results. \citet{Christ:2003tx} showed that the KdV equation is \emph{not} uniformly $C^{0}$-wellposed in $\Hs_{0}^{s}$ with $-1 \le s < -1/2$.
Moreover, \citet{Bourgain:1997gg} proved that the KdV equation is \emph{not} $C^{3}$-wellposed in $\Hs_{0}^{s}$ with $s < -1/2$.
In \citet{Molinet:2012il} showed that KdV is illposed in $\Hs_{0}^{s}$ for $s < -1$.

The following result answers in particular the question, whether the KdV equation is $C^{1}$ or $C^{2}$-wellposed on $\Hs_{0}^{s}$ for $-1 < s < -1/2$.

\begin{thm}
\label{thm:kdv-wp}
\begin{equivenum}
\item For any $2\le p < \infty$ and $-1/2\le s \le 0$, the KdV equation is $C^{\om}$-wellposed on $\FL_{0}^{s,p}$.

\item For any $-2/3 \le s < -1/2$ and $t > 0$, the solution map
\[
  S^{t}\colon \Hs_{0}^{s}\to \Hs_{0}^{s}
\]
is nowhere locally uniformly continuous. In particular, the KdV equation is not $C^{k}$-wellposed, $k\ge 1$, in $\Hs_{0}^{s}$ for any $-2/3 \le s < -1/2$.

\item For any $-1 < s < -2/3$ and $T > 0$, the solution map
\[
  S\colon \Hs_{0}^{s}\to C^{0}([-T,T],\Hs_{0}^{s})
\]
is nowhere locally uniformly continuous. In particular, the KdV equation is not $C^{k}$-wellposed, $k\ge 1$, in $\Hs_{0}^{s}$ for any $-1 < s < -2/3$.~\fish
\end{equivenum}
\end{thm}

\begin{rem}
We expect that statement (ii) of Theorem~\ref{thm:kdv-wp} remains valid for $-1\le s < -2/3$.\map
\end{rem}

Before proving Theorem~\ref{thm:kdv-wp}, we first prove corresponding results in Birkhoff coordinates.
By Theorem~\ref{kdv-freq-as} from the previous section,  the KdV frequencies $\om_{n}^{(1)}$ give rise to a flow $\Sc_{\Ph} \colon(t,z) \mapsto (\ph_{n}^{t}(z))_{n\in\Z}$ in Birkhoff coordinates on $\ell_{0}^{s+1/2,p}$ with coordinate functions
\begin{equation}
  \label{kdv-coordinate-functions}
  \ph_{n}^{t}(z) = \e^{\ii \om_{n}^{(1)}(z) t}z_{n},\qquad n\in\Z.
\end{equation}
Here, the KdV frequencies are viewed as analytic functions of the Birkhoff coordinates and as such have been extended to the bi-infinite sequence $(\om_{n}^{(1)})_{n\in\Z}$ by setting
\[
  \om_{0}^{(1)}(z) = 0,\qquad \om_{-n}^{(1)}(z) = -\om_{n}^{(1)}(z),\quad n\ge 1.
\]
The KdV solution map on $\FL_{0}^{s,p}$ is then given by
\begin{equation}
  \label{kdv-soln-map}
  \Sc^{t} = \Ph^{-1}\circ \Sc_{\Ph}^{t}\circ \Ph.
\end{equation}
We first establish properties of the map $\Sc_{\Ph}$ corresponding to the ones of $\Sc$.

\begin{thm}
\label{thm:kdv-wp-bh}
\begin{equivenum}
\item For any $-1/2\le s \le 0$, $2\le p < \infty$, and $T > 0$, the map $\Sc_{\Ph}\colon \ell_{0}^{s+1/2,p}\to C([-T,T],\ell_{0}^{s+1/2,p})$ is real-analytic.

\item For any $-2/3 \le s < -1/2$ and $t > 0$, the map $\Sc_{\Ph}^{t}\colon h_{0}^{s+1/2}\to h_{0}^{s+1/2}$ is nowhere locally uniformly continuous.

\item For any $-1 < s < -2/3$ and $T > 0$, the map $\Sc_{\Ph}\colon h_{0}^{s+1/2}\to C([-T,T],h_{0}^{s+1/2})$ is nowhere locally uniformly continuous.~\fish
\end{equivenum}
\end{thm}

\begin{proof}
(i)
Suppose $-1/2\le s\le 0$ and $2 \le p < \infty$, then by Theorem~\ref{kdv-freq-as} (i) the map $\om^{(1)\star}\colon \ell_{0}^{s+1/2,p}\to \ell^{\infty}$ is real analytic. The analyticity of $\Sc_{\Ph}$ thus follows from Theorem~\ref{thm:freq-flow} (iii).

\newcommand{\vt}{\vartheta}
(ii)
For $-2/3\le s < -1/2$ let $\sg = -(s+1/2)$ so that $0 < \sg \le 1/6$.
We show that for any $t > 0$ and any nonempty open subset $U\subset h_{0}^{-\sg}$, the map $\Sc_{\Ph}^{t}\big|_{U}\colon U\to h_{0}^{\sg}$ is not locally uniformly continuous.
After possibly shrinking $U$, by Theorem~\ref{kdv-freq-as} (ii) there exists $N_{\star}\ge 1$ so that
\begin{equation}
  \label{omn-z-coordinates}
  \om_{n}^{(1)\star}(z) = -6z_{n}z_{-n} + r_{n}(z),
\end{equation}
with $\sup_{n\ge N_{\star}}\abs{r_{n}(z)} \le \pi/(4t)$ for all $z\in U$.
We show that there exist two sequences $p^{(m)}$ and $q^{(m)}$ in $U$ and $\eta_{0} > 0$ so that
\[
  \n{p^{(m)}-q^{(m)}}_{h^{-\sg}} \to 0,\qquad
  \n{\Sc_{\Ph}^{t}(p^{(m)})-\Sc_{\Ph}^{t}(q^{(m)})}_{h^{-\sg}} \ge \eta_{0}.
\]
To this end, fix any $z^{\o}\in U$ so that there exists $N\ge N_{\star}$ with $z_{\pm n}^{\o} = 0$ for all $n\ge N$.
For $\dl > 0$ define $p_{\pm n}^{\dl,m} = q_{\pm n}^{\dl,m} = z_{\pm n}^{\o}$ if $1\le n\le N$, and for $n > N$,
\[
  p_{\pm n}^{\dl,m} = \begin{cases}
  \dl n^{\sg}, & n = 2^{m},\\
  0, & \text{otherwise},
  \end{cases}
  \qquad
  q_{\pm }^{\dl,m} = \begin{cases}
  p_{\pm n}^{\dl,m} \pm \ii \dl m^{1/2}, & n = 2^{m},\\
  0, & \text{otherwise},
  \end{cases}
\]
A straightforward computation gives with $n_{m} \defl 2^{m}$,
\[
  \frac{1}{\sqrt{2}}\n{p^{\dl,m}-z^{\o}}_{h^{-\sg}} = \dl,\qquad
  \frac{1}{\sqrt{2}}\n{q^{\dl,m}-z^{\o}}_{h^{-\sg}} = \dl\sqrt{1+ n_{m}^{-2\sg}m}.
\]
Since $\sg > 0$ we can choose $\dl_{0}\in (0,1)$ so that the sequences $(p^{\dl,m})$ and $(q^{\dl,m})$ are both contained in $U$ for any $0 < \dl < \dl_{0}$. Moreover,
\[
  \frac{1}{\sqrt{2}}\n{p^{\dl,m}-q^{\dl,m}}_{h^{-\sg}}
  \le
  \dl_{0} n_{m}^{-\sg}m^{1/2}
  \to 0,\qquad m\to \infty,
\]
and by~\eqref{omn-z-coordinates} one has
\[
  \om_{n_{m}}^{(1)}(p^{\dl,m}) - \om_{n_{m}}^{(1)}(q^{\dl,m})
  = -6\dl^{2}m + r_{n_{m}}(p^{\dl,m}) - r_{n_{m}}(q^{\dl,m}).
\]
Choose $k\ge 1$ so that $\dl\equiv \dl(t) = \sqrt{\pi/6tk}\le \dl_{0}$. Consequently,
\[
  \p*{\frac{m}{k}-\frac{1}{2}}\pi
  \le
  \p*{\om_{n_{m}}^{(1)}(p^{\dl,m}) - \om_{n_{m}}^{(1)}(q^{\dl,m})}t
  \le
  \p*{\frac{m}{k}+\frac{1}{2}}\pi,\qquad
  n_{m}\ge N.
\]
With $m_{j} = (2j+1)k$ we conclude
\[
  \abs*{\exp\p*{\ii \p*{\om_{n_{m_{j}}}^{(1)}(p^{\dl,m_{j}}) - \om_{n_{m_{j}}}^{(1)}(q^{\dl,m_{j}})}t}-1} \ge 1,\qquad
  n_{m_{j}}\ge N.
\]
Thus, by comparing only the $n_{m_{j}}$th component,
\begin{align*}
  \frac{1}{\sqrt{2}}\n{\Sc_{\Ph}^{t}(p^{\dl,m_{j}})-\Sc_{\Ph}^{t}(q^{\dl,m_{j}})}_{h^{-\sg}}
  &\ge n_{m_{j}}^{-\sg}\abs{p_{n_{m_{j}}}^{\dl,m_{j}}} - n_{m_{j}}^{-\sg}\abs{p_{n_{m_{j}}}^{\dl,m_{j}}-q_{n_{m_{j}}}^{\dl,m_{j}}}\\
  &\ge \dl - \n{p^{\dl,m_{j}}-q^{\dl,m_{j}}}_{h^{-\sg}}\\
  &\ge \dl/2,
\end{align*}
for all $j$ sufficiently large.

(ii)
For $-1 < s < -2/3$ let $\sg = -(s+1/2)$ so that $1/6 < \sg < 1/2$.
We show that for any $T > 0$ and any nonempty open subset $U\subset h_{0}^{-\sg}$, the map $\Sc_{\Ph}\big|_{U}\colon U\to C([-T,T],h_{0}^{\sg})$ is not locally uniformly continuous.
After possibly shrinking $U$, we have by Theorem~\ref{kdv-freq-as} (ii) that
\begin{equation}
  \label{omn-z-coordinates-2}
  \om_{n}^{(1)\star}(z) = -6z_{n}z_{-n} + r_{n}(z),
\end{equation}
where $\n{(r_{n})_{n\ge 1}}_{\ell^{-\vt,\infty}}$ is bounded uniformly on $U$ with
\[
  \vt
   \equiv \vt(\sg)
   =
   3\sg-1/2.
\]
We show that there exist two sequences $p^{(m)}$ and $q^{(m)}$ in $U$, a sequence of times $t_{m}\to 0$, and $\eta_{0} > 0$ so that
\[
  \n{p^{(m)}-q^{(m)}}_{h^{-\sg}} \to 0,\qquad
  \n{\Sc_{\Ph}^{t_{m}}(p^{(m)})-\Sc_{\Ph}^{t_{m}}(q^{(m)})}_{h^{-\sg}} \ge \eta_{0}.
\]
To this end, fix any $z^{\o}\in U$ so that there exists $N\ge 1$ with $z_{\pm n}^{\o} = 0$ for all $n\ge N$.
For $\dl > 0$ define $p_{\pm n}^{(m)} = q_{\pm n}^{(m)} = z_{\pm n}^{\o}$ if $1\le n\le N$, and for $n > N$,
\[
  p_{\pm n}^{(m)} = \begin{cases}
  \dl n^{\sg}, & n = 2^{m},\\
  0, & \text{otherwise},
  \end{cases}
  \qquad
  q_{\pm }^{(m)} = \begin{cases}
  p_{\pm n}^{(m)} \pm \ii \dl n^{\vt/2}m^{1/2}, & n = 2^{m},\\
  0, & \text{otherwise},
  \end{cases}
\]
A straightforward computation gives with $n_{m} \defl 2^{m}$,
\[
  \frac{1}{\sqrt{2}}\n{p^{(m)}-z^{\o}}_{h^{-\sg}} = \dl,\qquad
  \frac{1}{\sqrt{2}}\n{q^{(m)}-z^{\o}}_{h^{-\sg}} = \dl\sqrt{1+ n_{m}^{-(2\sg-\vt)}m}.
\]
Since $2\sg-\vt = 1/2-\sg > 0$, we can choose $\dl > 0$ so that the sequences $(p^{(m)})$ and $(q^{(m)})$ are both contained in $U$. Moreover,
\[
  \frac{1}{\sqrt{2}}\n{p^{(m)}-q^{(m)}}_{h^{-\sg}}
  =
  \dl n_{m}^{-(2\sg-\vt)/2}m^{1/2}
  \to 0,\qquad m\to \infty,
\]
and since $\vt\ge 0$, one has by~\eqref{omn-z-coordinates-2} that
\[
  \abs{\om_{n_{m}}^{(1)}(p^{(m)}) - \om_{n_{m}}^{(1)}(q^{(m)})}
  = \abs{6 \dl^{2} n_{m}^{\vt}m + n_{m}^{\vt}\ell_{m}^{\infty}} \ge m^{1/2},
  \qquad
  m\ge M,
\]
where $M$ is chosen sufficiently large. Consequently, one can choose a sequence of times $t_{m}\to 0$ so that
\[
  \abs*{\exp\p*{\ii \p*{\om_{n_{m}}^{(1)}(p^{(m)}) - \om_{n_{m}}^{(1)}(q^{(m)})}t_{m}}-1} \ge 1,\qquad
  m\ge M.
\]
Thererfore,
\begin{align*}
  \frac{1}{\sqrt{2}}\n{\Sc_{\Ph}^{t_{m}}(p^{(m)})-\Sc_{\Ph}^{t_{m}}(q^{(m)})}_{h^{-\sg}}
  &\ge n_{m}^{-\sg}\abs{p_{n_{m}}^{(m)}} - n_{m}^{-\sg}\abs{p_{n_{m}}^{(m)}-q_{n_{m}}^{(m)}}\\
  &\ge \dl - \n{p^{(m)}-q^{(m)}}_{h^{-\sg}}\\
  &\ge \dl/2,
\end{align*}
for all $m\ge M$ sufficiently large.\qed
\end{proof}

\begin{proof}[Proof of Theorem~\ref{thm:kdv-wp}.]

Since the Birkhoff map $\Ph$ is bi-real-analytic, all claims follow from Theorem~\ref{thm:kdv-wp-bh} and the identity $\Sc^{t} = \Ph^{-1}\circ\Sc_{\Ph}^{t}\circ\Ph$.\qed
\end{proof}

\section{KdV2}

\subsection{Frequencies}
\label{ss:KdV2-freq}

Proceeding as for the KdV equation explained in the previous section, we derive in this section formulae for the frequencies of the KdV2 equation. Our starting point is the following identity for the KdV2 frequencies which a priori holds on $\Hs_{0}^{2}\cap (\Ws\setminus Z_{n})$
\[
  \om_{n}^{(2)} = \pbr{\Hm_{2},\th_{n}}.
\]
By~\eqref{exp-om-star-kdv2} the renormalized KdV2 frequencies are given by
\[
  \om_{n}^{(2)\star} = \om_{n}^{(2)} - (2n\pi)^{5} - 20(2n\pi) \Hm_{0},\qquad n\ge 1.
\]

\begin{lem}
\label{om-n-2-star-formula}
For any real-valued finite-gap potential cf.\eqref{fin-gap} with $[q] = 0$ and any $n\ge 1$
\begin{equation}
  \label{om-2-star-Om-nk}
  \om_{n}^{(2)\star} = -160\pi^{2}\sum_{k\ge 1} k^{3} \Om_{nk}^{(2)}
  + 80\sum_{k\ge 1} k\Om_{nk}^{(4)}.~\fish
\end{equation}
\end{lem}

\begin{proof}
We argue as in the proof of Lemma~\ref{om-n-kRnk}.
Suppose $q$ is a finite-gap potential, then there exists $S\subset \N$ finite so that $\gm_{k}(q) \neq 0$ if and only if $k \in S$.
By Lemma~\ref{F2-analytic} (ii), the function $F^{2}$ is analytic outside a sufficiently large circle $C_{r}$ which encloses all open gaps $G_{k}$, $k\in S$, and whose exterior contains $G_{0}$.
Furthermore, $F$ admits according to \eqref{F-exp} an asymptotic expansion for $\nu_{k} = (k+1/2)\pi$. In particular,
\[
  F(\lm)^{6} = 
  -\lm^{3} + \frac{3}{2}\Hm_{0}\lm + \frac{3}{8}\Hm_{1} 
           + \frac{3}{32}(\Hm_{2} - 10\Hm_{0}^{2})\frac{1}{\lm} + O(\lm^{-2}).
\]
so that by Cauchy's Theorem
\[
  \frac{3}{32}(\Hm_{2} - 10\Hm_{0}^{2}) = \frac{1}{\ii 2\pi}\int_{C_{r}} F^{6}(\lm)\,\dlm.
\]
Let $n\in S$, then $\gm_{n}(q)\neq 0$ hence $\th_{n}$ modulo $\pi$ is analytic near $q$.
Since $\pbr{\th_{n},F(\lm)} = \frac{\pbr{\th_{n},\Dl(\lm)}}{\sqrt[c]{\Dl^{2}(4)-4}}$ by Lemma~\ref{F-prop} (i) and $2\{\th_{n},\Dl(\lm)\} = \ps_{n}(\lm)$ by \cite[Proposition F.3]{Kappeler:2003up}, one obtains
\begin{align*}
  \pbr{\Hm_{2} - 10\Hm_{0}^{2},\th_{n}}
  &=
  -\frac{16}{3}\frac{1}{\ii \pi}\int_{C_{r}} \pbr{\th_{n},F^{6}(\lm)}\,\dlm\\
  &=
  -\frac{32}{\ii \pi}\int_{C_{r}} \frac{F^{5}(\lm)\pbr{\th_{n},\Dl(\lm)}}{\sqrt[c]{\Dl^{2}(\lm)-4}}\,\dlm\\
  &=
  -\frac{16}{\ii \pi}\int_{C_{r}} \frac{F^{5}(\lm)\psi_{n}(\lm)}{\sqrt[c]{\Dl^{2}(\lm)-4}}\,\dlm.
\end{align*}
By Lemma~\ref{F2-analytic} (ii) and formula~\eqref{psi-form}, the integrand is analytic on $U_{0}$, while for any $k\in \N\setminus S$, one has $\sg_{k}^{n} = \tau_{k}$ and $\vs_{k}(\lm) = \tau_{k}-\lm$ so that in view of the product representations~\eqref{c-root} and~\eqref{psi-form}, the integrand extends analytically to $U_{k}$. Consequently, the integrand is analytic on $\C\setminus\bigcup_{k\in S} G_{k}$ and one obtains by contour deformation
\[
  \pbr{\Hm_{2} - 10\Hm_{0}^{2},\th_{n}}
   = -\frac{16}{\ii\pi}\sum_{k\in S} \int_{\Gm_{k}} \frac{F^{5}(\lm)\psi_{n}(\lm)}{\sqrt[c]{\Dl^{2}(\lm)-4}}\,\dlm.
\]
Expanding $F(\lm)^{5} = (F_{k}(\lm)-\ii k\pi)^{5}$ yields
\begin{align*}
  F^{5}(\lm) &= F_{k}^{5}(\lm) - 5\ii(k\pi) F_{k}^{4}(\lm)
  -10(k\pi)^{2} F_{k}^{3}(\lm)
  +10\ii(k\pi)^{3} F_{k}^{2}(\lm)\\
  &\qquad +5(k\pi)^{4} F_{k}(\lm)
  -\ii(k\pi)^{5}.
\end{align*}
Recalling from Lemma~\ref{Om-nk-prop} (iii) that $\Om_{nk}^{(5)}$, $\Om_{nk}^{(3)}$, and $\Om_{nk}^{(1)}$ vanish for any $n,k\ge1$, and that $\Om_{nk}^{(0)} = 2\pi\dl_{nk}$ by Lemma~\ref{Om-nk-prop} (i) thus gives
\begin{align*}
  &\pbr{\Hm_{2} - 10\Hm_{0}^{2},\th_{n}}\\
   &\qquad = -\frac{16}{\pi}\sum_{k\in S}
   \p*{
   -5(k\pi) \Om_{nk}^{(4)} + 10 (k\pi)^{3}\Om_{nk}^{(2)} - (k\pi)^{5}\Om_{nk}^{(0)}
   }\\
   &\qquad= \sum_{k\ge1} \p*{80 k \Om_{nk}^{(4)}
   - 160\pi^{2} k^{3}\Om_{nk}^{(2)}
   + (2k\pi)^{5}\dl_{kn}}.
\end{align*}
Here, we used in the last line that $\Om_{nk}^{(m)} = 0$ for all $k\in\N\setminus S$.
Since $\pbr{\Hm_{0}^{2},\th_{n}} = 4n\pi \Hm_{0}$ and $\om_{n}^{(2)\star} = \om_{n}^{(2)} - 20(2n\pi)\Hm_{0} - (2n\pi)^{5}$, it follows that~\eqref{om-2-star-Om-nk} holds for any $n\ge 1$ with $\gm_{n}(q)\neq 0$. One argues as in the proof of Lemma~\ref{om-n-kRnk} to show that the identity also holds for $n\ge 1$ with $\gm_{n}(q) = 0$.~\qed
\end{proof}

The asymptotics of $\Om_{nk}^{(2)}$ have been obtained in Section~\ref{ss:kdv-freq}. Hence it remains to study the asymptotics of $\Om_{nk}^{(4)}$.

\begin{lem}
\label{decay-kOm2nk}
For any $n\ge 1$ and any $q\in\Ws^{s,p}$ with $(s,p)$ admissible with $-1\le s\le 0$
\begin{align*}
  k\Om_{nk}^{(4)} &=
  \frac{n}{n^{2}-k^{2}}k^{-s-5}\gm_{k}^{5}\ell_{k}^{p},
  \quad k\neq n,\\
  n\Om_{nn}^{(4)} &= \frac{3}{16n\pi} \frac{\gm_{n}^{4}}{64n^{2}\pi^{2}}\p*{1
     + n^{-s-1}\ell_{n}^{p}},
\end{align*}
where the estimates hold locally uniformly on $\Ws^{s,p}$.~\fish
\end{lem}

\begin{proof}
By Lemma~\ref{Om-nk-prop} it suffices to consider the case $\gm_{k}\neq 0$ since otherwise $\Om_{nk}^{(4)} = 0$.
We first prove the estimate for $k\neq n$. 
By~\eqref{zt-k-n} and~\eqref{zt-k-n-est},
\[
  \frac{\psi_{n}(\lm)}{\sqrt[c]{\Dl^{2}(\lm)-4}}
   = \frac{\sg_{k}^{n}-\lm}{\vs_{k}(\lm)}\zt_{k}^{n}(\lm),\qquad
  (n^{2}-k^{2})\left.\zt_{k}^{n}(\lm)\right|_{G_{k}} =
  \frac{n}{k}
   \p*{\frac{\ii}{\pi^{2}} + k^{-s-1}\ell_{k}^{p}}.
\]
Shrinking the contour of integration $\Gm_{k}$ to $G_{k}^{-}\cup G_{k}^{+}$ and using~\eqref{s-root-sides} gives
\[
  \Om_{nk}^{(4)} = 2\int_{G_{k}^{-}} \frac{F_{k}^{4}(\lm)(\sg_{k}^{n}-\lm)\zt_{k}^{n}(\lm)}{\vs_{k}(\lm)}\,\dlm.
\]
Since $\abs{\vs_{k}(\lm)}_{G_{k}} = \abs{\gm_{k}}/2$ by~\eqref{s-root-sides}, Lemma~\ref{Fk-exp} gives uniformly for $\lm\in G_{k}^{\pm}$
\[
  F_{k}(\lm)^{4}
   = \frac{1}{(2k\pi)^{4}}\p*{\vs_{k}^{4}(\lm) + \gm_{k}^{4}k^{-s-1}(\ell_{k}^{p/2}+\ell_{k}^{1+})}.
\]
Consequently, uniformly for $\lm\in G_{k}^{\pm}$
\[
  \frac{F_{k}^{4}(\lm)\psi_{n}(\lm)}{\sqrt[c]{\Dl^{2}(\lm)-4}}
  = \frac{n}{k}
  \frac{ \p*{\vs_{k}^{4}(\lm) + \gm_{k}^{4}k^{-s-1}\ell_{k}^{p}}
         \p*{\sg_{k}^{n}-\lm}
         \p*{\ii/\pi^{2} + k^{-s-1}\ell_{k}^{p}}}{(2k\pi)^{4}(n^{2}-k^{2})\vs_{k}(\lm)}.
\]
Since $\abs{\sg_{k}^{n}-\lm}_{G_{k}} = O(\gm_{k})$ by~\eqref{psi-root-asymptotics},
Lemma~\ref{s-root-reciprocal-estimate} thus gives
\[
  (2k\pi)^{4}\Om_{nk}^{(4)}
  =
  \frac{\ii}{\pi^{2}}
  \frac{n}{k}
  \frac{2}{n^{2}-k^{2}}
  \p*{
  \int_{G_{k}^{-}} (\sg_{k}^{n}-\lm)\vs_{k}^{3}(\lm) \,\dlm
   + \gm_{k}^{5}k^{-s-1}\ell_{k}^{p}}.
\]
A straightforward computation using~\eqref{s-root-sides} further shows
\begin{align*}
  &\int_{G_{k}^{-}} \vs_{k}^{3}(\lm)\,\dlm
  = -\ii \int_{-1}^{1} \p*{\frac{\gm_{k}}{2}}^{4}\sqrt[+]{1-t^{2}}^{3}\,\dt
   = -\ii\frac{3\pi}{128}\gm_{k}^{4},\\
  &\int_{G_{k}^{-}} (\tau_{k}-\lm)\vs_{k}^{3}(\lm)\,\dlm
  = \ii \int_{-1}^{1} \p*{\frac{\gm_{k}}{2}}^{5}t\sqrt[+]{1-t^{2}}^{3}\,\dt
  = 0,
\end{align*}
where the latter integral is zero since the integrand is an odd function of $t$.
Writing $\sg_{k}^{n}-\lm = (\sg_{k}^{n}-\tau_{k}) + (\tau_{k}-\lm)$ and using $\sg_{k}^{n}-\tau_{k} = \gm_{k}k^{-s-1}\ell_{k}^{p}$, which is deduced from~\eqref{psi-root-asymptotics} and~\eqref{taun-lmn-est}, we hence obtain
\[
  \int_{G_{k}^{-}} (\sg_{k}^{n}-\lm)\vs_{k}^{3}(\lm) \,\dlm = 
  -\ii\frac{3\pi}{128}\gm_{n}^{4}(\sg_{k}^{n}-\tau_{k})
  =
  \gm_{n}^{5}k^{-s-1}\ell_{k}^{p}.
\]
Altogether we arrive at
\[
  k\Om_{nk}^{(4)} =
  \frac{n}{n^{2}-k^{2}}k^{-s-5}\gm_{k}^{5}\ell_{k}^{p}.
\]

If $k=n$, then by~\eqref{psi-c-root-quot} and Proposition~\eqref{std-prd-est}
\[
  \frac{\ps_{n}(\lm)}{\sqrt[c]{\Dl^{2}(\lm)-4}}
   = \frac{\ii}{\vs_{n}(\lm)}\zt_{n}(\lm),
  \qquad
  \left.\zt_{n}(\lm)\right|_{G_{n}} = 1 + n^{-1-s}\ell_{n}^{p}.
\]
Thus, uniformly for $\lm\in G_{n}^{-}$,
\[
  \frac{F_{n}^{4}(\lm)\psi_{n}(\lm)}{\sqrt[c]{\Dl^{2}(\lm)-4}}
  = \frac{ \p*{\vs_{n}^{4}(\lm) + \gm_{n}^{4}n^{-s-1}\ell_{n}^{p}}
           \p*{\ii + n^{-s-1}\ell_{n}^{p}}}{(2n\pi)^{4}\vs_{n}(\lm)},
\]
and hence
\[
  (2n\pi)^{4}\Om_{nn}^{(4)}
   = \ii 2 \int_{G_{n}^{-}} \vs_{n}^{3}(\lm)\,\dlm
     + \gm_{n}^{4}n^{-s-1}\ell_{n}^{p}
   = \frac{3\pi}{64}\gm_{n}^{4}\p*{1
     + n^{-s-1}\ell_{n}^{p}},
\]
which establishes the claimed estimate in the case $k=n$.

Going through the arguments of the proofs one sees that the estimates hold locally uniformly on $\Ws^{s,p}$.\qed
\end{proof}

Our first main result for the KdV2 frequencies establishes the following formula for their analytic extensions.

\begin{thm}
\label{om-2-n-analytic}
For any $n\ge 1$ and any $(s,p)$ admissible with $s>-1$, the sum $-160\pi^{2}\sum_{k\ge 1} k^{3} \Om_{nk}^{(2)}
  + 80\sum_{k\ge 1} k\Om_{nk}^{(4)}$ converges locally uniformly on $\Ws^{s,p}$ to the analytic function $\om_{n}^{(2)\star}$,
\[
  \om_{n}^{(2)\star} = -160\pi^{2}\sum_{k\ge 1} k^{3} \Om_{nk}^{(2)}
  + 80\sum_{k\ge 1} k\Om_{nk}^{(4)}.~\fish
\]
\end{thm}

\begin{rem}
\label{rem:om-2-n-analytic}
Arguing as in the proof of~\cite[Theorem~20.3]{Grebert:2014iq}, one sees that for any $n\ge 1$ and $(s,p)$ admissible with $s> -1$, the frequency $\om_{n}^{(2)\star}$ is a real analytic function of the actions on the complex neighborhood $\Vs^{2s+1,p/2}$ of $\ell_{+}^{2s+1,p/2}$ introduced in~\eqref{Vsp}.\map
\end{rem}

\begin{proof}
By Lemma~\ref{decay-kRnk} and Lemma~\ref{decay-kOm2nk} we have for any $(s,p)$ admissible with $-1< s\le 0$, $\rho = 0$, and any $k\neq n$
\[
  k^{3}\Om_{nk}^{(2)} = \frac{n^{-s+1}}{n^{2}-k^{2}}k^{-3s-1}\ell_{k}^{p/4},
  \qquad
  k\Om_{nk}^{(4)} = \frac{n}{n^{2}-k^{2}}k^{-6s-5}\ell_{k}^{p/4},
\]
locally uniformly on $\Ws^{s,p}$. Moreover, the moments $\Om_{nk}^{(2)}$ and $\Om_{nk}^{(4)}$ are analytic functions on $\Ws$ for any $n,k\ge 1$ by Lemma~\ref{Om-nk-prop} (ii).
Suppose $p=2$, then for $k\neq n$,
\[
  k^{3}\Om_{nk}^{(2)} = \frac{n^{-s+1}}{n^{2}-k^{2}}k^{-3s-1}\ell_{k}^{1}
  \le \frac{n^{2}k^{2}}{n^{2}-k^{2}}\ell_{k}^{1}
  = n^{2}\p*{1 + \frac{n^{2}}{n^{2}-k^{2}}}\ell_{k}^{1}
  = n^{3}\ell_{k}^{1}.
\]
Thus for $-1 < s \le 0$ and every fixed $n\ge 1$, the sum $\sum_{k\ge 1} k^{3}\Om_{nk}^{(2)}$ converges absolutely and locally uniformly to a real analytic function on $\Ws^{s,2}$.
Moreover, if $s = -1/2$ and $2 \le p < \infty$ arbitrary, then for $k\neq n$
\[
  k^{3}\Om_{nk}^{(2)} = \frac{nk}{n^{2}-k^{2}}\ell_{k}^{p/4}
  =
  \frac{n}{n-k}\ell_{k}^{p/4},
\]
hence by Hölder's inequality, the sum $\sum_{k\ge 1} k^{3}\Om_{nk}^{(2)}$  converges absolutely and locally uniformly to a real analytic function on $\Ws^{s,p}$. Finally, if $s = -1$ and $p=2$, then for $k\neq n$
\[
  k\Om_{nk}^{(4)} = \frac{nk}{n^{2}-k^{2}}\ell_{k}^{1}
  = \frac{n}{n-k}\ell_{k}^{1}.
\]
Thus for every fixed $n$, the sum $\sum_{k\ge 1} k\Om_{nk}^{(4)}$ converges absolutely and locally uniformly to a real analytic function on $\Ws^{-1,2}$. Altogether this shows that the functional 
$\Om_{n}^{(4)} \defl -160\pi^{2}\sum_{k\ge 1} k^{3} \Om_{nk}^{(2)} + 80\sum_{k\ge 1} k\Om_{nk}^{(4)}$
is real analytic on $\Ws^{s,p}$ for any $(s,p)$ admissible with $s> -1$.
Furthermore, $\Om_{n}^{(4)}$ coincides with $\om_{n}^{(2)\star}$ at every finite-gap potential by Lemma~\ref{om-n-2-star-formula}. Thus for any $n\ge 1$, $\Om_{n}^{(4)}$ is the unique real analytic extension of $\om_{n}^{(2)\star}$ to $\Ws^{s,p}$.\qed
\end{proof}

Our second main result for the KdV2 frequencies establishes the following asymptotics for the frequency map $\om^{(2)\star} = (\om_{n}^{(2)\star})_{n\ge 1}$.

\begin{thm}
\label{kdv2-freq-as}
\begin{equivenum}
\item
The map $\om^{(2)\star}$ admits a real analytic extension
\[
  \om^{(2)\star}\colon \begin{cases}
  \Hs_{0}^{s} \to \ell^{2s-1,1}, & -1< s \le 0,\\
  \FL_{0}^{s,p} \to \ell^{2s-1,p/2}, & -1/2\le s\le 0,\quad 2 < p < \infty\\
  \Hs_{0}^{s} \to \ell^{2s-1,1}, & \phantom{-}0 < s < 1/2,\\
  \Hs_{0}^{1/2} \to \ell^{r}, & \phantom{-}r > 1.
  \end{cases}
\]

\item
For any $(s,p)$ admissible with $-1< s\le 0$,
$\om_{n}^{(2)\star} = \ell_{n}^{2s-1,p/2}$ locally uniformly on $\Ws^{s,p}$. More precisely,
\[
  \om_{n}^{(2)\star} + 80 n^{2}\pi^{2} I_{n}
  =
  \begin{cases}
  n^{-3s}(\ell_{n}^{p/3} + \ell_{n}^{1}), & -1< s < 0,\\
  \ell_{n}^{p/3} + \ell_{n}^{1+}, & s = 0.~\fish
  \end{cases}
\]
\end{equivenum}
\end{thm}

\begin{rem}
\label{rem:kdv2-freq-as}
(i)
By Remark~\ref{rem:om-n-analytic} one sees that the asymptotics of $\om_{n}^{(2)\star}$, $n\ge1$, viewed as real analytic functions of the actions on $\Vs^{2s+1,p/2}$, are the same as the ones stated in Theorem~\ref{kdv2-freq-as} (ii) for $\om_{n}^{(2)\star}$ on $\Ws^{s,p}$.
In particular, $\om^{(2)\star}$, viewed as a function of the actions, is a real analytic map
\[
  \om^{(2)\star}\colon \begin{cases}
  \ell^{2s+1,1} \to \ell^{2s-1,1}, & -1< s \le 0,\\
  \ell^{2s+1,p/2} \to \ell^{2s-1,p/2}, & -1/2\le s\le 0,\quad 2 < p < \infty\\
  \ell^{2s+1,1} \to \ell^{2s-1,1}, & \phantom{-}0 < s < 1/2,\\
  \ell^{3/2,1} \to \ell^{r}, & \phantom{-}r > 1.
  \end{cases}
\]

(ii) 
Suppose $u\in\Hs_{c}^{s}$ with $c$ an arbitrary real number. Write $u = c + q$ with $q\in\Hs_{0}^{s}$, then
\begin{align*}
  \om_{n}^{(2)}(u) &= \om_{n}^{(2)}(q) + 10c\om_{n}^{(1)}(q) + 60n\pi c^{2}\\
  &= (2n\pi)^{5} + 10(2n\pi)^{3} c + 20(2n\pi)\Hm_{0}(q) + 60n\pi c^{2}\\
  &\qquad + \om_{n}^{(2)\star}(q) + 10c\om_{n}^{(1)\star}(q).
\end{align*}
Since by Theorem~\ref{kdv-freq-as} (ii) one has $\om_{n}^{(1)\star} = n^{-2s-1}\ell_{n}^{1}$ for $-1< s < -1/2$ and $p=2$, and $\om_{n}^{(1)\star} = \ell_{n}^{1+}$ for $s=-1/2$ and $2\le p < \infty$, we conclude that for any $c\in\R$, the asymptotic estimate stated in Theorem~\ref{kdv2-freq-as} (ii) holds for $\om_{n}^{(2)\star} + 10c\om_{n}^{(1)\star}$ as well.\map
\end{rem}

\begin{proof}
Let us first prove (ii).
By Theorem~\ref{om-2-n-analytic} we have for any $(s,p)$ admissible with $s> -1$ and any $n\ge 1$, $\om_{n}^{(2)\star} = -160\pi^{2}\sum_{k\ge 1} k^{3} \Om_{nk}^{(2)} + 80\sum_{k\ge 1} k\Om_{nk}^{(4)}$.
We consider the asymptotics of the terms $\sum_{k\ge 1} k^{3} \Om_{nk}^{(2)}$ and  $\sum_{k\ge 1} k\Om_{nk}^{(4)}$ separately.

Suppose $-1 < s < -1/2$ and $p=2$, then Lemma~\ref{decay-kRnk} and the asymptotics~\eqref{taun-lmn-est}, $\gm_{k} = k^{-s}\ell_{k}^{p}$, yields for $k\neq n$ with $\rho = 0$
\[
  k^{3}\Om_{nk}^{(2)} = n^{1-s}\frac{k^{-3s-1}\ell_{k}^{1}}{n^{2}-k^{2}}.
\]
Since in this case $-2 \le 3s+1 < -1/2$, we get with Lemma~\ref{ht-2-x}
\[
  \sum_{k\neq n} k^{3}\Om_{nk}^{(2)} = n^{-4s-1}\ell_{n}^{1+}.
\]
Now consider the case $-1/2\le s \le 0$ and $2\le p < \infty$. By Lemma~\ref{decay-kRnk} we obtain for $k\neq n$ with $\rho = 1$
\[
  k^{3}\Om_{nk}^{(2)} = n\frac{k^{-4s-1}\ell_{k}^{p/4}}{n^{2}-k^{2}}.
\]
If $-1/2 \le s \le -1/4$, then $-1\le 4s+1 \le 0$ and Lemma~\ref{ht-2} gives
\[
  \sum_{k\neq n} k^{3}\Om_{nk}^{(2)} = n^{-4s-1}(\ell_{n}^{p/4}+\ell_{n}^{1+}).
\]
Moreover, if $-1/4 < s < 0$ and $2\le p < \infty$, then $0 < 4s+1$ and Lemma~\ref{ht-2} yields
\[
  \sum_{k\neq n} k^{3}\Om_{nk}^{(2)} = \ell_{n}^{p/4}+\ell_{n}^{1+}.
\]
Finally, if $s = 0$ and $2\le p < \infty$, then $k^{-4s-1}\ell_{k}^{p/4} = \ell_{k}^{1}$ and we conclude with Lemma~\ref{ht-2} that
\[
  \sum_{k\neq n} k^{3}\Om_{nk}^{(2)} = \ell_{n}^{1+}.
\]
For $k=n$ and $(s,p)$ admissible with $-1 < s\le 0$, Lemma~\ref{decay-kRnk} combined with the asymptotics~\eqref{taun-lmn-est} of $\gm_{n}$ gives
\begin{align*}
  n^{3}\Om_{nn}^{(2)}
   = \frac{n\gm_{n}^{2}}{16\pi}\p*{1+n^{-s-1}\ell_{n}^{p}}
   = \frac{n\gm_{n}^{2}}{16\pi}+n^{-3s}\ell_{n}^{p/3}
   = n^{-2s+1}\ell_{n}^{p/2}.
\end{align*}
Altogether, we thus have for any $(s,p)$ admissible with $-1< s\le 0$, using that for $-1 < s < 0$ one has $-3s > (-4s-1)_{+}$ and hence 
\begin{equation}
  \label{Om-nk-2-est-sum}
  \sum_{k\ge 1} k^{3}\Om_{nk}^{(2)} =
  \frac{n\gm_{n}^{2}}{16\pi} + 
  \begin{cases}
  n^{-3s}(\ell_{n}^{p/3} + \ell_{n}^{1}), & -1 < s < 0,\\
  \ell_{n}^{p/3} + \ell_{n}^{1+}, & s = 0.
  \end{cases}
\end{equation}

Next we consider the term $\sum_{k\ge 1} k\Om_{nk}^{(4)}$. By Lemma~\ref{decay-kOm2nk} and the asymptotics~\eqref{taun-lmn-est} of~$\gm_{n}$, we have for $k\neq n$,
\[
  k\Om_{nk}^{(4)} =
  n\frac{k^{-6s-5}\ell_{k}^{p/6}}{n^{2}-k^{2}}.
\]
If $-1 \le s < -1/2$ and $p=2$, then $-1\le 6s+5 \le 0$ and hence by Lemma~\ref{ht-2}
\begin{equation}
  \label{Omnk-4-est-1}
  \sum_{k\neq n}k\Om_{nk}^{(4)}
  = n^{(-6s-5)_{+}}\ell_{n}^{1+}.
\end{equation}
If $s=-1/2$, then $-6s-5 = -2$ and hence $k^{-6s-5}\ell_{k}^{p/6} = \ell_{k}^{1}$, so~\eqref{Omnk-4-est-1} is also valid for $-1/2\le s\le 0$ and $2 \le p < \infty$.

For the case $k=n$ we further obtain with Lemma~\ref{decay-kOm2nk} and the estimate of $\gm_{n}$,
\begin{align*}
  n\Om_{nn}^{(4)}
   = \frac{3}{2^{10} n^{3}\pi^{3}}\gm_{n}^{4}\p*{1 + n^{-s-1}\ell_{n}^{p}}
   = n^{-4s-3}\ell_{n}^{p/4}.
\end{align*}
Therefore, for any $(s,p)$ admissible with $s\ge -1$,
\begin{equation}
  \label{Om-nk-4-est-sum}
  \sum_{k\ge 1} k\Om_{nk}^{(4)} = n^{(-6s-5)_{+}}\ell_{n}^{1+}.
\end{equation}

By~\eqref{In-gmn-est} we have for any $(s,p)$ admissible with $-1\le s\le 0$
\begin{equation}
  \label{gmn2-In-asymptotics-2}
  \frac{\gm_{n}^{2}}{8n\pi} = I_{n} + n^{-3s-2}(\ell_{n}^{p/4}+\ell_{n}^{1}).
\end{equation}

Estimates~\eqref{Om-nk-2-est-sum}, \eqref{Om-nk-4-est-sum}, and~\eqref{gmn2-In-asymptotics-2} give for $(s,p)$ admissible with $-1< s\le 0$
\begin{equation}
  \label{om-2-est-temp}
  \om_{n}^{(2)\star} 
  = 
  -80n^{2}\pi^{2}I_{n}
  +
  \begin{cases}
  n^{-3s}(\ell_{n}^{p/3}+\ell_{n}^{1}), & -1< s < 0,\\
  \ell_{n}^{p/3}+\ell_{n}^{1+}, & s = 0.
  \end{cases}
\end{equation}
By going through the arguments of the proof, one sees that the estimates hold locally uniformly on $\Ws^{s,p}$.

We now prove (i). By combining the estimate~\eqref{om-2-est-temp} of $\om_{n}^{(2)\star}+80n^{2}\pi^{2}I_{n}$ with the estimate~\eqref{taun-lmn-est}, $\gm_{n}=n^{-s}\ell_{n}^{p}$, and noting that $-2s+1 > -3s$ for $-1< s\le 0$, one obtains for $(s,p)$ admissible with $-1< s\le 0$
\[
  \om_{n}^{(2)\star} = 
  n^{-2s+1}\ell_{n}^{p/2}.
\]
Moreover, for $s=0$ and $p=2$ one has by~\eqref{om-2-est-temp}, $\om_{n}^{(2)\star} + 80n^{2}\pi^{2}I_{n} = \ell_{n}^{1+}$, and by Lemma~\ref{In-lmd-est} that $I_{n} = \frac{\gm_{n}^{2}}{8n\pi}(1+n^{-1}\ell_{n}^{1+})$.
Since by~\eqref{taun-lmn-est} for $p=2$ and any $s\ge 0$ one has $\gm_{n} = n^{-2s+1}\ell_{n}^{1}$, we conclude 
\[
  \om_{n}^{(2)\star} = 
  \begin{cases}
  n^{-2s+1}\ell_{n}^{1}, & 0 < s < 1/2,\quad p=2,\\
  \ell_{n}^{1+}, & s=1/2,\quad p=2.
  \end{cases}
\]
All estimates hold locally uniformly on $\Ws^{s,p}$ for $(s,p)$ admissible. Since each $\om_{n}^{(2)\star}$, $n\ge 1$, is analytic on $\Ws^{s,p}$ by Lemma~\ref{om-2-n-analytic}, the claimed analyticity of $\om^{(2)\star}$ thus follows.~\qed
\end{proof}

Corollary~\ref{cor:kdv2-freq-localdiffeo} follows from the following more detailed statement.

\begin{cor}
For any $-1< s < 1/2$,
\begin{equivenum}

\item
$\om^{(2)\star}\colon \ell_{+}^{2s+1,1}\to \ell^{2s-1,1}$ is a local diffeomorphism near $I=0$.

\item
For any $I\in\ell_{+}^{2s+1,1}$, the linear operator
\[
  \Om_{I} = \ddd_{I}\om^{(2)\star} + 20\diag((4n^{2}\pi^{2})_{n\ge 1})\colon \ell^{2s+1,1}\to \ell^{2s-1,1}
\]
is compact.
\item
$\om^{(2)\star}\colon \ell_{+}^{2s+1,1}\to \ell^{2s-1,1}$ is a Fredholm map of index zero everywhere.
\item
$\om^{(2)\star}\colon \ell_{+}^{2s+1,1}\to \ell^{2s-1,1}$ is a local diffeomorphism on an open and dense subset of $\ell_{+}^{2s+1,1}$.~\fish
\end{equivenum}
\end{cor}

\begin{proof}
(i): By Theorem~\ref{kdv2-freq-as} and Remark~\ref{rem:kdv2-freq-as}, the map $\om^{(2)\star}\colon \Vs^{2s+1,1}\to \ell_{\C}^{2s-1,1}$ is analytic for any $-1< s < 1/2$, where the neighborhood $\Vs^{2s+1,1}$ of $\ell_{+}^{2s+1,1}$ in $\ell_{\C}^{2s+1,1}$ was introduced in~\eqref{Vsp}. By Theorem~\ref{thm:bnf-kdv-kdv2} we have
\[
  \Lm\defl\ddd_{0}\om^{(2)\star} = \diag((-80n^{2}\pi^{2})_{n\ge 1}).
\]
Since $\Lm \colon \ell_{\C}^{2s+1,1}\to \ell_{\C}^{2s-1,1}$ is an isomorphism, it follows that $\om^{(2)\star}$ is a local diffeomorphism at $I=0$ for $-1< s < 1/2$.

(ii): Suppose $-1< s < 0$, then by Theorem~\ref{kdv2-freq-as} (ii) $\om_{n}^{(2)\star} + 20(2n\pi)^{2}I_{n} = n^{-3s}\ell_{n}^{1}$ locally uniformly on $\Vs^{2s+1,1}$, hence by Cauchy's estimate the map $\Om_{I}\colon \ell_{\C}^{2s+1,1}\to \ell_{\C}^{-3s,1}$ is bounded for any $I\in\Vs^{2s+1,1}$. Since $\ell_{\C}^{-3s,1}$ embeds compactly into $\ell_{\C}^{2s-1,1}$ for $-1< s < 0$, we conclude that $\Om_{I}\colon \ell_{\C}^{2s+1,1}\to \ell_{\C}^{2s-1,1}$ is compact.
If $0\le s < 1/2$, then by Theorem~\ref{kdv2-freq-as} (ii) $\om_{n}^{(2)\star} + 20(2n\pi)^{2}I_{n} = \ell_{n}^{1+}$ locally uniformly on $\Vs^{2s+1,1}$, hence by Cauchy's estimate the map $\Om_{I}\colon \ell_{\C}^{2s+1,1}\to \ell_{\C}^{r}$ is bounded for any $I\in\Vs^{s,1}$ and any $r > 1$. Since $\ell_{\C}^{r}$ embeds compactly into $\ell_{\C}^{2s-1,1}$ if $r > 1$ is chosen sufficiently small, we conclude that $\Om_{I}\colon \ell_{\C}^{2s+1,1}\to \ell_{\C}^{2s-1,1}$ is compact also in this case.

(iii): By item (ii), $\ddd_{I}\om^{(2)\star}$ is a compact perturbation of $\Lm$ and hence a Fredholm operator of index zero, meaning that $\om^{(2)\star}$ is a Fredholm map of index zero.

(iv): Consider the real analytic map $f = \Lm^{-1}\om^{(2)\star}\colon \Vs^{2s+1,1}\to \ell_{\C}^{2s+1,1}$. By item (iii) for any $I\in\Vs^{2s+1,1}$ the differential $\ddd_{I}f$ is a compact perturbation of the identity on $\ell_{\C}^{2s+1,1}$ if $-1< s < 1/2$. So Proposition~\ref{diffeo-prop} applies yielding that $f$ is a local diffeomorphism generically.\qed
\end{proof}

\begin{rem}
\label{rem:kdv2-freq-localdiffeo}
(i) Since by Remark~\ref{rem:kdv2-freq-as}, $\om^{(2)\star}\colon \Vs^{2s+1,p/2}\to \ell_{\C}^{2s-1,p/2}$ is analytic for $(s,p)$ admissible with $-1< s < 1/2$, and $\ddd_{0}\om^{(2)\star} = \Lm$ by Theorem~\ref{thm:bnf-kdv-kdv2}, it follows that $\om^{(2)\star}$ is a local diffeomorphism at $I=0$ also for $-1/2\le s \le 0$ and $p > 2$.

(ii): 
For $-1/2\le s \le 0$ and $2<p<\infty$, we have by Theorem~\ref{kdv2-freq-as} (ii),
\[
  \om_{n}^{(2)\star} + 80n^{2}\pi^{2}I_{n} = 
  \begin{cases}
  n^{-3s}(\ell_{n}^{p/3} +\ell_{n}^{1}), & -1/2 \le  s < 0,\quad 2 < p  < \infty,\\
  \ell_{n}^{p/3} +\ell_{n}^{1+}, & \phantom{-1/2 \le  }\,s = 0,\quad 2 < p  < \infty.
  \end{cases}
\]
Therefore,
\[
  \Om_{I}\colon \ell_{\C}^{2s+1,p/2} \to
  \begin{cases}
  \ell_{\C}^{-3s,1}, & -1/2\le s < 0,\quad 2 < p \le 3,\\
  \ell_{\C}^{-3s,p/3}, & -1/2\le s < 0,\quad  3 < p  < \infty,\\
  \ell_{\C}^{r}, & r > 1,\;\;s = 0,\quad 2 < p \le 3,\\
  \ell_{\C}^{p/3}, & \phantom{-1/2 \le  } s = 0,\quad  3 < p  < \infty,
  \end{cases}
\]
is bounded. In all four cases, the range embeds compactly into $\ell^{2s-1,p/2}$, hence $\Om_{I}\colon \ell_{\C}^{2s+1,p/2} \to \ell_{\C}^{2s-1,p/2}$ is compact and $\ddd_{I}\om^{(2)\star}$ is a compact perturbation of $\Lm$.

(iii): The map $f = \Lm^{-1}\om^{(2)\star}\colon \Vs^{2s+1,p/2}\to \ell_{\C}^{2s+1,p/2}$ is analytic and for any $I\in \Vs^{2s+1,p/2}$ its differential is a compact perturbation of the identity for any $-1/2\le s \le 0$ and $p>2$. Proposition~\ref{diffeo-prop} yields that $f$ is a local diffeomorphism generically.~\map
\end{rem}

\begin{lem}
\label{om-2-unif}
\begin{equivenum}
\item
For any $r > 1$, the map $\om^{(2)\star}\colon \ell_{+}^{2,1} \to \ell^{r}$ is uniformly continuous on bounded subsets.
\item
For any $r > 1$, the map $\om^{(2)\star}\colon \Hs_{0}^{1/2} \to \ell^{r}$ is uniformly continuous on bounded subsets.~\fish
\end{equivenum}

\end{lem}

\begin{proof}
(i) By Theorem~\ref{kdv2-freq-as} (ii), $\om_{n}^{(2)\star} + 80n^{2}\pi^{2}I_{n} = \ell_{n}^{1+}$ on $\ell_{+}^{1,1}$. Recall that $\Lm = \diag((-80n^{2}\pi^{2})_{n\ge 1})$, hence the map
\[
  \om^{(2)\star} - \Lm\colon \ell_{+}^{1,1}\to \ell^{r}
\]
is real analytic for any $r > 1$ and thus uniformly continuous on compacts. Since $\ell_{+}^{2,1}$ embeds compactly into $\ell_{+}^{1,1}$, and $\Lm\big|_{\ell_{+}^{2,1}}\colon \ell_{+}^{2,1}\to \ell_{+}^{1}$ is Lipschitz continuous, we conclude that
$\om^{(2)\star} \colon \ell_{+}^{2,1}\to \ell^{r}$ is uniformly continuous on bounded subsets of $\ell_{+}^{2,1}$ for any $r > 1$.

(ii) Since the Birkhoff map $\Ph\colon \Hs_{0}^{1/2}\to \ell^{2,1}$ is uniformly continuous on bounded subsets of $\Hs_{0}^{1/2}$ by Lemma~\ref{Om-unif-continous}, the claim follows immediately with (i).\qed
\end{proof}

\subsection{Hamiltonian}
\label{ss:kdv2-hamiltonian}

In this section, we derive in analogous fashion to \cite{Kappeler:CNzeErmy} a formula for the renormalized KdV2 Hamiltonian
\[
  \Hm_{2}^{\star} = \Hm_{2} - 10 \Hm_{0}^{2} - \sum_{n\ge 1} (2n\pi)^{5} I_{n}.
\]
This formula will allow us to extend the latter, when written as a function of the actions, from $\ell_{+}^{3,1}$ to $h_{+}^{1}$.

For convenience we introduce for any integer $n\ge 1$ and $m\ge 1$ the (conveniently normalized) moments
\[
  R_{n}^{(m)} = -\frac{1}{\pi}\int_{\Gm_{n}} F_{n}^{m}(\lm)\,\dlm.
\]

\begin{lem}
\label{Rnk-prop}
\begin{equivenum}
\item $R_{n}^{(1)} = I_{n}$ for any $n\ge 1$.
\item Each moment $R_{n}^{(m)}$, $n\ge1$, $m\ge 1$, is real analytic on $\Ws$.
\item $R_{n}^{(2l)} = 0$ for all $n\ge 1$ and $l\ge 0$.
\item $R_{n}^{(m)} = O(\gm_{n}^{m+1}/n^{m})$ locally uniformly on $\Ws$ and uniformly as $n\to \infty$. In particular, $R_{n}^{(m)}$ vanishes if $\gm_{n}$ vanishes.
\item On $\Hs_{0}^{-1}$, $R_{n}^{(m)} \ge 0$ and $R_{n}^{(m)}$ vanishes if and only if $\gm_{n}$ vanishes.\fish
\end{equivenum}
\end{lem}

\begin{proof}
(i) follows from~\eqref{action} and integration by parts.

(ii) Arguing as in the proof of Lemma~\ref{Om-nk-prop} (ii), one sees that each moment $R_{n}^{(m)}$ is analytic on $\Ws$.

(iii) Since in view of Lemma~\ref{F2-analytic} (i) any even power of $F_{n}$ is analytic on $U_{n}$, the moments $R_{n}^{(2l)}$ vanish.

(iv) If $\gm_{n} = 0$, then $F_{n}$ is analytic on $U_{n}$ by Lemma~\ref{F-prop} (i). On the other hand, if $\gm_{n}\neq 0$, then by Lemma~\ref{F-prop} (iii) we have $\abs{F_{n}}_{G_{n}} = O(\gm_{n}/n)$ and hence $R_{n}^{(m)} = O(\gm_{n}^{m+1}/n^{m})$.

(v) If $q$ is real, then $G_{n}\subset\R$ and $F_{n}(\lm)\big|_{G_{n}^{\pm}} = \mp\cosh^{-1}((-1)^{n}\Dl(\lm)/2)$ by \cite[Lemma~2.2]{Molnar:2016hq}. Thus $R_{n}^{(m)}$ is real.
Since $(-1)^{n}\Dl(\lm)> 2$ for $\lm_{n}^{-} < \lm < \lm_{n}^{+}$, $R_{n}^{(m)}$ vanishes if and only if $G_{n}$ is a single point.\qed
\end{proof}

\begin{lem}
\label{H2-star-formula}
For any real-valued finite-gap potential with $[q] = 0$
\[
  \Hm_{2}^{\star} = \sum_{n\ge 1} \p*{
    - \frac{40}{3}(2n\pi)^{3}R_{n}^{(3)} + 16 (2n\pi)R_{n}^{(5)}
  }.~\fish
\]
\end{lem}

\begin{proof}
Suppose $q$ is a finite-gap potential, meaning that $S = \setdef{n\ge \N}{\gm_{n}(q) \neq 0}$ is finite. By Lemma~\ref{F2-analytic} (ii), the function $F^{2}$ is analytic outside a sufficiently large circle $C_{r}$ which encloses all open gaps $G_{n}$, $n\in S$, and whose exterior contains $G_{0}$. According to~\eqref{F-exp}, the function $F$ admits an asymptotic expansion for $\nu_{n} = (n+1/2)\pi$. In particular,
\[
  F(\lm)^{6} = -\lm^{3} + \frac{3}{2}\Hm_{0}\lm + \frac{3}{8}\Hm_{1} + \frac{3}{32}\p*{\Hm_{2} - 10\Hm_{0}^{2}}\frac{1}{\lm} + O(\lm^{-2}),
\]
hence by Cauchy's Theorem
\[
  \frac{1}{2\pi\ii} \int_{C_{r}} F(\lm)^{6}\,\dlm = \frac{3}{32}(\Hm_{2} - 10 \Hm_{0}^{2}).
\]
Since $F^{2}$ is analytic on $\C\setminus\bigcup_{n\in S} G_{n}$, we obtain by contour deformation
\[
  \frac{1}{2\pi\ii} \int_{C_{r}} F(\lm)^{6}\,\dlm
   =
  \sum_{n\in S} \frac{1}{2\pi\ii} \int_{\Gm_{n}} F^{6}(\lm)\,\dlm.
\]
Expanding $F(\lm)^{6} = (F_{n}-\ii n\pi)^{6}$ into
\[
  F_{n}^{6} - \ii 6 n\pi F_{n}^{5} - 15n^{2}\pi^{2}F_{n}^{4}
  + \ii 20n^{3}\pi^{3}F_{n}^{3} + 15 n^{4}\pi^{4}F_{n}^{2}
  - \ii 6 n^{5}\pi^{5} F_{n} - n^{6}\pi^{6}
\]
and using that by Lemma~\ref{Rnk-prop} (iii) $R_{n}^{(2)}$ and $R_{n}^{(4)}$ vanish for all $n\in\N$ yields
\begin{align*}
  \Hm_{2} - 10 \Hm_{0}^{2}
  &=
  \frac{32}{3}
  \sum_{n\in S}
  \p*{
  3n^{5}\pi^{5} R_{n}^{(1)} - 10 n^{3}\pi^{3} R_{n}^{(3)} + 3 n\pi R_{n}^{(5)}
  }
  \\
  &= \sum_{n\ge 1} \p*{(2n\pi)^{5} I_{n} - \frac{40}{3} (2n\pi)^{3} R_{n}^{(3)} + 16 (2n\pi)R_{n}^{(5)}}.
\end{align*}
Here, we used that by Lemma~\ref{Rnk-prop} (i) $R_{n}^{(1)} = I_{n}$, $n\ge 1$, and that  by Lemma~\ref{Rnk-prop} (iv) $R_{n}^{(1)}$, $R_{n}^{(3)}$, and $R_{n}^{(5)}$ vanish for $n\in \N\setminus S$.\qed
\end{proof}

\begin{thm}
\label{Hm2-star-analytic}
The renormalized KdV2 Hamiltonian $\Hm_{2}^{\star}$, when written as a function of the actions, extends real analytically to $h_{+}^{1}$. Moreover, 
for any $I\in h_{+}^{1}$, the $\ell^{2}$-gradient $\partial_{I}\Hm_{2}^{\star}$ can be identified with $\om^{(2)\star}(I)$ and 
there exists a neighborhood $U$ of $I=0$ in $h_{+}^{1}$ on which $\Hm_{2}^{\star}$ is strictly concave in the sense that
\[
  \lin{\ddd_{I}\Hm_{2}^{\star}J,J}_{h^{-1},h^{1}} \le \frac{1}{2}\lin{\Lm J,J}_{h^{-1},h^{1}},\qquad \forall I\in U,\quad J\in h^{1},
\]
where $\Lm \defl \diag\p[\Big]{(-80n^{2}\pi^{2})_{n\ge 1}}$.~\fish
\end{thm}

\begin{proof}
By Lemma~\ref{Rnk-prop} (ii) all moments $R_{n}^{(3)}$, $R_{n}^{(5)}$, $n\ge 1$, are real analytic on $\Ws$. Moreover, by Lemma~\ref{Rnk-prop} (iv) we have uniformly in $n\ge 1$ and locally uniformly on $\Ws$,
\[
  R_{n}^{(2k+1)} = O(\gm_{n}^{2k+2}/n^{2k+1}).
\]
Hence by~\eqref{taun-lmn-est} for any $(s,p)$ admissible
\[
  n^{3}R_{n}^{(3)} = \gm_{n}^{4} = n^{-4s}\ell_{n}^{p/4},\qquad
  nR_{n}^{(5)} = \gm_{n}^{6}/n^{4} = n^{-4-4s}\ell_{n}^{p/4}.
\]
Therefore, on $\Ws^{0,4}$ the sum
\[
  \sum_{n\ge 1} \p*{- \frac{40}{3}(2n\pi)^{3}R_{n}^{(3)} + 16 (2n\pi)R_{n}^{(5)}}
\]
is absolutely and locally uniformly convergent to an analytic function $\tilde H$. Since $\tilde H = \Hm_{2}^{\star}$ at any real-valued finite-gap potential by Lemma~\ref{H2-star-formula}, and the finite-gap potentials are dense in $\FL_{0}^{0,4}$, $\tilde H$ is the unique analytic extension of $\Hm_{2}^{\star}$ to $\Ws^{0,4}$. Arguing as in the proof of \cite[Theorem~20.3]{Grebert:2014iq}, one sees that $\Hm_{2}^{\star}$, viewed as a function of the actions, is real analytic on the complex neighborhood $\Vs^{1,2}$ of $\ell_{+}^{1,2} = h_{+}^{1}$.

For $I\in h_{+}^{2}$ we have by definition $\om^{(2)\star}(I) = \partial_{I}H_{2}^{\star}$, and both sides of the identity admit real analytic extensions to $h_{+}^{1}$ hence the identity extends as well. Therefore, $\ddd_{I}^{2}\Hm_{2}^{\star} = \ddd_{I}\om^{(2)\star}$ and by Remark~\ref{rem:kdv2-freq-localdiffeo}, $\ddd_{0}\om^{(2)\star} = \Lm$. Thus, the strict concavity in a neighborhood of $I=0$ follows from continuous dependence of $\ddd_{I}\om^{(2)\star}$ on $I$.~\qed
\end{proof}

\begin{rem}
We expect that one can adapt the arguments of \cite{Bikbaev:1993jl} for $\Hm_{1}^{\star}$ to prove that $\Hm_{2}^{\star}$ is concave on $h_{+}^{1}$, which in view of the analyticity obtained in Theorem~\ref{Hm2-star-analytic} then proves $\ddd_{I}^{2}\Hm_{2}^{\star} \le 0$ on $h_{+}^{1}$.
Using the asymptotics of $\om^{(2)\star}$ stated in Theorem~\ref{kdv2-freq-as} for the case $s=0$ and $p=4$, one obtains by the same arguments as in the proof of Corollary~\ref{cor:kdv2-freq-localdiffeo} that $\om^{(2)\star}\colon \Vs^{1,2} \to h_{\C}^{1}$ is a local diffeomorphism generically and in turn that $\Hm_{2}^{\star}$ is strictly concave on an open and dense subset of $h_{+}^{1}$ in the sense that $\ddd_{I}^{2}\Hm_{2}^{\star} < 0$.\map
\end{rem}

\subsection{Wellposedness}

To discuss known results on the wellposedness of the KdV2 equation, we introduce for any $d>0$ the level sets of $\Hm_{0}$
\begin{align*}
  &\Ms_{0,d}^{s} \defl \setdef{q\in \Hs_{0}^{s}}{\Hm_{0}(q) = d},\\
  &m_{0,d}^{s} = \Ph(\Ms_{0,d}^{s}) = \setdef[\bigg]{z\in h_{0}^{s}}{\sum_{n\ge 1} (2n\pi) z_{n}z_{-n} = d}.
\end{align*}

According to~\cite{Battig:1997ek}, for any $d>0$, the KdV2 equation is globally $C^{\om}$-wellposed in $\Ms_{0,d}^{s}$ for any $s\ge 1$, i.e. for any $T > 0$ the solution map
\[
  \Sc\colon \Ms_{0,d}^{s}\to C^{0}([T,T],\Ms_{0,d}^{s})
\]
is real-analytic.


By Theorem~\ref{kdv2-freq-as} from the previous section, the frequencies
\begin{equation}
  \label{kdv2-freq}
  \om_{n}^{(2)} = (2n\pi)^{5} + 20(2n\pi)H_{0} + \om_{n}^{(2)\star}
\end{equation}
give rise to a flow $\Sc_{\Ph} \colon(t,z) \mapsto (\ph_{n}^{t}(z))_{n\in\Z}$ in Birkhoff coordinates on $h_{0}^{s+1/2}$ with coordinate functions
\begin{equation}
  \label{kdv2-coordinate-functions}
  \ph_{n}^{t}(z) = \e^{\ii \om_{n}^{(2)}(z) t}z_{n},\qquad n\in\Z.
\end{equation}
Here, the KdV2 frequencies are viewed as analytic functions of the Birkhoff coordinates and as such have been extended to the bi-infinite sequence $(\om_{n}^{(2)})_{n\in\Z}$ by setting
\[
  \om_{0}^{2}(z) = 0,\qquad \om_{-n}^{2}(z) = -\om_{n}^{(2)}(z),\quad n\ge 1.
\]
The KdV2 solution map is then given by
\begin{equation}
  \label{kdv2-soln-map}
  \Sc^{t} = \Ph^{-1}\circ \Sc_{\Ph}^{t}\circ \Ph.
\end{equation}
We first consider properties of the map $\Sc_{\Ph}$ corresponding to the ones of $\Sc$ claimed in Theorem~\ref{thm:kdv2-wp}.

\begin{thm}
\label{thm:wp-kdv2-x}
\begin{equivenum}
\item
Suppose $s\ge 0$. For any $z\in h_{0}^{s+1/2}$ the curve $\R \to h_{0}^{s+1/2}$, $t\mapsto \Sc_{\Ph}(t,z)$ is continuous. Moreover, for any $T > 0$ the map $\Sc_{\Ph}\colon h_{0}^{s+1/2}\to C([-T,T],h_{0}^{s+1/2})$ is continuous and has the group property. In particular, $\Sc_{\Ph}^{t}\colon h_{0}^{s+1/2}\to h_{0}^{s+1/2}$ for any $t\in\R$ is a homeomorphism.

\item
For any $s\ge 1/2$, $d> 0$, and $T>0$, the map $\Sc_{\Ph}\colon m_{0,d}^{s+1/2}\to C([-T,T],m_{0,d}^{s+1/2})$ is real-analytic and uniformly continuous on bounded subsets.

\item
In contrast, for any $s\ge 1/2$ and any $t > 0$, the map $\Sc_{\Ph}^{t}\colon h_{0}^{s+1/2}\to h_{0}^{s+1/2}$ is nowhere locally uniformly continuous.

\item
For any $0 \le s < 1/2$, $d\ge 0$, and $t >0$, the map $\Sc_{\Ph}^{t}\colon m_{0,d}^{s+1/2}\to m_{0,d}^{s+1/2}$ is nowhere locally uniformly continuous.

\item
For each $n\in\Z$ and $t > 0$, the coordinate function $h_{0}^{1/2}\to \C$, $z\mapsto (\Sc_{\Ph}^{t}(z))_{n}$ cannot be extended continuously to points $z\in h_{0}^{s+1/2}$ with $z_{n}\neq 0$ for any $-1< s < 0$.~\fish

\end{equivenum}

\end{thm}

\begin{proof}
(i)-(ii): We apply Theorem~\ref{thm:freq-flow}. In view of~\eqref{kdv2-freq}, the KdV2 frequencies $\om_{n}^{(2)}$ are well defined and continuous on $h_{0}^{s+1/2}$ for $s\ge 0$, hence Theorem~\ref{thm:freq-flow} (i)-(ii) apply proving the continuity of $\Sc_{\Ph}\colon h_{0}^{s+1/2}\to C([-T,T],h_{0}^{s+1/2})$. Moreover, for any fixed $d>0$ and any $s\ge 1/2$, the map $\om^{(2)\star}\colon m_{0,d}^{s+1/2}\to \ell^{\infty}$ is real analytic by Theorem~\ref{kdv2-freq-as}
and uniformly continuous on bounded subsets by Lemma~\ref{om-2-unif}. Therefore, the analyticity and uniform continuity of $\Sc_{\Ph}\colon m_{0,d}^{s+1/2}\to C([-T,T],m_{0,d}^{s+1/2})$ follow by Theorem~\ref{thm:freq-flow} (iii)-(iv).

(iii):
To simplify notation, put $\sg = s+1/2$ so that $\sg \ge 1$.
We show that for any nonempty open subset $U$ of $h_{0}^{\sg}$ and any $t > 0$ the map $\Sc_{\Ph}^{t}\big|_{U}\colon U\to h_{0}^{\sg}$ is not uniformly continuous.
 By Theorem~\ref{kdv2-freq-as} (ii) and Remark~\ref{rem:kdv2-freq-as}, after possibly shrinking $U$, there exists an integer $N_{\star}\ge 1$ so that
\begin{equation}
  \label{omn-2-wp-est-1}
  \om_{n}^{(2)}(z) = (2n\pi)^{5} + 40n\pi H_{0}(z) - 80n^{2}\pi^{2}z_{n}z_{-n} + r_{n}(z),
\end{equation}
where $z_{n}z_{-n} = I_{n}$ and $\sup_{\abs{n}\ge N_{\star}} \abs{r_{n}(z)} \le \pi/(4t)$ on all of $U$.

We show that there exist two sequences $p^{(m)}$ and $q^{(m)}$ in $U$ and a real number $\dl_{0} > 0$ so that
\[
  \n{p^{(m)}-q^{(m)}}_{h^{\sg}} \to 0,\qquad
  \n{\Sc_{\Ph}^{t}(p^{(m)})-\Sc_{\Ph}^{t}(q^{(m)})}_{h^{\sg}} \ge \dl_{0}.
\]

Take any $z^{\o}\in U$ so that there exists $N\ge N_{\star}$ with $z_{\pm n}^{\o} = 0$ for $n \ge N$. Let $n_{m} \defl 2^{m}$. For $\dl > 0$ and $m\ge 1$ with $n_{m} > N$, we define $p^{m,\dl}$, $q^{m,\dl}\in h_{0}^{\sg}$ by putting $p_{\pm n}^{m,\dl} = q_{\pm n}^{m,\dl} = z_{\pm n}^{\o}$ for $n < N$,
\[
  p_{\pm N}^{m,\dl} = \dl n_{m}^{-1/2}m^{1/2},\qquad
  q_{\pm N}^{m,\dl} = 0,
\]
and for $n > N$,
\[
  p_{\pm n}^{m,\dl} = q_{\pm n}^{m,\dl} =
  \begin{cases}
  \dl n_{m}^{-\sg}, & n=n_{m},\\
  0, & \text{otherwise}.
  \end{cases}
\]
Then $\Hm_{0}(p^{m,\dl}) - \Hm_{0}(q^{m,\dl}) = 2N\pi\dl^{2} n_{m}^{-1}m$, and
\[
  \frac{1}{\sqrt{2}}\n{p^{m,\dl}-z^{\o}}_{h^{\sg}}
  \le
  N^{\sg}\dl n_{m}^{-1/2}m^{1/2} + \dl
  = O(\dl),\qquad m\to \infty,
\]
while similarly,
\begin{align*}
  \frac{1}{\sqrt{2}}\n{q^{m,\dl}-z^{\o}}_{h^{\sg}} &= \dl.
\end{align*}
Therefore, we can fix $\dl_{0} \in (0,1)$ sufficiently small so that the sequences $(q^{m,\dl})$ and $(p^{m,\dl})$ are contained in $U$ for all $0 < \dl \le \dl_{0}$. Further note that
\begin{align*}
  &\frac{1}{\sqrt{2}}\n{p^{m,\dl}-q^{m,\dl}}_{h^{\sg}}
  \le
  N^{\sg}\dl_{0} n_{m}^{-1/2}m^{1/2} \to 0,\qquad m\to \infty.
\end{align*}
Now, by~\eqref{omn-2-wp-est-1} we have for any $m$ with $n_{m} > N$
\begin{align*}
  \om_{n_{m}}^{(2)}(p^{m,\dl}) - \om_{n_{m}}^{(2)}(q^{m,\dl})
  &= 40n_{m}\pi\p*{H_{0}(p^{m,\dl}) - H_{0}(q^{m,\dl})} + r_{n_{m}}(p^{m,\dl}) - r_{n_{m}}(q^{m,\dl})\\
  &= 80N\pi\dl^{2}m + r_{n_{m}}(p^{m,\dl}) - r_{n_{m}}(q^{m,\dl}).
\end{align*}
Choose $k\ge 1$ so that $\dl \equiv \dl(N,t,k) = 1/\sqrt{80Ntk} \le \dl_{0}$. Consequently,
\[
  \p*{\frac{m}{k}-\frac{1}{2}}\pi \le 
  \p*{\om_{n_{m}}^{(2)}(p^{m,\dl}) - \om_{n_{m}}^{(2)}(q^{m,\dl})}t
  \le
  \p*{\frac{m}{k}+\frac{1}{2}}\pi,\qquad n_{m} > N.
\]
With $m_{j} = k(2j+1)$ we conclude
\[
  \abs*{\exp\p*{\ii \p*{\om_{n_{m_{j}}}^{(2)}(p^{m_{j},\dl}) - \om_{n_{m_{j}}}^{(2)}(q^{m_{j},\dl})}t}-1} \ge 1,
\]
for all $\nu$ sufficiently large.
Thus, by comparing only the $n_{m_{\nu}}$th component we get
\begin{align*}
  \frac{1}{\sqrt{2}}\n{\Sc_{\Ph}^{t}(p^{m_{j},\dl})-\Sc_{\Ph}^{t}(q^{m_{j},\dl})}_{h^{\sg}}
  &\ge 
   n_{m_{j}}^{\sg}\abs{p_{n_{m_{j}}}^{m_{j},\dl}}
    - n_{m_{j}}^{\sg}\abs{p_{n_{m_{j}}}^{m_{j},\dl} - q_{n_{m_{j}}}^{m_{j},\dl}}
  = \dl,
\end{align*}
for all $j$ sufficiently large.

(iv):
To simplify notation, put $\sg = s+1/2$ so that $1/2\le \sg < 1$.
We show that for any nonempty open subset $U$ of $m_{0,d}^{\sg}$ and any $t > 0$ the map $\Sc_{\Ph}^{t}\big|_{U}\colon U\to m_{0,d}^{\sg}$ is not uniformly continuous.
By Theorem~\ref{kdv2-freq-as} (ii) and Remark~\ref{rem:kdv2-freq-as}, after possibly shrinking $U$, there exists an integer $N_{\star}\ge 1$ so that
\begin{equation}
  \label{omn-2-wp-est-2}
  \om_{n}^{(2)}(z) = (2n\pi)^{5} + 40n\pi d - 80n^{2}\pi^{2}z_{n}z_{-n} + r_{n}(z),
\end{equation}
where $z_{n}z_{-n} = I_{n}$ and $\sup_{\abs{n}\ge N_{\star}} \abs{r_{n}(z)} \le \pi/(4t)$ on all of $U$.

We show that there exist two sequences $p^{(m)}$ and $q^{(m)}$ in $U$ and a real number $\eta > 0$, so that
\[
  \n{p^{(m)}-q^{(m)}}_{h^{\sg}} \to 0,\qquad
  \n{\Sc_{\Ph}^{t}(p^{(m)})-\Sc_{\Ph}^{t}(q^{(m)})}_{h^{\sg}} \ge \eta.
\]

Take any $z^{\o}\in U$ so that there exists $N\ge N_{\star}$ with $z_{\pm N}^{\o}\neq 0$ and $z_{\pm n}^{\o} = 0$ for $n> N$. Let $\ep=\sqrt[+]{z_{N}^{\o}z_{-N}^{\o}}$ and $n_{m} = 2^{m}$.
For $\dl > 0$ and $m\ge 1$ with $n_{m} > N$, we define $p^{m,\dl}$, $q^{m,\dl}\in h_{0}^{\sg}$ by putting $p_{\pm n}^{m,\dl} = q_{\pm n}^{m,\dl} = z_{\pm n}^{\o}$ for $n < N$,
\[
  p_{\pm N}^{m,\dl} = \ep\sqrt{1 - \dl^{2} N^{-1}n_{m}^{1-2\sg}},\qquad
  q_{\pm N}^{m,\dl} = \ep\sqrt{1 - \dl^{2} N^{-1}( n_{m}^{1-2\sg}  + n_{m}^{-1}m)},
\]
where $0 < \dl < \ep$ is chosen so that the radicands are positive, and for $n > N$,
\[
  p_{\pm n}^{m,\dl} =
  \begin{cases}
  \dl\ep n_{m}^{-\sg}, & n=n_{m},\\
  0, & \text{otherwise},
  \end{cases}
  \qquad
  q_{\pm n}^{m,\dl} =
  \begin{cases}
  p_{\pm n_{m}}^{(m)} \pm \ii \dl\ep
  n_{m}^{-1}m^{1/2}, & n=n_{m},\\
  0, & \text{otherwise}.
  \end{cases}
\]
Then $H_{0}(z^{\o}) = H_{0}(p^{m,\dl}) = H_{0}(q^{m,\dl}) = d$, hence for any $\dl >0$, $p^{m,\dl}$ and $q^{m,\dl}$ are sequences contained in $m_{0,d}^{\sg}$. Moreover, using that $\abs{\sqrt{1-x}-1}\le \abs{x}$ if $\abs{x}\le 1/2$,
\[
  \frac{1}{\sqrt{2}\ep}\n{p^{m,\dl}-z^{\o}}_{h^{\sg}}
  \le
  N^{\sg}\abs*{ \sqrt{1-\dl^{2}N^{-1}n_{m}^{1-2\sg}} - 1}
  + \dl
  = O(\dl),
\]
Similarly,
\begin{align*}
  \frac{1}{\sqrt{2}\ep}\n{q^{m,\dl}-z^{\o}}_{h^{\sg}} &\le
  N^{\sg}\abs*{ \sqrt{1-\dl^{2} N^{-1}( n_{m}^{1-2\sg}  + n_{m}^{-1}m)} - 1}
  + \dl\\
  &\qquad + \dl n_{m}^{\sg-1}m^{1/2} = O(\dl).
\end{align*}
Therefore,  we can choose $\dl_{0} \in (0,\ep)$ so that the sequences $(q^{m,\dl})$ and $(p^{m,\dl})$ are contained in $U$ for any $0 < \dl < \dl_{0}$. Further note that
\begin{align*}
  &\frac{1}{\sqrt{2}\ep}\n{p^{m,\dl}-q^{m,\dl}}_{h^{\sg}}
  \le \dl n_{m}^{\sg-1}m^{1/2}\\
  &\qquad + N^{\sg}
  \abs*{\sqrt{1 - \dl^{2} N^{-1}n_{m}^{1-2\sg}} - \sqrt{1 - \dl^{2} N^{-1}( n_{m}^{1-2\sg}  + n_{m}^{-1}m)}},
\end{align*}
thus for all $m$ sufficiently large,
\[
  \frac{1}{\sqrt{2}\ep}\n{p^{m,\dl}-q^{m,\dl}}_{h^{\sg}}
  \le \dl_{0} n_{m}^{\sg-1}m^{1/2} + \dl_{0}^{2}n_{m}^{-1}m
  \to 0,\qquad m\to \infty.
\]

By~\eqref{omn-2-wp-est-2} we have for all $m\ge 1$ with $n_{m} > N$,
\[
  \om_{n_{m}}^{(2)}(p^{m,\dl}) - \om_{n_{m}}^{(2)}(q^{m,\dl})
   = 
  80\dl^{2}\ep^{2} \pi^{2}m + r_{n_{m}}(p^{m,\dl}) - r_{n_{m}}(q^{m,\dl})
\]
Choose $k\ge 1$ so that $\dl \equiv \dl(\ep,t,k) = 1/\sqrt{80\ep^{2}\pi tk} \le \dl_{0}$, then
\[
  \p*{\frac{m}{k}-\frac{1}{2}}\pi \le 
  \p*{\om_{n_{m}}^{(2)}(p^{m,\dl}) - \om_{n_{m}}^{(2)}(q^{m,\dl})}t
  \le
  \p*{\frac{m}{k}+\frac{1}{2}}\pi,\qquad n_{m} > N.
\]
With $m_{j} = k(2j+1)$ we conclude
\[
  \abs*{\exp\p*{\ii \p*{\om_{n_{m_{j}}}^{(2)}(p^{m_{j},\dl}) - \om_{n_{m_{j}}}^{(2)}(q^{m_{j},\dl})}t}-1} \ge 1,
\]
for all $\nu$ sufficiently large.
Thus, by comparing only the $n_{m_{j}}$th component, we get
\begin{align*}
  \frac{1}{\sqrt{2}}\n{\Sc_{\Ph}^{t}(p^{m_{j},\dl})-\Sc_{\Ph}^{t}(q^{m_{j},\dl})}_{h^{\sg}}
  &\ge n_{m_{j}}^{\sg}\abs{p_{n_{m_{j}}}^{m_{j},\dl}} - n_{m_{j}}^{\sg}\abs{p_{n_{m_{j}}}^{m_{j},\dl} - q_{n_{m_{j}}}^{m_{j},\dl}}\\
  &\ge \dl\ep - \n{p^{m_{j},\dl}-q^{m_{j},\dl}}_{h^{\sg}}\\
  &\ge \frac{1}{2}\dl\ep \defr \eta_{0},
\end{align*}
for all $j$ sufficiently large.

(v): Let $-1 < s < 0$ and take any initial datum $z\in h_{0}^{s+1/2}\setminus h_{0}^{1/2}$
with $z_{n}\neq 0$ for any given $n\ge 1$.
By Theorem~\ref{kdv2-freq-as},  $\tilde \om_{n}^{(2)} = \om_{n}^{(2)} - 40n\pi \Hm_{0}$ extends real analytically to $h_{0}^{s+1/2}$ for any $-1< s < 0$, whereas $\Hm_{0}(z) = \sum_{m\ge 1} (2m\pi)z_{m}^{2}$ is infinite on such $h_{0}^{s+1/2}$.\qed
\end{proof}

\begin{proof}[Proof of Theorem~\ref{thm:kdv2-wp}.]
Since by Theorem~\ref{bhf} and Lemma~\ref{Om-unif-continous} the Birkhoff map $\Ph\colon \Hs_{0}^{s}\to h_{0}^{s+1/2}$ and its inverse are both real-analytic for $s\ge 0$ and uniformly continuous on bounded subsets for $s > 0$, all claims of Theorem~\ref{thm:kdv2-wp} follow from Theorem~\ref{thm:wp-kdv2-x} and the identity $\Sc^{t} = \Ph^{-1}\circ\Sc_{\Ph}^{t}\circ\Ph$.\qed
\end{proof}

\begin{rem}
Consider the PDE with Hamiltonian
\[
  \tilde \Hm_{2} = \Hm_{2} - 10 \Hm_{0}^{2}.
\]
The frequencies of this integrable PDE are given by
\begin{align*}
  \tilde \om_{n}^{(2)}(u) &= (2n\pi)^{5}  + \om_{n}^{(2)\star}(q),
\end{align*}
where by Theorem~\ref{kdv2-freq-as}, each $\tilde \om_{n}^{(2)}(u)$ is real analytic on $\Hs^{s}$ with $s > -1$.
It follows from Theorem~\ref{thm:freq-flow} that this PDE is globally $C^{0}$-wellposed in $\Hs^{s}$ for $s>-1$.\map
\end{rem}

\appendix

\begin{appendix}

%
%
%
%
%

\section{\boldmath Discrete Hilbert transform on $\ell^{p}$}

In this appendix we recall some very well known facts on the discrete Hilbert Transform on $\ell_{\C}^{p}\equiv \ell^{p}(\N,\C)$ -- see e.g.~\cite{Zygmund:1957up}.

\begin{lem}
\label{ht-1}
For any $1 < p < \infty$ the mapping
\[
  H\colon (x_{m})_{m\in\Z} \mapsto \p*{\sum_{m\neq n} \frac{x_m}{m-n}}_{n\in\Z}
\]
defines a linear isomorphism on $\ell^{p}_{\C}(\Z,\C)$.~\fish
\end{lem}

As an immediate corollary one obtains the following result on the weighted $\ell^{p}$-spaces.

\begin{lem}
\label{ht-2}
For any $1 < p < \infty$ and any $-1\le s\le 0$ the map
\begin{equation}
  \label{A-op}
    A\colon \ell_{\C}^{s,p} \to \ell_{\C}^{s+1,p},\quad
    (x_{m})_{m\ge1}=x \mapsto 
  Ax = \p*{\sum_{m\neq n} \frac{x_{m}}{m^{2}-n^{2}}}_{n\ge 1}
\end{equation}
is bounded.~\fish
\end{lem}

\begin{proof}
First note that
\[
  \frac{1}{m^2-n^2}
   = \frac{1}{2m(m-n)}+\frac{1}{2m(m-(-n))}
   = \frac{1}{2n(m-n)}-\frac{1}{2n(m-(-n))}.
\]
Consequently,
\[
  (Ax)_{n} = \frac{1}{2n} (Hx)_{n}  - \frac{1}{2n} (Hx)_{-n},
\]
implying that $A\colon\ell_{\C}^{p}\to \ell_{\C}^{1,p}$ is bounded for any $1 < p < \infty$. Similarly,
\[
  (Ax)_{n} = \frac{1}{2} (H\tilde x)_{n} + \frac{1}{2} (H\tilde{x})_{-n},
\]
where $\tilde x_{m} = x_{m}/m$. Therefore, also $A\colon \ell_{\C}^{-1,p}\to\ell_{\C}^{p}$ is bounded for any $1 < p < \infty$. Interpolation then gives the claim for $-1 < s < 0$.\qed
\end{proof}

One easily checks that for $p=1$ the operator $A\colon \ell^{s,1}_{\C}\to \ell^{s+1,1}_{\C}$, introduced in~\eqref{A-op}, is \emph{unbounded} for any $-1\le s\le 0$. However, the following is still true.

\begin{lem}
\label{ht-2-x}
For any $-2\le s\le 0$ and $r > 1$ the map $A\colon \ell_{\C}^{s,1}\to \ell_{\C}^{s+1,r}$ is bounded.~\fish
\end{lem}

\begin{proof}
For $-1\le s\le 0$ the claim follows form Lemma~\ref{ht-2}. Now consider the case $s = -2$.
It is to show that $A\colon \ell_{\C}^{-2,1} \to \ell_{\C}^{-1,1}$ is bounded. Let $x_{m} = m^{2}\tilde x_{m}$ with $\tilde x_{m} = \ell_{m}^{1}$. Then for any $r>1$
\begin{align*}
  \sum_{m\neq n} \frac{m^{2}\tilde x_{m}}{m^{2}-n^{2}}
  =
  \sum_{m\neq n} \tilde x_{m}
  +
  n^{2}\sum_{m\neq n} \frac{\tilde x_{m}}{m^{2}-n^{2}}
  = \sum_{m\neq n} \tilde x_{m} + n\ell_{n}^{r}
  = n\ell_{n}^{r}.
\end{align*}
The case $-2 < s < -1$ follows by interpolation.\qed
\end{proof}

\begin{rem}
Lemma~\ref{ht-2-x} is optimal with respect to $s$ in the following sense: On the one hand, for $s < -2$, the sequence $\frac{x_{m}}{m^{2}-n^{2}}$ does generically not converge to zero. On the other hand, for $s > 0$, one may consider $x_{m} = -\dl_{1m}$, then (for $n\neq 1$) $(Ax)_{n} = 1/(n^{2}-1) = n^{-1}\ell_{n}^{r}$ for any $r > 1$ but not better.\map
\end{rem}

\begin{lem}
\label{appl-young}
For any $1\le p \le \infty$ the map
\[
  G\colon (x_{m})_{m\ge 1} \mapsto \p*{ \sum_{m\neq n} \frac{x_{m}}{\abs{m-n}^{2}} }_{n\ge 1}
\]
defines an operator on $\ell_{\C}^{p}$ bounded by $4$.~\fish
\end{lem}

\begin{proof}
For $p=1$ one has $\n{Gx}_{\ell^{1}} \le \sum_{m\ge 1} \abs{x_{m}} \p*{\sum_{m\neq n}\frac{1}{\abs{m-n}^{2}}} \le 4\n{x}_{\ell^{1}}$ while for $p=\infty$ we find $\n{Gx}_{\ell^{\infty}} \le \sup_{m\ge 1} \abs{x_{m}}\p*{\sum_{m\neq n}\frac{1}{\abs{m-n}^{2}}} \le 4\n{x}_{\ell^{\infty}}$. The case $1<p<\infty$ then follows by interpolation.\qed
\end{proof}

To simplify notation we introduce $\sg^{0} = (n^{2}\pi^{2})_{n\ge 1}$ and write for any $\sg = (\sg_{n})_{n\ge 1}$, $\tsg = \sg - \sg^{0}$.

\begin{lem}
\label{mht-2}
Let $\sg = \sg^{0}+\tsg$ and $\rho = \sg^{0}+\trho$ be two sequences of complex numbers with $\tsg,\trho\in \ell_{\C}^{-1,\infty}$. Suppose there exists some $c > 0$ so that for all $m\neq n$
\[
  \abs{\rho_{m}-\sg_{n}} \ge c^{-1}\abs{m^{2}-n^{2}},
\]
then for any $-1\le s\le 0$ and any $1 < p < \infty$ the mapping
\[
  B\colon \ell_{\C}^{s,p}\to \ell_{\C}^{s+1,p},\quad
  Bx = \p*{\sum_{m\neq n} \frac{x_m}{\rho_{m}-\sg_{n}}}_{n\ge 1}
\]
defines a bounded operator with
\[
  \n{B_{\ell^{s,p}\to\ell^{s+1,p}}} \le \frac{\n{A_{\ell^{s,p}\to\ell^{s+1,p}}} + 4c\n{\trho}_{-1,\infty} + 4c\n{\tsg}_{-1,\infty}}{\pi^{2}}.~\fish
\]
\end{lem}

\begin{proof}
Write
\begin{align*}
  \pi^{2}\sum_{m\neq n}\frac{x_m}{\rho_{m}-\sg_{n}}
  &= \sum_{m\neq n}\frac{x_m}{m^{2}-n^{2}}\p*{1- \frac{\trho_{m}-\tsg_{n}}{\rho_{m}-\sg_{n}}}\\
  &=
  (Ax)_{n} - (F_{\rho}x)_{n} + (G^{\sg}x)_{n},
\end{align*}
where
\[
  (F_{\rho}x)_{n} =
     \sum_{m\neq n}\frac{x_m}{m^{2}-n^{2}}\frac{\trho_{m}}{\rho_{m}-\sg_{n}},\qquad
  (G^{\sg}x)_{n} =
     \sum_{m\neq n}\frac{x_m}{m^{2}-n^{2}}\frac{\tsg_{n}}{\rho_{m}-\sg_{n}}.
\]
Since $\abs{\trho_{m}} \le m\n{\trho}_{-1,\infty}$, $\abs{\rho_{m}-\sg_{n}}\ge c^{-1}\abs{m^{2}-n^{2}}$ for $m\neq n$, as well as $(m^{2}-n^{2})^{2} = (m-n)^{2}(m+n)^{2}$ and $\sup_{m,n\ge 1} \frac{n^{1+s}m^{1-s}}{(m+n)^{2}} \le 1$, we conclude with Lemma~\ref{appl-young} that
\[
  \n{F_{\rho}x}_{1+s,p}
   \le c\n{\trho}_{-1,\infty}
   \p*{\sum_{n\ge 1}\abs*{\sum_{m\neq n}\frac{m^{s}\abs{x_m}}{\abs{m-n}^{2}}}^{p}}^{1/p}
   \le 4c\n{\trho}_{-1,\infty}\n{x}_{s,p}.
\]
In a similar fashion one obtains that $\n{G^{\sg}x}_{1+s,p} \le 4c\n{\tilde\sg}_{-1,\infty}\n{x}_{s,p}$.\qed
\end{proof}

\section{Infinite products}
\label{app:inf-prod}

First let us recall some definitions and facts on infinite products form~\cite{Grebert:2014iq}. Let $a\defl (a_{n})_{n\ge 1}$ be a sequence of complex numbers. We say that the infinite product $\prod_{n\ge 1} (1+a_{n})$ \emph{converges} if the limit $\lim_{N\to\infty} \prod_{1\le n\le N} (1+a_{n})$ exists, and $\prod_{n\ge 1} (1+a_{n})$ is said to be \emph{absolutely convergent} if $\prod_{n\ge 1} (1+\abs{a_{n}})$ converges. One verifies that absolute convergence implies convergence. A sufficient condition for absolute convergence is that $\n{a}_{\ell^{1}} \defl \sum_{n\ge 1} \abs{a_{n}} < \infty$.

The following result is obtained from~\cite[Lemma C1]{Grebert:2014iq} by considering sequences $a=(a_{m})_{m\in\Z}$ with $a_{m} = 0$ for $m\le 0$.

\begin{lem}
\label{inf-prod-est}
Assume that for any $n\ge 1$, $(a_{m,n})_{m\ge 1}$ is an $\ell^{1}$-sequence with $\abs{a_{m,n}}\le 1/2$ for any $m,n$. Then
\[
  \abs*{\prod_{m\ge 1}(1+a_{m,n}) -1} \le A_{n} \e^{S_{n}} + B_{n}\e^{S_{n}+S_{n}^{2}}
\]
with $A_{n} = \abs*{\sum_{m\ge 1} a_{m,n}}$, $B_{n}= \sum_{m\ge 1} \abs{a_{m,n}}^{2}$, and $S_{n} = \sum_{m\ge 1}\abs{a_{m,n}}$.~\fish
\end{lem}

We say a sequence $\sg = (\sg_{n})_{n\ge 1}$ of complex numbers is simple if $\sg_{m}\neq \sg_{n}$ for any $n\neq m$, and define $\sg^{0} = (n^{2}\pi^{2})_{n\ge 1}$.

\begin{lem}
For $\tilde\sg = \sg-\sg^{0}\in \ell^{s,p}$, $-1\le s\le 0$, $1 < p < \infty$, and $n\ge 1$,
\[
  f_{n}(\lm,\tilde \sg) = \frac{1}{n^{2}\pi^{2}}\prod_{m\neq n} \frac{\sg_{m}-\lm}{m^{2}\pi^{2}}
\]
defines an analytic function on $\C\times\ell_{\C}^{s,p}$ with roots $\sg_{m}$, $m\neq n$, listed with their multiplicites. In particular, if $\sg$ is simple, then $f_{n}$ has simple roots $\sg_{m}$, $m\neq n$, and no other roots and
\[
  \frac{1}{f_{n}(\lm,\tilde \sg)} = n^{2}\pi^{2}\prod_{m\neq n} \frac{m^{2}\pi^{2}}{\sg_{m}-\lm}
\]
is meromorphic with simple poles $\sg_{m}$, $m\neq n$.~\fish
\end{lem}

To proceed, we introduce the complex discs
\[
  D_{n} = \setdef{\lm\in\C}{\abs{\lm-\sg_{n}^{0}} < n},\qquad n\ge 1.
\]

\begin{lem}
\label{inf-prod-quot}
Suppose $\sg = \sg^{0} + \tilde \sg$ and $\rho = \sg^{0} + \tilde\rho$ are two complex sequences with $\rho$ simple, and for some $n_{0}\ge 1$ and $c > 0$
\[
  \min_{\lm\in D_{n}} \abs{\rho_{m}-\lm} \ge c^{-1}\abs{m^{2}-n^{2}},\qquad m\neq n,\quad n\ge n_{0}.
\]
If $\tilde\sg,\tilde\rho\in \ell^{-1,\infty}$ and $\sg-\rho\in \ell^{s,p}$ for some $-1\le s\le 0$, $1 < p < \infty$, then
\[
  \sup_{\lm\in D_{n}} \abs*{\prod_{m\neq n}\frac{\sg_{m}-\lm}{\rho_{m}-\lm} - 1}
   = n^{-1-s}\ell_{n}^{p}
\]
uniformly with respect to $\n{\sg-\rho}_{\ell^{s,p}}$ and $\n{\tilde\rho}_{\ell^{-1,\infty}}$.
In more detail, if $N\ge n_{0}$ is such that
\[
  \frac{2c}{N}\n{\sg-\rho}_{\ell^{-1,p}} + c\n{R_{N/2}(\sg-\rho)}_{\ell^{-1,p}} \le 1/2,
\]
where $R_{n}(\sg-\rho) = (\sg_{m}-\rho_{m})_{m\ge n}$, then
\[
  \sum_{n\ge N} n^{(1+s)p}\abs*{\sup_{\lm\in D_{n}}\abs*{\prod_{m\neq n}\frac{\sg_{m}-\lm}{\rho_{m}-\lm} - 1}}^{p}
  \le C\n{\sg-\rho}_{\ell^{s,p}}^{p},
\]
with $C\equiv C(c,\n{\sg-\rho}_{\ell^{s,p}},\n{\tilde \rho}_{\ell^{-1,\infty}})$.\fish
\end{lem}

\begin{proof}
Given any sequence $(\lm_{n})_{n\ge 1}\subset \C$ with $\lm_{n}\in D_{n}$ for any $n\ge 1$, introduce
\[
  a_{m,n} \defl \frac{\sg_{m}-\lm_{n}}{\rho_{m}-\lm_{n}} - 1
  = \frac{\al_{m}}{\rho_{m}-\lm_{n}},\qquad
  \al_{m} \defl \sg_{m}-\rho_{m}.
\]
Since $\al_{m} = m^{-s}\ell_{m}^{p}$ and $\abs{\rho_{m}-\lm_{n}} \ge c^{-1}\abs{m^{2}-n^{2}}$ for $m\neq n$ and $n\ge n_{0}$, there exists $N\ge n_{0}$ so that for all $n\ge N$
\[
  \abs{a_{m,n}} \le \begin{cases}
  \frac{2c}{n}\n{\al}_{-1,p} \le \frac{1}{2}, & \abs{m-n}> n/2,\\
  c\n{R_{n/2}\al}_{-1,p} \le \frac{1}{2}, & 1\le \abs{m-n} \le n/2.
  \end{cases}
\]
Therefore, Lemma~\ref{inf-prod-est} applies yielding
\[
  \abs*{\prod_{m\neq n}\frac{\sg_{m}-\lm_{n}}{\rho_{m}-\lm_{n}} - 1}
  \le A_{n}\e^{S_{n}} + B_{n}\e^{S_{n}+S_{n}^{2}},\qquad
  n\ge N,
\]
with $S_{n} = \sum_{m\neq n} \abs*{\frac{\al_{m}}{\rho_{m}-\lm_{n}}}$, $A_{n} = \abs*{\sum_{m\neq n} \frac{\al_{m}}{\rho_{m}-\lm_{n}}}$, and $B_{n} = \sum_{m\neq n} \abs*{\frac{\al_{m}}{\rho_{m}-\lm_{n}}}^{2}$. Since $1 < p < \infty$ we can apply Hölder's inequality to obtain
\[
  S_{n} \le c\sum_{m\neq n} \frac{m^{-1}\abs{\al_{m}}}{\abs{m-n}}
        \le c\p*{\sum_{m\neq n}\frac{1}{\abs{m-n}^{p'}}}^{1/p'}\n{\al}_{-1,p}
        \le C_{c,p}\n{\al}_{-1,p}.
\]
By Lemma~\ref{mht-2} one has
\[
  \n{(A_{n})_{n\ge N}}_{1+s,p} \le C_{c,\n{\tsg}_{-1,\infty},\n{\trho}_{-1,\infty},p}\n{\al}_{s,p},
\]
and finally for any $q\ge \max(1,p/2)$ using that $\sup_{m,n\ge 1} \frac{n^{2+2s}m^{-2s}}{\abs{m+n}^{2}}\le 1$ and Lemma~\ref{appl-young}
\begin{align*}
  \n{(B_{n})_{n\ge N}}_{2+2s,q}
  &\le c^{2}\p*{\sum_{n\ge N} \abs*{\sum_{m\neq n} \frac{n^{2+2s}m^{-2s}}{(m+n)^{2}}\frac{m^{2s}\abs{\al_{m}}^{2}}{(m-n)^{2}}}^{q}}^{1/q}
  \le 16c^{2}\n{\al}_{\ell^{s,2q}}^{2}.\qed
\end{align*}
\end{proof}

Finally recall that $\frac{\sin\sqrt{\lm}}{\sqrt{\lm}}$ admits the product representation
\[
  \frac{\sin\sqrt{\lm}}{\sqrt{\lm}} = \prod_{m\ge 1}\frac{m^{2}\pi^{2}-\lm}{m^{2}\pi^{2}}
  \qquad\text{or}\qquad
  \frac{n^{2}\pi^{2}}{n^{2}\pi^{2}-\lm}\frac{\sin\sqrt{\lm}}{\sqrt{\lm}} = \prod_{m\neq n}\frac{m^{2}\pi^{2}-\lm}{m^{2}\pi^{2}},
\]
which defines $\frac{n^{2}\pi^{2}}{n^{2}\pi^{2}-\lm}\frac{\sin\sqrt{\lm}}{\sqrt{\lm}}$ for $\lm = n^{2}\pi^{2}$.

\begin{lem}
\label{inf-prod-sin}
Let $\sg = \sg^{0} + \tilde \sg$ with $\tilde \sg\in \ell^{s,p}$, $1 < p < \infty$ and $-1\le s \le 0$, then for any $n\ge 1$,
\begin{equation}
  \label{prod-sin-1}
    \prod_{m\neq n} \frac{\sg_{m}-\lm}{m^{2}\pi^{2}}
  = \frac{n^{2}\pi^{2}}{n^{2}\pi^{2}-\lm}\frac{\sin \sqrt{\lm}}{\sqrt{\lm}}\p*{1 + n^{-1-s}\ell_{n}^{p}},
\end{equation}
uniformly in $\lm\in D_{n}$ and with respect to $\n{\sg-\sg^{0}}_{s,p}$. Write $\lm_{n}\in D_{n}$ as $\sqrt{\lm_{n}} = n\pi + \al_{n}$. Then $\abs{\al_{n}} \le 1/\pi$ and
\begin{equation}
  \label{prod-sin-2}
  \frac{n^{2}\pi^{2}}{n^{2}\pi^{2}-\lm}\frac{\sin\sqrt{\lm}}{\sqrt{\lm}}
  = \frac{(-1)^{n+1}}{2}\p*{1 + \bt_{n}},\qquad
  \abs{\bt_{n}} \le \frac{1}{n}\abs{\al_{n}} + \frac{1}{2}\abs{\al_{n}}^{2}.
\end{equation}
In particular, if $\al_{n} = n^{-1-s}\ell_{n}^{p}$ (or $\al_{n} = O(n^{-1-s})$ and $s+1>1/p$), then
\begin{equation}
  \label{prod-sin-3}
    \frac{n^{2}\pi^{2}}{n^{2}\pi^{2}-\lm}\frac{\sin\sqrt{\lm}}{\sqrt{\lm}}
  = \frac{(-1)^{n+1}}{2} + n^{-1-s}\ell_{n}^{p}.~\fish
\end{equation}
\end{lem}

\begin{proof}
\eqref{prod-sin-1} follows directly from the product representation of $\frac{\sin\sqrt{\lm}}{\sqrt{\lm}}$ and Lemma~\ref{inf-prod-quot}.
To obtain~\eqref{prod-sin-2} write $\sqrt{\lm_{n}} = n\pi + \al_{n}$, then by a straightforward computation
\[
  \frac{n^{2}\pi^{2}}{n^{2}\pi^{2}-\lm}\frac{\sin\sqrt{\lm}}{\sqrt{\lm}}
  = \frac{(-1)^{n+1}}{2}\p*{1 - \frac{\al_{n}(3n\pi + \al_{n})}{(n\pi + \al_{n})(2n\pi+\al_{n})}}\frac{\sin \al_{n}}{\al_{n}}.
\]
Since $\abs{\lm_{n}-n\pi} < n$ we have that $\abs{\al_{n}} < 1/\pi$ and hence
\[
  \abs*{\frac{\al_{n}(3n\pi + \al_{n})}{(n\pi + \al_{n})(2n\pi+\al_{n})}} \le \frac{\abs{\al_{n}}}{n},\qquad
  \abs*{\frac{\sin \al_{n}}{\al_{n}}-1} \le \frac{\abs{\al_{n}}^{2}}{4}.
\]
Consequently,
\[
  \frac{n^{2}\pi^{2}}{n^{2}\pi^{2}-\lm}\frac{\sin\sqrt{\lm}}{\sqrt{\lm}} = \frac{(-1)^{n+1}}{2}\p*{1 + \bt_{n}},
\]
where
\[
  \abs{\bt_{n}} \le \frac{1}{n}\abs{\al_{n}} + \frac{1}{2}\abs{\al_{n}}^{2}.
\]
Finally~\eqref{prod-sin-3} is a consequence of~\eqref{prod-sin-2}.\qed
\end{proof}


\section{A diffeomorphism property}
\label{app:diffeo}

Let $Z$ be a $\K$-Banach space with $\K=\R$ or $\C$ and let $T\colon Z\to Z$ be a bounded linear operator. Suppose that $Z$ is the direct sum of closed subspaces $X$ and $Y$, $Z = X\oplus Y$, and that $T$ admits with respect to this direct sum the decomposition
\begin{equation}
  \label{decomp}
    T = \begin{pmatrix}
  A & B\\
  C & D
  \end{pmatrix}.
\end{equation}

\begin{lem}[Schur complement]
The operator $\Id_{Z}+T$ is invertible on $Z$ if and only if $\Id_{Y}+D$ is invertible on $Y$ and the Schur complement
\[
  S = \Id_{X} + A - B(\Id_{Y}+D)^{-1}C
\]
is invertible on $X$.~\fish
\end{lem}

As an immediate corollary we obtain the following sufficient condition.

\begin{cor}
\label{schur-cor}
Suppose $X$ is finite dimensional, then $\Id_{Z}+T$ is invertible if
\[
  \det S \neq 0,\qquad \n{D}_{L(Y)} < 1.~\fish
\]
\end{cor}

The following characterization of relatively compact sets in $\ell^{p}$ is well known.

\begin{lem}
\label{rel-comp-ellp}
A subset $B$ of $\ell_{\C}^{p}$, $1\le p < \infty$, is relatively compact if and only if it is bounded and for any $\ep > 0$ there exists an $N\ge 1$ so that $\n{\pi_{N}^{\bot}x}_{p} \le \ep$ for all $x\in B$. Here $\pi_{N}^{\bot} = \Id - \pi_{N}$ and $\pi_{N}\colon \ell_{C}^{p}\to \ell_{\C}^{p}$ is the projection of sequences $x=(x_{n})_{n\ge 1}\in\ell_{\C}^{p}$ to $\pi_{N}x$ given by $(\pi_{N}x)_{n} = x_{n}$ if $1\le n\le N$ and $(\pi_{N}x)_{n} = 0$ for $n\ge N+1$.~\fish
\end{lem}


We denote by $\ell_{+}^{s,p}$ the positive quadrant of $\ell^{s,p}\equiv \ell^{s,p}(\N,\R)$ introduced in~\eqref{ellp-pos-quad}. Furthermore, $\ell_{\C}^{s,p} \equiv \ell^{s,p}(\N,\C)$ where $s\in\R$ and $1\le p < \infty$.

\begin{prop}
\label{diffeo-prop}
Suppose  $f\colon \ell_{+}^{s,p}\to \ell^{s,p}$, $1\le p < \infty$, $s\in\R$, is a real analytic map with the properties that
\begin{equivenum}
\item
$\ddd_{z}f-\Id\colon \ell_{\C}^{s,p}\to \ell_{\C}^{s,p}$ is compact for every $z\in \ell_{+}^{s,p}$,

\item
$f$ is a local diffeomorphism at some point of $\ell_{+}^{s,p}$.

\end{equivenum}
Then $f$ is a local diffeomorphism on a dense open subset of $\ell_{+}^{s,p}$.~\fish
\end{prop}

\begin{proof}
To simplify notation write $T_{z} = \ddd_{z} f - \Id_{\ell_{\C}^{s,p}}$. By assumption (i), $T_{z}$ is a compact operator on $\ell_{\C}^{s,p}$ for every $z\in \ell_{+}^{s,p}$. In particular, the image of the unit ball in $\ell_{\C}^{s,p}$ is relatively compact in $\ell_{\C}^{s,p}$. By Lemma~\ref{rel-comp-ellp} there exists $N\ge 1$ (which might depend on $z$) so that $\n{\pi_{N}^{\bot}T_{z}}_{L(\ell_{\C}^{s,p})} \le 1/4$. Since $\n{\pi_{N}^{\bot}T_{z}}_{L(\ell_{\C}^{s,p})}$ depends continuously on $z$, there exists an complex neighborhood $V$ of $z$ within $\ell_{\C}^{s,p}$ so that $\n{\pi_{N}^{\bot}T_{w}}_{L(\ell_{\C}^{s,p})} \le 1/2$ for all $w\in V$.

Let $W$ be any nontrivial open subset of $\ell_{+}^{s,p}$ and denote by $z_{0}\in\ell_{+}^{s,p}$ the point of assumption (ii) at which the differential of $f$ is invertible. For any $z_{1}\in W$ the straight line $[z_{0},z_{1}]$ is compact in $\ell_{+}^{s,p}$ and hence can be covered by finitely many neighborhoods $V$ as constructed above. Consequently, there exists a complex neighborhood $U$ of $[z_{0},z_{1}]$ within $\ell_{\C}^{s,p}$ and an integer $N_{U}\ge 1$ so that
\[
  \n{\pi_{N_{U}}^{\bot}T_{z}}_{L(\ell_{\C}^{s,p})} \le 1/2,\qquad \forall z\in U.
\]
Write $\ell_{\C}^{s,p} = X^{N_{U}} \oplus Y^{N_{U}}$ where $X^{N_{U}} = \pi_{N_{U}}(\ell_{\C}^{s,p})$ and $Y^{N_{U}} = \pi_{N_{U}}^{\bot}(\ell_{\C}^{s,p})$. We can decompose for any $z\in U$ the operator $T_{z}$ according to~\eqref{decomp}. Since for any $z\in U$
\begin{align*}
  \n{D_{z}^{N_{U}}}_{L(X^{N_{U}})}
   \le \n{\pi_{N_{U}}^{\bot}T_{z}}_{L(\ell_{\C}^{s,p})}
   \le 1/2,
\end{align*}
by Corollary~\ref{schur-cor} the differential $\ddd_{z}f$ is invertible for all $z\in U$ with
\[
  \lm(z) = \det S_{z}^{N_{U}} \neq 0.
\]
Note that
\[
  U\to L(X^{N_{U}}),\qquad
  z\mapsto S_{z}^{N_{U}}
   = \Id_{X^{N_{U}}} - A_{z}^{N_{U}}
      + B_{z}^{N_{U}}(\Id_{Y^{N_{U}}}+D_{z}^{N_{U}})^{-1}C_{z}^{N_{U}},
\]
is analytic, hence the function $\lm\colon U\to \C$ is analytic. Since $\lm(z_{0})\neq0$, it follows that $\lm$ does not vanish identically on $U\cap W$. Consequently, the set $\Lm = \setdef{z\in \ell_{+}^{s,p}}{\ddd_{z}f\text{ is invertible}}$ has nontrivial intersection with $W$. Since $W$ was arbitrary, it follows that $\Lm$ is dense in $\ell_{+}^{s,p}$. Since $\Lm$ is open the claim follows.\qed
\end{proof}

\section{Birkhoff normal form}
\label{app:bnf-kdv2}

In this appendix we review the Birkhoff normal form of the KdV and KdV2 Hamiltonian provided in~\cite{Kappeler:2003up}.

\begin{thm}
\label{thm:bnf-kdv-kdv2}
\begin{equivenum}
\item
On $\Hs^{1}(\T,\R)$, the Birkhoff normal form of the KdV Hamiltonian $\Hm_{1}(u)$ of order four is given by
\[
  \Hm_{1}(u) = \sum_{n\ge 1} (2n\pi)^{3}I_{n} + 6[u]\Hm_{0}
  - 3\sum_{n\ge 1} I_{n}^{2} + \dotsb,
\]
where $\Hm_{0} = \sum_{n\ge 1} (2n\pi) I_{n}$.
\item
On $\Hs^{2}(\T,\R)$, the Birkhoff normal form of the KdV2 Hamiltonian $\Hm_{2}(u)$ of order four is given by
\begin{align*}
  \Hm_{2}(u) &= \sum_{n\ge1} (2n\pi)^{5}I_{n} + 10[u] \sum_{n\ge 1} (2n\pi)^{3}I_{n} + 30[u]^{2} \Hm_{0}\\
  & \qquad+ 10\Hm_{0}^{2}
    - 10\sum_{n\ge 1} (2n\pi)^{2} I_{n}^{2}
    - 30[u]\sum_{n\ge 1} I_{n}^{2} + \dotsb.~\fish
\end{align*}

\end{equivenum}

\end{thm}

For a proof of item~(i) we refer to \cite[Theorem~14.2]{Kappeler:2003up}. Concerning item (ii), it turns out that some of the coefficients in the expansion of $\Hm_{2}$, given in \cite[Theorem~14.5]{Kappeler:2003up}, need to be corrected. We therefore present a detailed derivation.

Given $u\in\Hs^{s}$ by choosing $c=[u]$ we have that $q = u - c \in\Hs_{0}^{s}$ and the KdV2 Hamiltonian satisfies the relation
\[
  \Hm_{2}(u) = \Hm_{2}(q) + 10c \Hm_{1}(q) + 30c^{2}\Hm_{0}(q) + \frac{5}{2}c^{4},
\]
where
\[
  \Hm_{0}(q) = \sum_{n\ge 1} (2n\pi)I_{n},\qquad
  \Hm_{1}(q) = \sum_{n\ge 1} (2n\pi)^{3}I_{n} - 3\sum_{n\ge 1} I_{n}^{2} + \dotsb.
\]
It thus suffices to compute the Birkhoff normal form of $\Hm_{2}(q)$ up to order four.

To begin denote by $\Pc_{k}$ the space of homogenous polynomials of order $k$ and write $\Hm_{2} = H^{2} + H^{3} + H^{4}$ with $H^{k}\in\Pc_{k}$.
Putting $\Hm_{2}$ into Birkhoff normal form of order four amounts to the construction of a coordinate change $\Phi$ so that
\[
  \Hm_{2}\circ\Phi = H^{2} + N^{4} + \dotsb
\]
where $N^{4}\in\Pc_{4}$, $\{H^{2},N^{4}\} = 0$, and $\dotsb$ comprises terms of order at least five. The map $\Phi$ is obtained as the composition of two time-1-maps of Hamiltonian vector fields whose Hamiltonians are chosen properly. More to the point, $\Phi = F_{3}\circ F_{4}$ with $F_{k}\in\Pc_{k}$.

By the chain rule and the fact that $\{F_{k},F_{l}\}\in \Pc_{k+l-2}$ one has
\[
  \Hm^{2}\circ F_{3} = H^{2} + H^{3} + \{H^{2},F_{3}\} + H^{4} + \frac{1}{2}\{\{H^{2},F_{3}\},F_{3}\} + \{H^{3},F_{3}\} + \dotsb,
\]
where $\dotsb$ comprises terms of at least order five. Moreover,
\begin{align*}
  \Hm^{2}\circ F_{3}\circ F_{4}
  &=
  H^{2} + H^{3} + \{H^{2},F_{3}\}\\
  &\qquad + H^{4} + \frac{1}{2}\{\{H^{2},F_{3}\},F_{3}\} + \{H^{3},F_{3}\} + \{H^{2},F_{4}\} + \dotsb
\end{align*}
Since $H\circ\Phi$ is in Birkhoff normal form we have $H^{3} + \{H^{2},F_{3}\} = 0$, so that
\[
  \Hm^{2}\circ \Phi
  =
  H^{2} + H^{4} + \frac{1}{2}\{H^{3},F_{3}\} + \{H^{2},F_{4}\} + \dotsb
\]
Denote by $\Nc_{4}$ the kernel of the map $\chi_{4}\colon\Pc_{4}\to \Pc_{4}$, $F\mapsto \{H^{2},F\}$.
The condition $\{H^{2},N^{4}\} = 0$ is tantamount to $N^{4}$ being an element of $\Nc_{4}$, and the term $\{H^{2},F_{4}\}$ is used to remove the contributions of the complement. Therefore,
\begin{equation}
  \label{N4}
  N^{4} = \frac{1}{2}\pi_{\Nc_{4}}\{H^{3},F_{3}\} + \pi_{\Nc_{4}}H^{4}.
\end{equation}
In the sequel we proceed by computing the coefficients of the two terms of $N^{4}$.

We note two combinatorial properties which can be easily verified by direct computation.

\begin{lem}
\label{comb-3}
Suppose $k,l,m\neq 0$ with $k+l+m=0$ then
\[
  k^{5}+l^{5}+m^{5} = \frac{5}{2}klm(k^{2}+l^{2}+m^{2}) \neq 0.~\fish
\]
\end{lem}


\begin{lem}
\label{comb-4}
Suppose $k,l,m,n\neq 0$ with $k+l+m+n = 0$ then
\[
  k^{5}+l^{5}+m^{5}+n^{5} = 5(k+l)(k+m)(k+n)\xi_{klm}
\]
and $\xi_{klm} = (k^{2}+kl+l^{2}+km+lm+m^{2})$ does not vanish.~\fish
\end{lem}


To compute the coefficients of~\eqref{N4} we make the ansatz
\[
  u = \sum_{m\neq 0} \gm_{m}u_{m}e_{2m},\qquad \gm_{m} = \sqrt{2\abs{m}\pi}.
\]
A straightforward computation gives first
\begin{align}
\label{H2}
\begin{split}
  H^{2}
   &= \frac{1}{2}\sum_{\atop{k,l\neq 0}{k+l = 0}} (2k\pi\ii)^{2}(2l\pi\ii)^{2}\gm_{k}\gm_{l}u_{k}u_{l}\\
   &= \frac{1}{2}\sum_{k\neq 0} (2\abs{k}\pi)^{5}u_{k}u_{-k}
    = \sum_{k\ge 1} \lm_{k}u_{k}u_{-k},\qquad \lm_{k} = (2k\pi)^{5},
\end{split}
\end{align}
second
\begin{align}
\label{H3}
\begin{split}
  H^{3}
   &= -5\sum_{k+l+m=0} (2k\pi)(2l\pi)\gm_{k}\gm_{l}\gm_{m} u_{k}u_{l}u_{m}\\
   &= \frac{10\pi^{2}}{3}\sum_{k+l+m=0} (k^{2}+l^{2}+m^{2})\gm_{k}\gm_{l}\gm_{m} u_{k}u_{l}u_{m}
\end{split}
\end{align}
where we used that $kl + lm + mk = -\frac{1}{2}(k^{2}+l^{2}+m^{2})$ given $k+l+m=0$, and third
\begin{align}
\label{H4}
  H^{4}
   = \frac{5}{2}\sum_{k+l+m+n=0} \gm_{k}\gm_{l}\gm_{m}\gm_{n} u_{k}u_{l}u_{m}u_{n}.
\end{align}

Since $H\circ\Phi$ does not contain terms of order three, we have $H^{3} = -\{H^{2},F_{3}\}$. To compute the coefficients of $F_{3}$, write $F_{3} = \sum_{k+l+m=0} F_{klm}^{3}u_{k}u_{l}u_{m}$, then
\begin{align*}
  &\frac{10\pi^{2}}{3}\sum_{k+l+m=0} (k^{2}+l^{2}+m^{2})\gm_{k}\gm_{l}\gm_{m} u_{k}u_{l}u_{m}\\
  &\qquad = \ii \sum_{k+l+m=0} (\lm_{k} + \lm_{l} + \lm_{m})F_{klm}^{3} u_{k}u_{l}u_{m}
\end{align*}
so that
\begin{align*}
  F_{klm}^{3}
   &= \frac{10\pi^{2}}{3\cdot 2^{5}\pi^{5}}\frac{k^{2}+l^{2}+m^{2}}{k^{5}+l^{5}+m^{5}}\gm_{k}\gm_{l}\gm_{m}\\
   &= -\ii\frac{1}{3\cdot2^{3}\pi^{3}}\frac{1}{klm}\gm_{k}\gm_{l}\gm_{m}
   = -\frac{\ii}{3} \frac{1}{\tilde \gm_{k}\tilde \gm_{l}\tilde \gm_{m}},\qquad
   \tilde\gm_{k} = \sg_{k}\gm_{k}.
\end{align*}
These coefficients match those of (14.4) in~\cite{Kappeler:2003up}.

\begin{lem}
For any smooth $u$ with $[u] = 0$,
\begin{align*}
  \frac{1}{2}\pi_{\Nc_{4}} \{H^{3},F_{3}\}
  = -20\sum_{k,l\ge 1} (2k\pi)(2l\pi) \abs{u_{k}}^{2}\abs{u_{l}}^{2}
      +5\sum_{n\ge 1} (2n\pi)^{2} \abs{u_{n}}^{4}.~\fish
\end{align*}
\end{lem}

\begin{proof}
First we compute from~\eqref{H3}
\begin{align*}
  \frac{3}{10\pi^{2}}\partial_{j}H^{3}
   = 3\sum_{l+m= -j} (j^{2}+l^{2}+m^{2})\gm_{j}\gm_{l}\gm_{m} u_{l}u_{m}
\end{align*}
and second
\begin{align*}
  \partial_{-j}F_{3} = 3\sum_{l+m= j} -\frac{\ii}{3} \frac{1}{\tilde \gm_{-j}\tilde \gm_{l}\tilde \gm_{m}}.
\end{align*}
Together this gives
\begin{align*}
  \{H^{3},F_{3}\}
  &= \ii 10\pi^{2}(-\ii) \sum_{j\neq 0}\sg_{j}
  \p*{\sum_{k+l= -j} (j^{2}+k^{2}+l^{2})\gm_{j}\gm_{k}\gm_{l} u_{k}u_{l}}
  \p*{\sum_{m+n= j} \frac{1}{\tilde \gm_{-j}\tilde \gm_{m}\tilde \gm_{n}}u_{m}u_{n}}\\
  &= -10\pi^{2} \sum_{j\neq 0}
  \p*{\sum_{k+l= -j} (j^{2}+k^{2}+l^{2})\gm_{k}\gm_{l} u_{k}u_{l}}
  \p*{\sum_{m+n= j} \frac{1}{\tilde \gm_{m}\tilde \gm_{n}}u_{m}u_{n}}\\
  &= -10\pi^{2} \sum_{j\neq 0}
  \p*{\sum_{\atop{k+l= -j}{m+n= j}} (j^{2}+k^{2}+l^{2}) \frac{\gm_{k}\gm_{l}}{\tilde \gm_{m}\tilde \gm_{n}}u_{k}u_{l}u_{m}u_{n}}\\
  &= -10\pi^{2} \sum_{\atop{k+l+m+n=0}{k+l\neq 0}}
   ((k+l)^{2}+k^{2}+l^{2}) \frac{\gm_{k}\gm_{l}}{\tilde \gm_{m}\tilde \gm_{n}}u_{k}u_{l}u_{m}u_{n}.
\end{align*}
Note that on $\Nc_{4}$ in view of Lemma~\ref{comb-4} we either have $k+l=0$ or $k+m=0$ or $k+n=0$. We first compute
\begin{align*}
  &\sum_{\atop{k+l+m+n=0}{k+l\neq 0,\, k+m = 0}}
   ((k+l)^{2}+k^{2}+l^{2}) \frac{\gm_{k}\gm_{l}}{\tilde \gm_{m}\tilde \gm_{n}}u_{k}u_{l}u_{m}u_{n}\\
   &=
   \sum_{k+l\neq 0} \sg_{-k}\sg_{-l}
   ((k+l)^{2}+k^{2}+l^{2}) \abs{u_{k}}^{2}\abs{u_{l}}^{2}\\
   &=
   2\sum_{k,l\ge 1}
   ((k+l)^{2}+k^{2}+l^{2}) \abs{u_{k}}^{2}\abs{u_{l}}^{2}
   -
   2\sum_{\atop{k,l\ge 1}{k\neq l}}
   ((k-l)^{2}+k^{2}+l^{2}) \abs{u_{k}}^{2}\abs{u_{l}}^{2}\\
   &=
   12\sum_{n\ge 1}
   n^{2} \abs{u_{n}}^{4}
   +
   8\sum_{\atop{k,l\ge 1}{k\neq l}}
   kl \abs{u_{k}}^{2}\abs{u_{l}}^{2},
\end{align*}
and second
\begin{align*}
  &\sum_{\atop{k+l+m+n=0}{k+l\neq 0,\, k+m \neq 0,\, k+n = 0}}
   ((k+l)^{2}+k^{2}+l^{2}) \frac{\gm_{k}\gm_{l}}{\tilde \gm_{m}\tilde \gm_{n}}u_{k}u_{l}u_{m}u_{n}\\
   &=
   \sum_{\atop{k+l\neq 0}{k-l\neq 0}} \sg_{-k}\sg_{-l}
   ((k+l)^{2}+k^{2}+l^{2}) \abs{u_{k}}^{2}\abs{u_{l}}^{2}\\
   &=
   2\sum_{\atop{k,l\ge 1}{k\neq l}}
   ((k+l)^{2}+k^{2}+l^{2}) \abs{u_{k}}^{2}\abs{u_{l}}^{2}
   -
   2\sum_{\atop{k,l\ge 1}{k\neq l}}
   ((k-l)^{2}+k^{2}+l^{2}) \abs{u_{k}}^{2}\abs{u_{l}}^{2}\\
   &=
   8\sum_{\atop{k,l\ge 1}{k\neq l}}
   kl \abs{u_{k}}^{2}\abs{u_{l}}^{2}.
\end{align*}
So that we arrive at
\begin{align*}
  \frac{1}{2}\pi_{\Nc_{4}} \{H^{3},F_{3}\}
  &= -5\p*{
  3\sum_{n\ge 1}
   (2n\pi)^{2} \abs{u_{n}}^{4}
   +
   4\sum_{\atop{k,l\ge 1}{k\neq l}}
   (2k\pi)(2l\pi) \abs{u_{k}}^{2}\abs{u_{l}}^{2}
  }\\
  &= -20\sum_{k,l\ge 1} (2k\pi)(2l\pi) \abs{u_{k}}^{2}\abs{u_{l}}^{2}
      +5\sum_{n\ge 1} (2n\pi)^{2} \abs{u_{n}}^{4}.\qed
\end{align*}
\end{proof}

\begin{lem}
For any smooth $u$ with $[u] = 0$,
\begin{align*}
  \pi_{\Nc_{4}}H^{4} =
  30\sum_{k,l\ge 1} \gm_{k}^{2}\gm_{l}^{2} \abs{u_{k}}^{2}\abs{u_{l}}^{2}
  - 15\sum_{n\ge 1} \gm_{n}^{4}\abs{u_{n}}^{4}.~\fish
\end{align*}
\end{lem}

\begin{proof}
Recall from~\eqref{H4} that
\[
  H^{4}
   = \frac{5}{2}\sum_{k+l+m+n=0} \gm_{k}\gm_{l}\gm_{m}\gm_{n} u_{k}u_{l}u_{m}u_{n}.
\]
On $\Nc_{4}$ we either have $k+l=0$ or $k+m=0$ or $k+n=0$. Therefore, the projection of $\frac{2}{5}H^{4}$ onto $\Nc_{4}$ is given by
\begin{align*}
  &\sum_{\atop{k+l+m+n=0}{k+l=0}} \gm_{k}\gm_{l}\gm_{m}\gm_{n} u_{k}u_{l}u_{m}u_{n}
   +
  \sum_{\atop{k+l+m+n=0}{k+l\neq 0,\, k+m = 0}} \gm_{k}\gm_{l}\gm_{m}\gm_{n} u_{k}u_{l}u_{m}u_{n}\\
 &\qquad +
  \sum_{\atop{k+l+m+n=0}{k+l\neq 0,\, k+m \neq 0,\, k+n = 0}} \gm_{k}\gm_{l}\gm_{m}\gm_{n} u_{k}u_{l}u_{m}u_{n}\\
  &=
  4\sum_{k,l\ge 1} \gm_{k}^{2}\gm_{l}^{2} \abs{u_{k}}^{2}\abs{u_{l}}^{2}
   +
  2\sum_{k,l\ge 1} \gm_{k}^{2}\gm_{l}^{2} \abs{u_{k}}^{2}\abs{u_{l}}^{2}
  +
  2\sum_{\atop{k,l\ge 1}{k\neq l}} \gm_{k}^{2}\gm_{l}^{2} \abs{u_{k}}^{2}\abs{u_{l}}^{2}\\
  &\qquad
  +4\sum_{\atop{k,l\ge 1}{k\neq l}} \gm_{k}^{2}\gm_{l}^{2} \abs{u_{k}}^{2}\abs{u_{l}}^{2}\\
  &=
  12\sum_{k,l\ge 1} \gm_{k}^{2}\gm_{l}^{2} \abs{u_{k}}^{2}\abs{u_{l}}^{2}
  - 6\sum_{n\ge 1} \gm_{n}^{4}\abs{u_{n}}^{4}.
\end{align*}
Consequently,
\[
  \pi_{\Nc_{4}}H^{4} =
  30\sum_{k,l\ge 1} \gm_{k}^{2}\gm_{l}^{2} \abs{u_{k}}^{2}\abs{u_{l}}^{2}
  - 15\sum_{n\ge 1} \gm_{n}^{4}\abs{u_{n}}^{4}.\qed
\]
\end{proof}

Altogether we find
\[
  \Hm_{2}\circ\Phi = H^{2} + 10\sum_{k,l\ge 1} (2k\pi)(2l\pi) \abs{u_{k}}^{2}\abs{u_{l}}^{2}
      - 10\sum_{n\ge 1} (2n\pi)^{2} \abs{u_{n}}^{4} + \dotsb
\]

\paragraph{Addendum: nondegeneracy of the KdV2 frequencies}

Since some of the coefficients in the Birkhoff normal form of $\Hm_{2}$ in \cite[Theorem~14.5]{Kappeler:2003up} had to be corrected, the analysis of the KdV2 frequencies presented in \cite[Appendix~J]{Kappeler:2003up} needs to be adapted accordingly. It turns out that the results stated in Appendix~J continue to hold but their proofs have to be slightly modified as explained in detail in what follows.

The Birkhoff normal form of $\Hm_{2}$ may also be written in the from
\[
  \Hm_{2} = \sum_{n\ge 1} \lm_{n}^{(2)}I_{n} - \frac{1}{2} \sum_{i,j\ge 1} C_{ij}^{(2)}I_{i}I_{j} + \dotsb,
\]
with
\begin{align*}
  \lm_{n}\equiv \lm_{n}^{(2)} &= (2n\pi)^{5} + 10c (2n\pi)^{3} + 30 c^{2}(2n\pi),
  \\
  C_{ij} \equiv C_{ij}^{(2)} &= \begin{cases}
  60c, & i = j,\\
  -20(2i\pi)(2j\pi), & i\neq j,
  \end{cases}
\end{align*}
so that
\[
  \om_{n}^{(2)} = \lm_{n} - \sum_{j\ge 1} C_{nj}I_{j} + \dotsb.\map
\]
For any finite set $A\subset\N$ let $Z = \N\setminus A$ and
\[
  \P\ell_{A} = \setdef{x_{A} = (x_{j})_{j\in A}}{x_{j}\ge 0}.
\]
We may decompose any $k\in\Z^{\N}$ into $k = k_{A} + k_{Z}$, where $k_{A}$ denotes the projection on $A$ and $k_{Z}$ the projection on $Z$, respectively.

\begin{prop}
\label{kdv2-Hm-nondegenerate}
For every finite index set $A\subset\N$, the following holds on $\P\ell_{A}$.
\begin{equivenum}
\item
There exists an $\abs{A}$-point set $\Ss_{A}\subset\R$ such that for $c\notin \Ss_{A}$,
\[
  \det (\partial_{I_{i}}\partial_{I_{j}} \Hm_{2})_{i,j\in A}  \nequiv 0.
\]
The subset $\Ss_{A}$ may be chosen in such a way that $0\notin \Ss_{A}$ if $\abs{A}\neq 1$ and $\Ss_{A} = \setd{0}$ if $\abs{A} = 1$.

\item
There exists an at most countable subset $\Es_{A}\subset\R$ accumulating at most at the points of $\Ss_{A}$ such that for $c\notin \Es_{A}$,
\[
  k\cdot \om^{(2)}\nequiv 0
\]
for any $k = k_{A} + k_{Z} \in \Z^{\N}$ with $1\le \abs{k_{Z}} \defl \sum_{j\in Z} \abs{k_{j}} \le 2$.

\item
The KdV2 Hamiltonian at $c=0$ is nondegenerate in the sense that for any $k\neq 0$ in $\Z^{\N}$ with $\abs{k_{Z}}\le 2$ one has
\[
  k\cdot \om^{(2)}\nequiv 0.~\fish
\]
\end{equivenum}

\end{prop}

Items (i)-(iii) of Proposition~\ref{kdv2-Hm-nondegenerate} follow from Lemma~\ref{kdv2-Q-nondeg}, Lemma~\ref{kdv2-nondeg-1}, and Lemma~\ref{kdv2-nondeg-2}, respectively.

For any finite set $A\subset\N$ let $C_{A} = (C_{ij})_{i,j\in A}$.

\begin{lem}
\label{kdv2-Q-nondeg}
For every finite set $A\subset\N$, there exists an $\abs{A}$-point set $\Ss_{A}\subset\R$ such that
\[
  \det C_{A} = 0 \iff c \in \Ss_{A}.
\]
In particular, if $A = \setd{i}$, then $\Ss_{A} = \setd{0}$, while for $A = \setd{i_{1} < \dotsb < i_{n}}$ one has
\[
  \Ss_{A} = \setd{c_{A}^{n} < \dotsb < c_{A}^{2} < 0 < c_{A}^{1}}
\]
with
\[
  -\frac{4}{3}\pi^{2}i_{\nu}^{2} < c_{A}^{\nu} < -\frac{4}{3}\pi^{2}i_{\nu-1}^{2},\qquad 2\le \nu \le n,
\]
and $c_{A}^{1}\to\infty$ as $\abs{A}\to\infty$.~\fish
\end{lem}

\begin{proof}
The matrix $C_{A}$ can be written in the form $C_{A} = D-B$, where $D = \diag(D_{i})_{i\in A}$ and $B=(B_{ij})_{i,j\in A}$ with coefficients
\[
  D_{i} = 80\pi^{2}i^{2} + 60c,\qquad B_{ij} = 80\pi^{2}ij.
\]
Since $B$ has rank one,
\[
  \det C_{A} = \det D - \sum_{i\in A} B_{ii} \prod_{j\in A,\; j\neq i} D_{i}.
\]
If one of the $D_{j}$ vanishes, say $D_{l} = 0$, then all other $D_{j}$ do not vanish, and we have
\[
  \det C_{A} = -B_{ll}\prod_{j\neq l} D_{j} \neq 0.
\]
Otherwise,
\[
  \det C_{A} = \det D \p[\bigg]{1- \sum_{i\in A} B_{ii}/D_{i}}
\]
and the determinant vanishes if and only if
\[
  1
   = \sum_{i\in A} \frac{B_{ii}}{D_{i}}
   = \sum_{i\in A} \frac{1}{1 + cf_{i}},\qquad
  f_{i} = \frac{3}{4\pi^{2}i^{2}}.
\]
Each summand is a hyperbola in $c$ which is monotonically decreasing on $(-\infty,c_{i})$ and $(c_{i},\infty)$ with $c_{i} = -4\pi^{2}i^{2}/3$ being the single pole. Furthermore, each summand has value $1$ at $c=0$, and asymptotic value $0$ as $c\to \pm \infty$. This proves the claim.~\qed
\end{proof}

\begin{rem}
The lemma shows that for any given $A\subset\N$ the Jacobian of the frequency map $I_{A}\mapsto \om_{A}^{(2)}$ of the KdV2 Hamiltonian \emph{does} become singular, at least at $I_{A} = 0$ for $c\in\Ss_{A}$. This is in contrast to the first KdV Hamiltonian, where the Jacobian is always regular at $I_{A} = 0$.
\end{rem}

We now fix a finite set $A\subset\N$ and consider for $0\neq k\in \Z^{\N}$ the frequency combinations $k\cdot \om^{(2)}$ as functions of $I_{A}$ on $\P\ell_{A}$. In view of $\om^{(2)}(I) = \lm - CI + \dotsb$ and the symmetry of the matrix $C$, we have
\[
  k\cdot \om^{(2)} = k\cdot \lm - (Ck)_{A}\cdot I_{A} + \dotsb
\]
on $\P\ell_{A}$. To prove that $k\cdot \om^{(2)}\nequiv 0$ on $\P\ell_{A}$, it is thus sufficient to show that
\begin{equation}
  \label{alternative}
    k\cdot \lm\neq 0\qquad\text{or}\qquad
  (Ck)_{A}\neq 0.
\end{equation}
We first prove a general statement to this fact. Recall that $C$ depends on the parameter $c\in\R$.

\begin{lem}
\label{kdv2-nondeg-1}
For each $k\in\Z^{\N}$ with $1\le \abs{k_{Z}}\le 2$, there exists at most one $c_{k}\in\R$ such that the alternative~\eqref{alternative} does not hold. This $c_{k}$ is a rational multiple of $\pi^{2}$. Moreover, within every compact subset $\R\setminus \Ss_{A}$ there are only finitely many such $c_{k}$.~\fish
\end{lem}

\begin{proof}
We have
\[
  (Ck)_{A} = C_{A}k_{A} + C_{AZ}k_{Z},
\]
where $C_{AZ} = (C_{ij})_{i\in A,\, j\in Z}$. The diagonal elements of $C_{A}$ are linear functions of $c$, namely $60c$, while all other coefficients of both matrices are integer multiples of $\pi^{2}$, namely $-80\pi^{2}ij$.
In particular, the vector $C_{AZ}k_{Z}$ has coefficients $-80\pi^{2}ip$, $i\in A$, where
\[
  p = k_{Z}\cdot \lm_{Z}^{\o},\qquad \lm_{Z}^{\o} = (j)_{j\in Z}.
\]
Clearly, $p$ does not vanish since $1 \le \abs{k_{Z}} \le 2$. Thus $(Ck)_{A}$ does not vanish if $k_{A} = 0$. On the other hand, if $k_{A}\neq 0$, then $(Ck)_{A}$ can vanish for at most one value of $c$, and this value must be a rational multiple of $\pi^{2}$.

To prove the remaining statements, suppose that $(Ck)_{A} = 0$, and that $c$ belongs to some compact set $F\subset\R\setminus\Ss_{A}$. Then $C_{A}$ is invertible,
\[
  k_{A} = -C_{A}^{-1}C_{AZ}k_{Z},
\]
and, since $F$ is compact, we can bound $C_{A}^{-1}$ uniformly for $c\in F$. Consequently, for any $c\in F$,
\begin{equation}
  \label{est-1}
  \abs{k_{A}} \le \abs{C_{A}^{-1}}\abs{C_{AZ}k_{Z}} \le K\abs{k_{Z}\cdot \lm_{Z}^{\o}},
\end{equation}
where here and below, $K$ stands for various constants bigger than $1$ that depend only on $A$ and the compact set $F$.

Now suppose also that
\[
  k\cdot \lm = k_{A}\cdot \lm_{A} + k_{Z}\cdot \lm_{Z} = 0.
\]
In view of $\lm_{n} = (2n\pi)^{5} + 10c(2n\pi)^{3} + 30 c^{2}(2n\pi)$, it is a routine estimate to show that for $1\le \abs{k_{Z}}\le 2$ one has
\[
  \abs{k_{Z}\cdot \lm_{Z}} \ge K^{-1}\abs{k_{Z}\cdot \lm_{Z}^{\o}}^{5} - K\abs{k_{Z}\cdot \lm_{Z}^{\o}}^{3}.
\]
It follows from Cauchy-Schwarz and~\eqref{est-1} that $K\abs{k_{Z}\cdot \lm_{Z}^{\o}}\abs{\lm_{A}} \ge \abs{k_{A}\cdot \lm_{A}}$ and hence in view of $k\cdot \lm = 0$ and the preceding estimate
\begin{equation}
  \label{est-2}
    K\abs{k_{Z}\cdot \lm_{Z}^{\o}}\abs{\lm_{A}} \ge \abs{k_{Z}\cdot \lm_{Z}}
  \ge K^{-1}\abs{k_{Z}\cdot \lm_{Z}^{\o}}^{5} - K\abs{k_{Z}\cdot \lm_{Z}^{\o}}^{3}.
\end{equation}
In particular, $\abs{k_{Z}\cdot \lm_{Z}^{\o}} \le K$ with a different constant. Combining this estimate with estimates~\eqref{est-1} and~\eqref{est-2}, we find
\[
  \abs{k_{A}} \le K,\qquad \abs{k_{Z}\cdot \lm_{Z}} \le K.
\]
Thus, for $c\in F$ there can be only finitely many $k\in\Z^{\N}$ with $1\le \abs{k_{Z}}\le 2$ for which alternative~\eqref{alternative} does not hold. Consequently, there can be at most finitely many exceptional values $c_{k}$ in $F$.\qed
\end{proof}

\begin{lem}
\label{kdv2-nondeg-2}
For every finite set $A\subset\N$, the KdV2 Hamiltonian at $c=0$ satisfies~\eqref{alternative} for any $k\neq 0$ in $\Z^{\N}$ with $\abs{k_{Z}}\le 2$ and is thus nondegenerate.~\fish
\end{lem}

\begin{proof}
Let $n=\abs{A}$. We first consider the case $n=1$ that is $A = \setd{i}$ for some $i\in\N$.
If $k\neq 0$ and $k_{Z} = 0$, then $k\cdot \lm = k_{i} (2i\pi)^{5} \neq 0$. On the other hand, if $k_{Z}\neq 0$, then $(Ck)_{i} = -80\pi^{2} i p$ where $p = \sum_{j\in Z}j k_{j} \neq 0$ since $1 \le \abs{k_{Z}}\le 2$. Thus~\eqref{alternative} holds for $n=1$.

Next consider the case where $n\ge 2$. If $k_{Z} = 0$, then $k_{A}\neq 0$. Since $\det C_{A}\neq 0$ by Lemma~\ref{kdv2-Q-nondeg}, we have $(Ck)_{A} = C_{A}k_{A} \neq 0$, hence~\eqref{alternative} holds. So it remains to consider the case $1\le \abs{k_{Z}}\le 2$.
In view of~\eqref{alternative} assume in addition that $(Ck)_{A} = 0$. It is to show that then $k\cdot \lm\neq 0$.
 At $c=0$, the coefficients of $C$ are, up to a common multiplicative factor, $(\dl_{ij}-1)ij$. Hence, the coefficients of $k$ satisfy
\[
  ik_{i} = \sum_{j\ge 1} j k_{j},\qquad i\in A.
\]
It follows that $ik_{i} = r$ is independent of $i$ for $i\in A$. Substituting these identities into the above sum, we obtain
\[
  r = \sum_{i\in A}r + \sum_{j\in Z} jk_{j} = nr + p,
\]
where $p = \sum_{j\in Z} j k_{j}$. Since $1 \le \abs{k_{Z}}\le 2$, it follows that $p\neq 0$ and without loss we may assume that $p > 0$, since otherwise we may choose $-k$ instead of $k$. As an immediate consequence,
\begin{equation}
  \label{eq-iki}
  ik_{i} = r = -\frac{p}{n-1},\qquad i\in A.
\end{equation}
In particular, all $k_{i}$ are distinct and strictly negative.

Now consider
\[
  k\cdot \lm = (2\pi)^{5}\p[\bigg]{\sum_{i\in A} i^{5}k_{i} + \sum_{j\in Z} j^{5}k_{j}}.
\]
Solving~\eqref{eq-iki} for $i$ and $k_{i}$ we have
\[
  -\sum_{i\in A} i^{5}k_{i} = \frac{p}{n-1}\sum_{i\in A} i^{4} = \p*{\frac{p}{n-1}}^{5}\sum_{i\in A} \frac{1}{k_{i}^{4}}.
\]
To show that $k\cdot \lm\neq 0$, it thus suffices to show that the two terms
\[
  I = \sum_{i\in A} \frac{1}{k_{i}^{4}} > 0
  \qquad
  \text{and}
  \qquad
  II = \p*{\frac{n-1}{p}}^{5} \sum_{j\in Z} j^{5}k_{j}
\]
are not equal. Using $1 \le \abs{k_{Z}} \le 2$, one easily checks that
\[
  \abs[\bigg]{\sum_{j\in Z} j^{5}k_{j}} \ge \frac{1}{2^{4}}\p[\bigg]{\sum_{j\in Z} jk_{j}}^{5} = \frac{p^{5}}{2^{4}}.
\]
Consequently, $II \ge \frac{(n-1)^{5}}{2^{4}}$. On the other hand, since all $k_{i}$ are distinct and have the same sign,
\[
  I = \sum_{i\in A} \frac{1}{k_{i}^{4}}
   \le \sum_{\nu=1}^{n}\frac{1}{\nu^{4}}
   \le \frac{4}{3}- \frac{1}{3n^{3}},
\]
and the right hand side is strictly less than $\frac{1}{2^{4}}(n-1)^{5}$ if $n\ge 3$.

For $n=2$, the above argument is still valid when $k_{Z}$ has only one nonzero component. In this case, $k_{Z} = le_{j_{0}}$ with $1\le l\le 2$ since by assumption $p = j_{0}l > 0$. Then,
\[
  \sum_{j\in Z} j^{5}k_{j} = \frac{(j_{0}l)^{5}}{l^{4}} = \frac{p^{5}}{l^{4}}
\]
and hence
\[
  II = II_{l} = \frac{1}{l^{4}},\qquad l = 1,2.
\]
However, since $n=2$ and the $k_{i}$ are distinct, one checks that neither $1 = II_{1}$ nor $1/2^{4} = II_{2}$ are possible values of $I$.

It thus remains to discuss the case $k_{Z} = e_{j_{1}} \pm e_{j_{2}}$ with $j_{1} > j_{2}$. Since the $\pm$ cases are treated similarly, we concentrate on $k_{Z} = e_{j_{1}} + e_{j_{2}}$ only.
If $\abs{k_{i}}\ge 3$ for $i\in A$, then $I \le \frac{1}{3^{4}} + \frac{1}{4^{4}} < \frac{1}{2^{4}} \le II$, hence we only need to consider the case where $\min_{i\in A} \abs{k_{i}} \le 2$. More precisely, with $A = \setd{i_{1},i_{2}}$, $1\le i_{2} < i_{1}$, there remain the two cases to be studied:
\begin{align*}
  \text{(i)} \quad& k_{i_{1}} = -1,\quad k_{i_{2}} \le -2,\qquad
  \text{(ii)}\quad  k_{i_{1}} = -2,\quad k_{i_{2}} \le -3.
\end{align*}
Suppose $I=II$, then in either case it follows that
\[
  i_{1}^{4} + i_{2}^{4} = p^{4}I = p^{4}II = \frac{j_{1}^{5} + j_{2}^{5}}{p}.
\]
Let $q = j_{1}-j_{2} \ge 1$ and substitute $2j_{1} = p+q$, $2j_{2} = p-q$ into the latter expression to get
\begin{equation}
  \label{i-j-eqn}
  2^{4}(i_{1}^{4} + i_{2}^{4}) = p^{4} + 10p^{2}q^{2} + 5q^{4}.
\end{equation}

(i): Let $\mu = -k_{2}$, then $p = i_{1} = \mu i_{2}$, and hence~\eqref{i-j-eqn} takes the form
\[
  2^{4}i_{2}^{4}
   = -15p^{4} + 10p^{2}q^{2} + 5q^{4}
   = 5(q^{4} + 2p^{2}q^{2} - 3p^{4}).
\]
Consequently, $5|i_{2}^{4}$ and hence $5|i_{2}$. Therefore, also $5|p$ and it follows that
\[
  5^{3}|q^{2}(q^{2}+2p^{2}).
\]
Since $5|p$ it follows that $5|q$. Let $p = 5\tilde p$, $q = 5\tilde q$, and $i_{2} = 5\tilde i_{2}$, then
\[
  2^{4}\tilde i_{2}^{4} = 5(\tilde q^{4} + 2\tilde p^{2}\tilde q^{2} - 3\tilde p^{4}).
\]
We are now in the same position as in the beginning with $\tilde p = \mu \tilde i_{2}$. Thus, we conclude $5|\tilde p$ and $5|\tilde q$. This, of course, can be repeated ad infinitum giving a contradiction. Therefore, $I\neq II$ in this case.

(ii): Let $\mu = -k_{2}$, then $p = 2i_{1} = \mu i_{2}$, and hence~\eqref{i-j-eqn} takes the form
\[
  2^{4}i_{2}^{4} = 5q^{2}(2p^{2}+q^{2}).
\]
Again, it follows that $5|i_{2}$ and hence $5|p$ so that also $5|q$.
This argument can be repeated ad infinitum which shows that $I\neq II$.\qed
\end{proof}

\section{Frequency flow in sequence spaces}

Suppose $1 \le p < \infty$, $\sg\in\R$, and let $X^{\sg,p}$ be any subset of $\ell^{\sg,p}$. In this appendix, we consider the flow generated by a sequence of frequency functions $\om_{n}\colon X^{\sg,p}\to \R$, $n\ge 1$. The corresponding flow in $\ell^{\sg,p}$ is denoted by $\ph^{t}(z) = (\ph_{n}^{t}(z))_{n\ge 1}$, where $z\in X^{\sg,p}$ denotes the initial value and
\begin{equation}
  \label{coord-function}
  \ph_{n}^{t}(z) \defl \e^{\ii \om_{n}(z) t}z_{n},\qquad n\ge 1.
\end{equation}
To simplify notation, we denote the constant part of $\om_{n}$ by $\om_{n}^{\o}$ and write
\[
  \om_{n}(z) = \om_{n}^{\o} + \om_{n}^{\star}.
\]
Further, $\om = (\om_{n})_{n\ge 1}$ denotes the frequency map.

\begin{thm}
\label{thm:freq-flow}
Let $1 \le p < \infty$, $\sg\in\R$, and $X^{\sg,p}\subset\ell^{\sg,p}$ be a subset invariant by the flow~\eqref{coord-function}.

\begin{equivenum}
\item
The map $\R\to X^{\sg,p}$, $t\mapsto \ph^{t}(z)$ defines a continuous curve in $X^{\sg,p}$ for any $z\in X^{\sg,p}$.

\item If each $\om_{n}^{\star}\colon X^{\sg,p}\to \R$ is continuous, then for any $T > 0$ the map
\begin{equation}
  \Sc\colon X^{\sg,p}\to C([-T,T],X^{\sg,p}),\quad z \mapsto (t\mapsto \ph^{t}(z)),
\end{equation}
is continuous and has the group property $\Sc(t+s,z) = \Sc(t,\Sc(s,z))$ for all $t,s\in\R$ and $z\in X^{\sg,p}$. In particular, for any $t\in\R$, $\ph^{t}\colon X^{\sg,p}\to X^{\sg,p}$  is a homeomorphism.

\item
If $\om^{\star}\colon X^{\sg,p}\to \ell^{\infty}$ is real analytic, then for any $T > 0$, the map $\Sc$ is real analytic. It means that for any $z\in X^{s,p}$ there exists a complex neighborhood $V$ of $z$ in $\ell_{\C}^{s,p}$ so that $\om^{\star}\colon V\to \ell_{\C}^{\infty}$ and $S\colon V\to C([-T,T],\ell_{\C}^{s,p})$ are analytic maps.

\item
If $\om^{\star}\colon X^{\sg,p}\to \ell^{\infty}$ is uniformly continuous on bounded subsets of $X^{s,p}$, then for any $T > 0$, the map  $\Sc$ is uniformly continuous on bounded subsets of $X^{s,p}$.~\fish

\end{equivenum}
\end{thm}

\begin{rem}
In our applications, the frequencies can be written in the form
\[
  \om_{n}(z) = \al_{n} + \bt_{n}(z) + \rho_{n}(z),
\]
where $\al_{n}$ is constant in $z$, $\bt_{n}(z)$ is a polynomial in $n$ whose coefficients are integrals of the equation, and $\rho_{n}(z)$ satisfies certain decay estimates. To invoke items (i) and (ii) of Theorem~\ref{thm:freq-flow}, we can choose $X^{\sg,p} = \ell^{\sg,p}$ and $\om_{n}^{\o} = \al_{n}$ as well as $\om_{n}^{\star} = \bt_{n}(z) + \rho_{n}(z)$ since no decay of $\om_{n}^{\star}$ is needed. To apply items (iii) and (iv), however, we have to restrict to certain invariant subspaces $X^{\sg,p}$ of $\ell^{\sg,p}$ on which the integrals involved in $\bt_{n}$ are fixed. Then the constant part of the frequencies is given by $\om_{n}^{\o} = \al_{n} + \bt_{n}$ and $\om_{n}^{\star}(z) = \rho_{n}(z)$.\map
\end{rem}

\begin{proof}

(i):
Fix any $z\in X^{\sg,p}$. All coordinate functions $\R\to \C$, $t\mapsto \ph_{n}^{t}(z)$ are continuous and $\abs{\ph_{n}^{t}(z)} = \abs{z_{n}} = \ell_{n}^{\sg,p}$ uniformly in $t\in\R$.
Consequently, $\R\to X^{\sg,p}$, $t\mapsto \ph^{t}(z)$ defines a continuous curve in $X^{\sg,p}$.


(ii):
Fix any $T > 0$. All coordinate functions $X^{\sg,p}\to C([-T,T],\C)$, $z\mapsto (t\mapsto \ph_{n}^{t}(z))$ are continuous. Since $\sup_{t\in [-T,T]} \abs{\ph_{n}^{t}(z)} = \abs{z_{n}} = \ell_{n}^{\sg,p}$, we conclude that $\Sc\colon X^{\sg,p}\to C([-T,T],X^{\sg,p})$ is continuous as well.
   The group property then follows from the representation~\eqref{coord-function} and the homeomorphism property is an immediate consequence.

(iii):
By assumption, each $z\in X^{\sg,p}$ admits a complex neighborhood $V$ so that $\om^{\star}\colon V\to \ell_{\C}^{\infty}$ is analytic and $\sup_{n\ge 1} \abs{\om_{n}^{\star}(w)} < \infty$ uniformly on $V$. Consequently, each coordinate function $V \to C^{0}([-T,T],\C)$, $w\mapsto \ph_{n}^{t}(w)$ is analytic and
\[
  \ph_{n}^{t}(w) =
  \e^{\ii \om_{n}^{\o}t}
  \e^{\ii \om_{n}^{\star}(w)t}
  n^{-\sg}\ell_{n}^{p}
  = n^{-\sg}\ell_{n}^{p},
\]
uniformly on $[-T,T]\times V$. Therefore, $\Sc\colon V\to C^{0}([-T,T],\ell_{\C}^{\sg,p})$ is analytic.

(iv):
Fix $R\ge 1$. For any $z,w \in B_{R}(0)\subset X^{\sg,p}$, we have
\begin{align*}
  &\n{\ph^{t}(z)-\ph^{t}(w)}_{\ell^{\sg,p}}
   \le \n{z-w}_{\ell^{\sg,p}}
  +
  R\sup_{n\ge 1}\abs{\e^{\ii(\om_{n}(z)-\om_{n}(w))t}-1}.
\end{align*}
Since $\om^{\star}\colon X^{\sg,p}\to \ell^{\infty}$ is uniformly continuous on bounded subsets,  for any $\ep > 0$ there exists $0 < \dl \le \ep$ so that for all $z,w\in B_{R}(0)$ with $\n{z-w}_{\ell^{\sg,p}} \le \dl$,
\[
  \sup_{n\ge 1}\abs{\om_{n}(z)-\om_{n}(w)}
   =   \sup_{n\ge 1}\abs{\om_{n}^{\star}(z)-\om_{n}^{\star}(w)}
   \le \ep.
\]
Since $\abs{\e^{\ii x}-1} = \abs{\int_{0}^{1}\ii x \e^{\ii xs}\,\ds} \le \abs{x}$ for all $x\in\R$, we conclude that for any $z,w\in B_{R}(0)$ with $\n{z-w}_{\ell^{\sg,p}} \le \dl$, and any $-T \le t \le T$,
\[
  \n{\ph^{t}(z)-\ph^{t}(w)}_{\ell^{\sg,p}}
   \le \dl + \ep \abs{t}R \le \ep(1 + TR),
\]
which proves that $\Sc$ is uniformly continuous on bounded subsets.\qed
\end{proof}

\end{appendix}


\begin{thebibliography}{31}
\providecommand{\natexlab}[1]{#1}
\providecommand{\url}[1]{#1}
\providecommand{\urlprefix}{}
\expandafter\ifx\csname urlstyle\endcsname\relax
  \providecommand{\doi}[1]{DOI~\discretionary{}{}{}#1}\else
  \providecommand{\doi}{DOI~\discretionary{}{}{}\begingroup
  \urlstyle{rm}\Url}\fi

\bibitem[{B{\"a}ttig et~al.(1997)B{\"a}ttig, Kappeler, \&
  Mityagin}]{Battig:1997ek}
B{\"a}ttig, D., Kappeler, T., Mityagin, B.: \emph{{On the Korteweg-de Vries
  equation: frequencies and initial value problem}}.
\newblock Pacific J. Math. \textbf{181}(1), 1--55, 1997.
\newblock \doi{10.2140/pjm.1997.181.1}.
\newblock \urlprefix\url{http://dx.doi.org/10.2140/pjm.1997.181.1}

\bibitem[{Bikbaev \& Kuksin(1993)}]{Bikbaev:1993jl}
Bikbaev, R.F., Kuksin, S.B.: \emph{{On the parametrization of finite-gap
  solutions by frequency and wavenumber vectors and a theorem of I.
  Krichever}}.
\newblock Lett. Math. Phys. \textbf{28}(2), 115--122, 1993.
\newblock \doi{10.1007/BF00750304}.
\newblock \urlprefix\url{http://dx.doi.org/10.1007/BF00750304}

\bibitem[{Bourgain(1995)}]{Bourgain:1995wk}
Bourgain, J.: \emph{{On the Cauchy problem for periodic KdV-type equations}}.
\newblock In: \emph{Proceedings of the Conference in Honor of Jean-Pierre
  Kahane (Orsay, 1993)}, pp. 17--86, 1995.
\newblock \urlprefix\url{http://www.ams.org/mathscinet-getitem?mr=MR1364878}

\bibitem[{Bourgain(1997)}]{Bourgain:1997gg}
Bourgain, J.: \emph{{Periodic Korteweg de Vries equation with measures as
  initial data}}.
\newblock Selecta Math. (N.S.) \textbf{3}(2), 115--159, 1997.
\newblock \doi{10.1007/s000290050008}.
\newblock \urlprefix\url{http://dx.doi.org/10.1007/s000290050008}

\bibitem[{Christ et~al.(2003)Christ, Colliander, \& Tao}]{Christ:2003tx}
Christ, M., Colliander, J., Tao, T.: \emph{{Asymptotics, frequency modulation,
  and low regularity ill-posedness for canonical defocusing equations}}.
\newblock Amer. J. Math. \textbf{125}(6), 1235--1293, 2003.
\newblock
  \urlprefix\url{http://muse.jhu.edu/journals/american_journal_of_mathematics/v125/125.6christ.pdf}

\bibitem[{Colliander et~al.(2003)Colliander, Keel, Staffilani, Takaoka, \&
  Tao}]{Colliander:2003fv}
Colliander, J., Keel, M., Staffilani, G., Takaoka, H., Tao, T.: \emph{{Sharp
  global well-posedness for KdV and modified KdV on R and T}}.
\newblock J. Amer. Math. Soc. \textbf{16}(3), 705--749 (electronic), 2003.
\newblock \doi{10.1090/S0894-0347-03-00421-1}.
\newblock \urlprefix\url{http://dx.doi.org/10.1090/S0894-0347-03-00421-1}

\bibitem[{Craig et~al.(2005)Craig, Guyenne, \& Kalisch}]{Craig:2005fv}
Craig, W., Guyenne, P., Kalisch, H.: \emph{{Hamiltonian long-wave expansions
  for free surfaces and interfaces}}.
\newblock Comm. Pure Appl. Math. \textbf{58}(12), 1587--1641, 2005.
\newblock \doi{10.1002/cpa.20098}.
\newblock \urlprefix\url{http://doi.wiley.com/10.1002/cpa.20098}

\bibitem[{Delpech(2009)}]{Delpech:2009ek}
Delpech, S.: \emph{{A short proof of Pitt's compactness theorem}}.
\newblock Proc. Amer. Math. Soc. \textbf{137}(4), 1371--1372, 2009.
\newblock \doi{10.1090/S0002-9939-08-09617-2}.
\newblock \urlprefix\url{http://dx.doi.org/10.1090/S0002-9939-08-09617-2}

\bibitem[{Djakov \& Mityagin(2009)}]{Djakov:2009fx}
Djakov, P., Mityagin, B.: \emph{{Spectral gaps of Schr\"odinger operators with
  periodic singular potentials}}.
\newblock Dyn. Partial Differ. Equ. \textbf{6}(2), 95--165, 2009.
\newblock \doi{10.4310/DPDE.2009.v6.n2.a1}.
\newblock \urlprefix\url{http://dx.doi.org/10.4310/DPDE.2009.v6.n2.a1}

\bibitem[{Erdo{\u{g}}an \& Tzirakis(2013)}]{Erdogan:2013ve}
Erdo{\u{g}}an, M.B., Tzirakis, N.: \emph{{Global smoothing for the periodic KdV
  evolution}}.
\newblock Int. Math. Res. Not. (20), 4589--4614, 2013.
\newblock \urlprefix\url{http://www.ams.org/mathscinet-getitem?mr=MR3118870}

\bibitem[{Gr{\'e}bert \& Kappeler(2014)}]{Grebert:2014iq}
Gr{\'e}bert, B., Kappeler, T.: \emph{{The defocusing NLS equation and its
  normal form}}.
\newblock European Mathematical Society (EMS), Z{\"u}rich, 2014.
\newblock \doi{10.4171/131}.
\newblock \urlprefix\url{http://dx.doi.org/10.4171/131}

\bibitem[{Gr{\"u}nrock(2010)}]{Grunrock:2010bg}
Gr{\"u}nrock, A.: \emph{{On the hierarchies of higher order mKdV and KdV
  equations}}.
\newblock Central European Journal of Mathematics \textbf{8}(3), 500--536,
  2010.
\newblock \doi{10.2478/s11533-010-0024-5}.
\newblock
  \urlprefix\url{http://www.springerlink.com/index/10.2478/s11533-010-0024-5}

\bibitem[{Guo et~al.(2013)Guo, Kwak, \& Kwon}]{Guo:2012th}
Guo, Z., Kwak, C., Kwon, S.: \emph{{Rough solutions of the fifth-order KdV
  equations}}.
\newblock J. Funct. Anal. \textbf{265}(11), 2791--2829, 2013.
\newblock \doi{10.1016/j.jfa.2013.08.010}

\bibitem[{Kappeler et~al.(2016)Kappeler, Maspero, Molnar, \&
  Topalov}]{Kappeler:CNzeErmy}
Kappeler, T., Maspero, A., Molnar, J.C., Topalov, P.: \emph{{On the Convexity
  of the KdV Hamiltonian}}.
\newblock Comm. Math. Phys. pp. 1--46, 2016.
\newblock \doi{10.1007/s00220-015-2563-x}.
\newblock \urlprefix\url{http://link.springer.com/10.1007/s00220-015-2563-x}

\bibitem[{Kappeler \& Mityagin(1999)}]{Kappeler:1999er}
Kappeler, T., Mityagin, B.: \emph{{Gap estimates of the spectrum of Hill's
  equation and action variables for KdV}}.
\newblock Trans. Amer. Math. Soc. \textbf{351}(2), 619--646, 1999.
\newblock \doi{10.1090/S0002-9947-99-02186-8}.
\newblock \urlprefix\url{http://dx.doi.org/10.1090/S0002-9947-99-02186-8}

\bibitem[{Kappeler \& M{\"o}hr(2001)}]{Kappeler:2001bi}
Kappeler, T., M{\"o}hr, C.: \emph{{Estimates for periodic and Dirichlet
  eigenvalues of the Schr\"odinger operator with singular potentials}}.
\newblock J. Funct. Anal. \textbf{186}(1), 62--91, 2001.
\newblock \doi{10.1006/jfan.2001.3779}.
\newblock \urlprefix\url{http://dx.doi.org/10.1006/jfan.2001.3779}

\bibitem[{Kappeler et~al.(2005)Kappeler, M{\"o}hr, \&
  Topalov}]{Kappeler:2005fb}
Kappeler, T., M{\"o}hr, C., Topalov, P.: \emph{{Birkhoff coordinates for KdV on
  phase spaces of distributions}}.
\newblock Selecta Math. (N.S.) \textbf{11}(1), 37--98, 2005.
\newblock \doi{10.1007/s00029-005-0009-6}.
\newblock \urlprefix\url{http://dx.doi.org/10.1007/s00029-005-0009-6}

\bibitem[{Kappeler \& P{\"o}schel(2003)}]{Kappeler:2003up}
Kappeler, T., P{\"o}schel, J.: \emph{{KdV {\&} KAM}}.
\newblock Springer, Berlin, 2003.
\newblock \urlprefix\url{http://www.ams.org/mathscinet-getitem?mr=MR1997070}

\bibitem[{Kappeler et~al.(2013)Kappeler, Schaad, \& Topalov}]{Kappeler:2013bt}
Kappeler, T., Schaad, B., Topalov, P.: \emph{{Qualitative features of periodic
  solutions of KdV}}.
\newblock Comm. Part. Diff. Eqs. \textbf{38}(9), 1626--1673, 2013.
\newblock \doi{10.1080/03605302.2013.814141}.
\newblock \urlprefix\url{http://dx.doi.org/10.1080/03605302.2013.814141}

\bibitem[{Kappeler et~al.(2008)Kappeler, Serier, \& Topalov}]{Kappeler:2008fl}
Kappeler, T., Serier, F., Topalov, P.: \emph{{On the symplectic phase space of
  KdV}}.
\newblock Proc. Amer. Math. Soc. \textbf{136}(5), 1691--1698, 2008.
\newblock \doi{10.1090/S0002-9939-07-09120-4}.
\newblock \urlprefix\url{http://dx.doi.org/10.1090/S0002-9939-07-09120-4}

\bibitem[{Kappeler \& Topalov(2003)}]{Kappeler:2003vh}
Kappeler, T., Topalov, P.: \emph{{Riccati representation for elements in
  $H^{-1}(\mathbb{T})$ and its applications}}.
\newblock Pliska Stud. Math. Bulgar. \textbf{15}, 171--188, 2003.
\newblock \urlprefix\url{http://www.ams.org/mathscinet-getitem?mr=MR2071691}

\bibitem[{Kappeler \& Topalov(2006)}]{Kappeler:2006fr}
Kappeler, T., Topalov, P.: \emph{{Global wellposedness of KdV in
  $H^{-1}(\mathbb{T},\mathbb{R})$}}.
\newblock Duke Math. J. \textbf{135}(2), 327--360, 2006.
\newblock \doi{10.1215/S0012-7094-06-13524-X}.
\newblock \urlprefix\url{http://dx.doi.org/10.1215/S0012-7094-06-13524-X}

\bibitem[{Kato(2012)}]{Kato:2012ih}
Kato, T.: \emph{{Well-posedness for the fifth order KdV equation}}.
\newblock Funkcialaj Ekvacioj \textbf{55}(1), 17--53, 2012.
\newblock \doi{10.1619/fesi.55.17}.
\newblock
  \urlprefix\url{http://joi.jlc.jst.go.jp/JST.JSTAGE/fesi/55.17?from=CrossRef}

\bibitem[{Kenig \& Pilod(2015)}]{Kenig:2015dx}
Kenig, C.E., Pilod, D.: \emph{{Well-posedness for the fifth-order KdV equation
  in the energy space}}.
\newblock Trans. Amer. Math. Soc. \textbf{367}(4), 2551--2612, 2015.
\newblock \doi{10.1090/S0002-9947-2014-05982-5}.
\newblock \urlprefix\url{http://dx.doi.org/10.1090/S0002-9947-2014-05982-5}

\bibitem[{Korotyaev(2003)}]{Korotyaev:2003gp}
Korotyaev, E.: \emph{{Characterization of the spectrum of Schr\"odinger
  operators with periodic distributions}}.
\newblock Int. Math. Res. Not. \textbf{2003}(37), 2019--2031, 2003.
\newblock \doi{10.1155/S1073792803209107}.
\newblock \urlprefix\url{http://dx.doi.org/10.1155/S1073792803209107}

\bibitem[{Korotyaev \& Kuksin(2015)}]{Korotyaev:2011tw}
Korotyaev, E., Kuksin, S.B.: \emph{{KdV Hamiltonian as a Function of Actions}}.
\newblock Journal of Dynamical and Control Systems pp. 1--22, 2015.
\newblock \doi{10.1007/s10883-015-9289-0}.
\newblock \urlprefix\url{http://link.springer.com/10.1007/s10883-015-9289-0}

\bibitem[{Molinet(2012)}]{Molinet:2012il}
Molinet, L.: \emph{{Sharp ill-posedness results for the KdV and mKdV equations
  on the torus}}.
\newblock Adv. Math. \textbf{230}(4-6), 1895--1930, 2012.
\newblock \doi{10.1016/j.aim.2012.03.026}.
\newblock \urlprefix\url{http://dx.doi.org/10.1016/j.aim.2012.03.026}

\bibitem[{Molnar(2016)}]{Molnar:2016hq}
Molnar, J.C.: \emph{{On two-sided estimates for the nonlinear Fourier transform
  of KdV}}.
\newblock Discrete Contin. Dyn. Syst. \textbf{36}(6), 3339--3356, 2016.
\newblock \doi{10.3934/dcds.2016.36.3339}.
\newblock
  \urlprefix\url{http://www.aimsciences.org/journals/displayArticlesnew.jsp?paperID=12126}

\bibitem[{Saut(1979)}]{Saut:1979bl}
Saut, J.C.: \emph{{Quelques g\'en\'eralisations de l'\'equation de Korteweg-de
  Vries. II}}.
\newblock J. Differential Equations \textbf{33}(3), 320--335, 1979.
\newblock \doi{10.1016/0022-0396(79)90068-8}.
\newblock \urlprefix\url{http://dx.doi.org/10.1016/0022-0396(79)90068-8}

\bibitem[{Savchuk \& Shkalikov(2003)}]{Savchuk:2003vl}
Savchuk, A.M., Shkalikov, A.A.: \emph{{Sturm-Liouville operators with
  distribution potentials}}.
\newblock Tr. Mosk. Math. Obs. \textbf{64}, 159--212, 2003.
\newblock \urlprefix\url{http://www.ams.org/mathscinet-getitem?mr=MR2030189}

\bibitem[{Zygmund(1957)}]{Zygmund:1957up}
Zygmund, A.: \emph{{On singular integrals}}.
\newblock Rend. Mat. e Appl. (5) \textbf{16}, 468--505, 1957.
\newblock \urlprefix\url{http://www.ams.org/mathscinet-getitem?mr=MR0096088}

\end{thebibliography}
\end{document}